\documentclass[12pt]{amsart}
\pdfoutput=1
\providecommand{\lang}{british}
\providecommand{\KOMAopt}{twoside=semi, DIV=current,headinclude}

\usepackage{newtxtext}
\usepackage{anyfontsize}
\usepackage{eurosym}
\usepackage[\lang]{babel}
\usepackage[dvipsnames,svgnames]{xcolor}
\definecolor{darkblue}{rgb}{0.0,0.0,0.6}
\usepackage[T1]{fontenc}
\usepackage[utf8]{inputenc}
\usepackage{lmodern}
\usepackage{csquotes}
\usepackage[babel=true]{microtype}
\usepackage[fontsize=12pt]{scrextend}
\usepackage[\KOMAopt]{typearea}
\usepackage{amsmath}
\usepackage{mathtools}
	\mathtoolsset{mathic}
	\renewcommand{\substack}{\crampedsubstack}
\usepackage{centernot}
\usepackage{amssymb}
\usepackage{amsfonts}
\usepackage{amsthm}
\usepackage{booktabs}
\usepackage{pdfpages}

\usepackage[varvw]{newtxmath}
\usepackage[cal=cm,bb=ams]{mathalpha}
\usepackage[only,llbracket,rrbracket]{stmaryrd}
\SetSymbolFont{stmry}{bold}{U}{stmry}{m}{n}
\usepackage{upgreek}
\usepackage{ellipsis}
\usepackage{combelow}

\makeatletter
\@namedef{subjclassname@2020}{\textup{2020} Mathematics Subject Classification}
\makeatother

\usepackage[pdflang=en-UK, colorlinks, urlcolor=darkblue, linkcolor=darkblue, citecolor=darkblue]{hyperref}
\usepackage{color}
\usepackage{tikz}
	\usetikzlibrary{calc}
	\usetikzlibrary{arrows,decorations.markings}
\usepackage{tikz-cd}

\usepackage{todonotes}

\usepackage{enumitem}
\setenumerate[1]{label=(\alph*)}

\usepackage{subcaption}

\newcommand{\Frac}{\operatorname{Frac}}
\newcommand{\ideal}[1]{\langle #1\rangle}

\newcommand{\op}{\mathrm{op}}

\DeclareMathOperator{\rank}{rank}

\newcommand{\Span}[1]{\langle#1\rangle}

\newcommand{\donothing}[1]{}

\newcommand{\dsum}{\mathbin{\oplus}}
	\newcommand{\bigdsum}{\bigoplus}
\DeclareMathOperator{\End}{End}
	\DeclareMathOperator{\stabEnd}{\underline{End}}
\DeclareMathOperator{\Ext}{Ext}
\DeclareMathOperator{\Hom}{Hom}

\newcommand{\Kgp}{\mathrm{K}}

\newcommand{\id}{\mathrm{id}}

\newcommand{\iso}{\cong}
\newcommand{\isoto}{\xrightarrow{\,\smash{\raisebox{-0.65ex}{\ensuremath{\scriptstyle\sim}}}\,}}

\newcommand{\map}[1]{\xrightarrow{#1}}

\newcommand{\tensor}{\mathbin{\otimes}}
	
\newcommand{\Wedge}[1]{{\textstyle{\bigwedge}}^{\! #1}}

\DeclareMathOperator{\Proj}{Proj}

\newcommand{\GL}{\mathrm{GL}}

\DeclareMathOperator{\gldim}{gl.dim}

\newcommand{\bdd}{\mathrm{b}}
\newcommand{\CC}{\mathbb{C}}

\newcommand{\e}{\mathrm{e}}

\newcommand{\KK}{\mathbb{K}}

\newcommand{\ZZ}{\mathbb{Z}}

\newcommand{\defeq}{\coloneqq}
	\newcommand{\eqdef}{\eqqcolon}

\DeclareMathOperator{\dimvec}{\underline{dim}}

\newcommand{\idemp}[1]{e_{#1}}
\newcommand{\src}[1]{s#1}

\newcommand{\tgt}[1]{t#1}

\newcommand{\compl}{\mathrm{c}}

	\newcommand{\disunion}{\sqcup}

\newcommand{\clustalg}[1]{\mathscr{A}_{#1}}
\newcommand{\cpa}[2]{#1\langle\!\langle #2\rangle\!\rangle}
\newcommand{\Endalg}[2]{\End_{#1}(#2)^{\op}}
	\newcommand{\stabEndalg}[2]{\stabEnd_{#1}(#2)^{\op}}

\newcommand{\powser}[2]{#1\llbracket#2\rrbracket}
	\newcommand{\powrat}[2]{#1(\!(#2)\!)}

	\DeclareMathOperator{\add}{add}

\DeclareMathOperator{\CM}{CM}
\newcommand{\dcat}[1][]{\mathcal{D}^{#1}}
	\newcommand{\bdcat}{\mathcal{D}^{\bdd}}

\DeclareMathOperator{\GP}{GP}
	
	\DeclareMathOperator{\GI}{GI}
	
	\DeclareMathOperator{\gproj}{gproj}
	\DeclareMathOperator{\stabgproj}{\underline{gproj}}
	\DeclareMathOperator{\ginj}{ginj}
	\DeclareMathOperator{\stabginj}{\underline{ginj}}
\newcommand{\hcat}[1][]{\mathcal{K}^{#1}}

	\DeclareMathOperator{\fgmod}{mod}

\DeclareMathOperator{\per}{per}
\DeclareMathOperator{\projcat}{proj}

\newcommand{\singcat}{\mathcal{D}_{\mathrm{sg}}}

\newcommand{\stab}[1]{\underline{#1}}

\newcommand{\Grass}[2]{\mathrm{Gr}_{#1}^{#2}}
	\newcommand{\Grasscone}[2]{\widehat{\mathrm{Gr}}{}_{#1}^{#2}}
	\newcommand{\qGrass}[2]{\mathrm{Gr}_{#1}(#2)}
\newcommand{\Plueck}[1]{\Delta_{#1}}
\newcommand{\PP}{\mathbb{P}}

\newcommand{\blank}{\mathord{\char123}}
\newcommand{\cat}{\mathcal}

\renewcommand{\epsilon}{\varepsilon}
\renewcommand{\geq}{\geqslant}
\renewcommand{\leq}{\leqslant}
\newcommand{\type}{\mathsf}
\renewcommand{\subset}{\subseteq}

\newcommand{\bdry}{\partial}
\newcommand{\close}[1]{\overline{#1}}

  \theoremstyle{plain}
\newtheorem{thm}{Theorem}[section]
\newtheorem{thm*}{Theorem}
\newtheorem{lem}[thm]{Lemma}
\newtheorem{cor}[thm]{Corollary}
\newtheorem{cor*}{Corollary}
\newtheorem{prop}[thm]{Proposition}
\newtheorem{conj}[thm]{Conjecture}

\theoremstyle{definition}
\newtheorem{defn}[thm]{Definition}
\newtheorem{eg}[thm]{Example}

\theoremstyle{remark}
\newtheorem{rem}[thm]{Remark}

\numberwithin{equation}{section} \usepackage[backend=biber,citestyle=numeric-comp,bibstyle=numeric,maxbibnames=99,giveninits=true,doi=false,isbn=false,url=false,eprint=true]{biblatex}

\renewbibmacro{in:}{\ifentrytype{article}{}{\printtext{\bibstring{in}\intitlepunct}}}

\DeclareFieldFormat[article,inbook,incollection,inproceedings,patent]{title}{#1}
\DeclareFieldFormat[thesis,unpublished]{title}{{\it #1}}

\DeclareFieldFormat[online]{date}{Preprint (#1)}

\DeclareFieldFormat[article]{volume}{\mkbibbold{#1}}
\renewbibmacro*{volume+number+eid}{\printfield{volume}\setunit{\addcomma\space}\printfield{eid}}

\DeclareFieldFormat[book,incollection]{number}{Vol.~#1}
\renewbibmacro*{series+number}{\iffieldundef{series}{}
    {\printfield{series}\iffieldundef{number}{}{\setunit*{\addcomma\addspace}\printfield{number}\newunit}}}

\renewbibmacro*{publisher+location+date}{\printlist{publisher}\setunit*{\addcomma\space}\usebibmacro{date}\newunit}

\DeclareFieldFormat{eprint:arxiv}{\ifentrytype{online}
  {\ifhyperref
    {\href{http://arxiv.org/abs/#1}{\nolinkurl{arXiv:#1}}}
    {\nolinkurl{arXiv:#1}}
   \iffieldundef{eprintclass}
    {}
    {{\texttt{[\thefield{eprintclass}]}}}}
  {\iffieldundef{eprintclass}
    {\mkbibparens{\ifhyperref
    {\href{http://arxiv.org/abs/#1}{\nolinkurl{arXiv:#1}}}
    {\nolinkurl{arXiv:#1}}}}
    {\mkbibparens{\ifhyperref
    {\href{http://arxiv.org/abs/#1}{\nolinkurl{arXiv:#1}}}
    {\nolinkurl{arXiv:#1}}
    {\texttt{[\thefield{eprintclass}]}}}}}}
\title[Quasi-coincidence of cluster structures on positroid varieties]{Quasi-coincidence of cluster structures\\on positroid varieties}
\author{Matthew Pressland}
\address[Matthew Pressland]{School of Mathematics \& Statistics\\University of Glasgow\\University Place\\Glasgow G20 8LR\\United Kingdom}
\curraddr{Laboratoire de Mathématiques Nicolas Oresme\\Université de Caen Normandie\\Boulevard Maréchal Juin\\14032 Caen Cedex 5\\France}
\email{\href{mailto:matthew.pressland@unicaen.fr}{matthew.pressland@unicaen.fr}}
\urladdr{\url{https://mdpressland.github.io}}
\dedicatory{Dedicated to Bernhard Keller on the occasion of his 60th birthday}

\newcommand{\posit}{\mathcal{P}}
\newcommand{\posvar}{\Pi}
	\newcommand{\posvarcone}{\widehat{\posvar}}
\newcommand{\openposvar}{\Pi^\circ}
	\newcommand{\openposvarcone}{\posvarcone^\circ}

\newcommand{\clust}[1]{\mathscr{#1}}
\newcommand{\froz}{\clust{F}}

\newcommand{\cvars}{\mathcal{V}}
\newcommand{\cmons}{\mathcal{W}}

\newcommand{\yhat}{\hat{y}}

\newcommand{\srcneck}{\mathcal{I}^{\smash{\mathrm{src}}}}
\newcommand{\tgtneck}{\mathcal{I}^{\smash{\mathrm{tgt}}}}

\newcommand{\MSsrc}[1]{\mathfrak{m}_{#1}^{\smash{\mathrm{src}}}}
\newcommand{\MStgt}[1]{\mathfrak{m}_{#1}^{\smash{\mathrm{tgt}}}}

\newcommand{\Segprod}{\mathbin{\overline{\otimes}}}

\newcommand{\ind}{\operatorname{ind}}
	
\newcommand{\coind}{\operatorname{coind}}

\renewcommand{\src}{\smash{\mathrm{src}}}
\renewcommand{\tgt}{\smash{\mathrm{tgt}}}

\newcommand{\srclab}[1]{I_{#1}^{\smash{\mathrm{src}}}}
\newcommand{\tgtlab}[1]{I_{#1}^{\smash{\mathrm{tgt}}}}

\newcommand{\Tsrc}{T^{\mathrm{src}}}
\newcommand{\Ttgt}{T^{\mathrm{tgt}}}

\newcommand{\GLsrc}{\eta^{\smash{\mathrm{src}}}}
\newcommand{\GLtgt}{\eta^{\smash{\mathrm{tgt}}}}

\newcommand{\CCsrc}{\Phi^{\mathrm{src}}}
\newcommand{\CCtgt}{\Phi^{\mathrm{tgt}}}

\newcommand{\bOmega}{\boldsymbol{\Omega}}
\newcommand{\bSigma}{\boldsymbol{\Sigma}}
\newcommand{\bsigma}{\boldsymbol{\upsigma}}

\newcommand{\Zdual}[1]{\smash{{#1}}^\vee}

\DeclareMathOperator{\degvec}{\underline{deg}}

\newcommand{\black}{+}
\newcommand{\white}{-}
\newcommand{\wt}[1][]{\operatorname{wt}^{#1}}

\newcommand{\pos}{\mathrm{pos}}
\renewcommand{\neg}{\mathrm{neg}}

\newcommand{\varPlueck}[1]{\Delta(#1)}

\newcommand{\ltwist}{\text{\reflectbox{\(\vec{\text{\reflectbox{\(\tau\)}}}\)}}}
\newcommand{\rtwist}{\vec{\tau}}

\usetikzlibrary{calc}
\usetikzlibrary{arrows,decorations.markings}
\newcommand{\strandcolor}{red}
\newcommand{\graphcolor}{black}
\newcommand{\quivcolor}{DarkGreen}
\newcommand{\bdrycolor}{gray}
\newcommand{\matchcolor}{magenta}
\newcommand{\frozcolor}{MediumBlue}

\tikzset{strand/.style={\strandcolor}, boundary/.style={thick, \bdrycolor},  bipedge/.style={\graphcolor, thick},
 quivarrow/.style={\quivcolor, -latex, thick},
 frozarrow/.style={\frozcolor, -latex, thick},
 matcharrow/.style={\matchcolor, -latex, very thick} }

\newcommand{\strarrow}{\arrow{angle 60}}
\pgfmathsetmacro{\bstart}{130} \pgfmathsetmacro{\seventh}{360/7}
\pgfmathsetmacro{\ninth}{360/9}
\newcommand{\dotrad}{1pt}  

\begin{document}

\begin{abstract}
By work of a number of authors, beginning with Scott and culminating with Galashin and Lam, the coordinate rings of positroid varieties in the Grassmannian carry cluster algebra structures.
In fact, they typically carry many such structures, the two best understood being the source-labelled and target-labelled structures, referring to how the initial cluster is computed from a Postnikov diagram or plabic graph.
In this article, we show that these two cluster algebra structures quasi-coincide, meaning in particular that a cluster variable in one structure may be expressed in the other structure as the product of a cluster variable and a Laurent monomial in the frozen variables.
This resolves a conjecture attributed to Muller and Speyer from 2017.
The proof depends critically on categorification: of the relevant cluster algebra structures by the author, of perfect matchings and twists by the author with Çanakçı and King, and of quasi-equivalences of cluster algebras by Fraser--Keller.
By similar techniques, we also show that Muller--Speyer's left twist map is a quasi-cluster equivalence from the target-labelled structure to the source-labelled structure.
\end{abstract}

\maketitle

\section{Introduction}

Positroid varieties have their origin in influential work of Postnikov \cite{Postnikov-PosGrass} on the totally positive Grassmannian, in which they appear as cells in a stratification of the Grassmannian that is particularly well-adapted to the study of positivity.
In the same paper, Postnikov introduced the combinatorics of alternating strand diagrams---now often called Postnikov diagrams---and plabic graphs, which were subsequently used by Scott \cite{Scott-Grass} to equip the coordinate ring of the Grassmannian with the structure of a cluster algebra.

Cluster algebra structures are useful when studying totally positive spaces since they provide a set of linearly independent functions (the cluster monomials) with non-negative structure constants, and this set can often be extended to a basis with the same property, albeit in several different ways.
More specifically, a cluster algebra comes with a preferred set of generators called cluster variables, together with a notion of compatibility between these variables, and a cluster monomial is, by definition, a product of compatible cluster variables.
Some cluster variables are frozen, and compatible with all others.

In Scott's cluster algebra structure on the Grassmannian \(\Grass{k}{n}\) of \(k\)-dimensional subspaces of \(\CC^n\), all Plücker coordinates are cluster variables, and two Plücker coordinates are compatible if and only if their labels (which are \(k\)-element subsets of \(\{1,\dotsc,n\}\)) are non-crossing.
This combinatorial condition appeared in earlier work by Leclerc and Zelevinsky \cite{LecZel}, who showed that two Plücker coordinates quasi-commute (as elements of the quantum matrix algebra) if and only if they have non-crossing labels.
Whenever \(2<k<n-2\), there are more cluster variables than just the Plücker coordinates, and indeed in all but finitely many such cases there are infinitely many cluster variables.

It was long expected (see for example \cite[Conj.~3.4]{MulSpe-LocAcycl}) that Scott's results should generalise from the Grassmannian to arbitrary (open) positroid varieties, a conjecture proved by Serhiyenko, Sherman-Bennett and Williams \cite{SSBW} for Schubert cells and subsequently by Galashin and Lam \cite{GalLam} in full generality.
Both results depend heavily on work of Leclerc on Richardson varieties \cite{Leclerc-Strata}; the cluster structures are obtained from Leclerc's via homogenisation.
In all cases, one chooses a Postnikov diagram corresponding to the positroid in question, and associates to it a quiver with Plücker coordinates as vertices.
This turns out to be the initial seed for the cluster algebra structure.

While the process for constructing the quiver is always the same, there are two commonly used ways of assigning Plücker coordinates to its vertices.
These correspond to the two natural bijections between the strands, which collectively describe a permutation of \(\{1,\dotsc,n\}\), and these \(n\) labels: one can either associate a strand to the label at its source, or to the label at its target.
For the full Grassmannian, the two labellings simply produce two different seeds in the same cluster algebra structure, but for more general positroids this is no longer the case.
This means each positroid is in fact equipped with two cluster algebra structures, which have the same cluster combinatorics (indeed, they are abstractly isomorphic) but are distinct in the sense that different sets of functions, and even different subsets of the Plücker coordinates, are cluster variables.

However, a further expectation, first remarked on by Muller and Speyer \cite[Rem.~4.7]{MulSpe-Twist} and later described as a conjecture by Fraser and Sherman-Bennett \cite[Conj.~1.1]{FSB}, was that these two cluster structures should in fact have the same set of cluster monomials. More precisely, the conjecture is that the cluster structures \emph{quasi-coincide}; this is a special case of a more general notion of quasi-equivalence, for cluster structures on possibly different rings, due to Fraser \cite{Fraser-Quasi} (see Definition~\ref{d:qcm}). In particular, it means that each non-frozen cluster variable \(x\) in one structure is equal to a product \(x'p\), in which \(x'\) is a cluster variable, and \(p\) a Laurent monomial in frozen variables, in the other cluster structure. Moreover, the resulting assignment \(x\mapsto x'\) is a bijection between the two sets of non-frozen cluster variables, respecting compatibility.

In this paper, we prove this conjecture. Perhaps surprisingly, our proof relies primarily on representation theory and homological algebra.
For connected positroids, the two cluster algebra structures have been categorified by the author \cite{Pressland-Postnikov}, and the main algebraic fact underpinning the proof is that the two categorifications are derived equivalent in an extremely natural way.
This gives us access to a homological method for detecting quasi-equivalences of cluster algebras, developed recently by Fraser and Keller \cite[Appendix]{KelWu}.
This categorical approach resolves the quasi-coincidence conjecture for connected positroids, and we show on the geometric side that this implies the general statement.

Fraser and Sherman-Bennett \cite[Cor.~6.8]{FSB} proved the quasi-coincidence conjecture for positroids satisfying a combinatorial condition which they call toggle-connectedness, while also pointing out \cite[Ex.~6.6]{FSB} that not all positroids (or even all connected positroids) have this property. Our argument is also rather different: Fraser and Sherman-Bennett's works by expressing the identity map as a composition of quasi-cluster equivalences (coming from the toggles referred to in the term `toggle-connectedness'), whereas we show directly that the identity is such an equivalence.

To make the paper as self-contained as possible, we provide more detailed background on positroid varieties and their cluster structures in Section~\ref{s:positroids} and on quasi-equivalences of cluster algebras in Section~\ref{s:quasi-equivalence}.
In Section~\ref{s:connected}, we show that the quasi-coincidence conjecture may be reduced to the case of connected positroids, for which the technology of categorification is available, and in Section~\ref{s:categorification} we describe this technology.
The proof of the quasi-coincidence conjecture is then given in Section~\ref{s:proof}.
In Section~\ref{s:twist}, we demonstrate that similar techniques can be used to prove a related fact, namely that Muller--Speyer's left twist map \cite{MulSpe-Twist} is a quasi-cluster equivalence. Shortly after this article first appeared as a preprint, Casals, Le, Sherman-Bennett and Weng \cite{CLSBW} gave an independent proof of this fact, using a connection to symplectic geometry via Legendrian weaves, rather than categorification. In particular, they show that the twist coincides with the Donaldson--Thomas transformation, which is known independently to be a quasi-cluster equivalence.

\section{Positroids, positroid varieties and their cluster structures}
\label{s:positroids}

\subsection{Postnikov diagrams, plabic graphs and perfect matchings}

In this section, we recall the combinatorics underlying open positroid varieties, and the cluster structures on their coordinate rings.
This will also allow us to set our conventions. In some cases, we will replace the original definitions by (non-trivially) equivalent statements, in order to make the exposition more concise.

\begin{defn}[\cite{Postnikov-PosGrass}]
A \emph{Postnikov diagram}, drawn in an oriented disc with \(n\) marked points on its boundary, consists of a set of \(n\) smooth curves, called \emph{strands}, such that one strand starts and one strand ends at each marked point.
These strands must satisfy the following rules.
\begin{enumerate}[label=(P\arabic*)]
\item\label{d:P1} The strands cross finitely many times, and each crossing is transverse and involves only two strands.
\item\label{d:P2} Moving along each strand, the signs of its crossings with other strands alternate.
\item\label{d:P3} No strand crosses itself.
\item\label{d:P4} If two strands cross twice, they must be oriented in opposite directions between the crossings.
Said differently, if we follow a pair of strands forwards (or backwards) from a crossing, they do not meet again.
\end{enumerate}
In all conditions except \ref{d:P3}, the marked points on the boundary are considered to be crossings by extending the strands out of the disc in the natural way.
We call the diagram \emph{connected} if the union of its strands is connected.
\end{defn}

For concreteness, we usually identify the marked points on the disc with the set \(\ZZ_n=\{1,\dotsc,n\}\), as in Figure~\ref{f:plabic}.
The notation reflects the fact that we will mostly treat \(\ZZ_n\) as a cyclically ordered set, in the clockwise order on the boundary of the disc, but occasionally it will be necessary to break this cyclic order to the usual linear order \(1<\dotsb<n\).

\begin{defn}
A Postnikov diagram \(D\) determines a permutation \(\pi_D\) of the marked points on the disc, via the rule that the strand starting at marked point \(i\) ends at marked point \(\pi_D(i)\).
\end{defn}

The data in a Postnikov diagram \(D\) may be encoded in a number of other equivalent ways, and we will use several of them here.
The first equivalent formulation is as a \emph{plabic graph} or \emph{dimer model}.
To describe this, we note that \(D\) cuts the disc into simply-connected regions, each of which has a boundary consisting of strand segments (and possibly some of the boundary of the disc).
We call the region \emph{oriented} if these strand segments are oriented consistently (either clockwise or anticlockwise), and \emph{alternating} otherwise, since in this case the orientations of the strand segments alternate around the boundary of the region because of \ref{d:P2}.

The corresponding plabic graph \(\Gamma_D\) has a black node for each anticlockwise oriented region, a white node for each clockwise region, and an edge for each crossing (connecting the nodes for the two oriented regions incident with this crossing).
An oriented region may meet the boundary of the disc at one or more of the marked points, and we complete the plabic graph by drawing half-edges from the node of this region to these marked points.
In this way, each marked point becomes incident with a unique half-edge, and \(\Gamma_D\) also cuts the disc into simply-connected regions, which are in natural bijection with the alternating regions of the Postnikov diagram.
An example of a Postnikov diagram and its dual plabic graph is shown in Figure~\ref{f:plabic}.

\begin{figure}
\begin{tikzpicture}[scale=3,baseline=(bb.base),yscale=-1]

\path (0,0) node (bb) {};

\foreach \n/\m/\a in {1/4/0, 2/3/0, 3/2/5, 4/1/10, 5/7/0, 6/6/-3, 7/5/0}
{ \coordinate (b\n) at (\bstart-\seventh*\n+\a:1.0);
  \draw (\bstart-\seventh*\n+\a:1.1) node {\(\m\)}; }

\foreach \n/\m in {8/1, 9/2, 10/3, 11/4, 14/5, 15/6, 16/7}
  {\coordinate (b\n) at ($0.65*(b\m)$);}

\coordinate (b13) at ($(b15) - (b16) + (b8)$);
\coordinate (b12) at ($(b14) - (b15) + (b13)$);

\foreach \n/\x/\y in {13/-0.03/-0.03, 12/-0.22/0.0, 14/-0.07/-0.03, 11/0.05/0.02, 16/-0.02/0.02}
  {\coordinate (b\n)  at ($(b\n) + (\x,\y)$); } 

\foreach \h/\t in {1/8, 2/9, 3/10, 4/11, 5/14, 6/15, 7/16, 
 8/9, 9/10, 10/11,11/12, 12/13, 13/8, 14/15, 15/16, 12/14, 13/15, 8/16}
{ \draw [bipedge] (b\h)--(b\t); }

\foreach \n in {8,10,12,15} 
  {\draw [\graphcolor] (b\n) circle(\dotrad) [fill=white];} \foreach \n in {9,11,13, 14,16}  
  {\draw [\graphcolor] (b\n) circle(\dotrad) [fill=\graphcolor];} 

\foreach \e/\f/\t in {2/9/0.5, 4/11/0.5, 5/14/0.5, 7/16/0.5, 
 8/9/0.5, 9/10/0.5, 10/11/0.5,11/12/0.5, 12/13/0.45, 8/13/0.6, 
 14/15/0.5, 15/16/0.6, 12/14/0.45, 13/15/0.4, 8/16/0.6}
{\coordinate (a\e-\f) at ($(b\e) ! \t ! (b\f)$); }

\draw [strand] plot[smooth]
coordinates {(b1) (a8-16) (a15-16) (b6)}
[postaction=decorate, decoration={markings,
 mark= at position 0.2 with \strarrow,
 mark= at position 0.5 with \strarrow, 
 mark= at position 0.8 with \strarrow }];
 
\draw [strand] plot[smooth]
coordinates {(b6) (a14-15) (a12-14)(a11-12) (a10-11) (b3)}
[postaction=decorate, decoration={markings,
 mark= at position 0.15 with \strarrow, mark= at position 0.35 with \strarrow,
 mark= at position 0.53 with \strarrow, mark= at position 0.7 with \strarrow,
 mark= at position 0.87 with \strarrow }];
 
\draw [strand] plot[smooth]
coordinates {(b3) (a9-10) (a8-9) (b1)}
[postaction=decorate, decoration={markings,
 mark= at position 0.2 with \strarrow,
 mark= at position 0.5 with \strarrow, 
 mark= at position 0.8 with \strarrow }];

\draw [strand] plot[smooth]
coordinates {(b2) (a9-10) (a10-11) (b4)}
 [postaction=decorate, decoration={markings,
 mark= at position 0.2 with \strarrow,
 mark= at position 0.5 with \strarrow, 
 mark= at position 0.8 with \strarrow }];

\draw [strand] plot[smooth]
coordinates {(b4) (a11-12) (a12-13) (a13-15) (a15-16) (b7)}
[postaction=decorate, decoration={markings,
 mark= at position 0.15 with \strarrow, mark= at position 0.35 with \strarrow,
 mark= at position 0.55 with \strarrow, mark= at position 0.7 with \strarrow,
 mark= at position 0.87 with \strarrow }];

\draw [strand] plot[smooth]
coordinates {(b7) (a8-16) (a8-13) (a12-13) (a12-14) (b5)}
[postaction=decorate, decoration={markings,
 mark= at position 0.15 with \strarrow, mark= at position 0.315 with \strarrow,
 mark= at position 0.5 with \strarrow, mark= at position 0.7 with \strarrow,
 mark= at position 0.88 with \strarrow }];

\draw [strand] plot[smooth]
coordinates {(b5) (a14-15) (a13-15) (a8-13) (a8-9) (b2)}
[postaction=decorate, decoration={markings,
 mark= at position 0.13 with \strarrow, mark= at position 0.33 with \strarrow,
 mark= at position 0.5 with \strarrow, mark= at position 0.7 with \strarrow,
 mark= at position 0.88 with \strarrow }];
 
 \draw [boundary] (0,0) circle(1.0);
\end{tikzpicture}
\caption{A Postnikov diagram and its corresponding plabic graph.}
\label{f:plabic}
\end{figure}
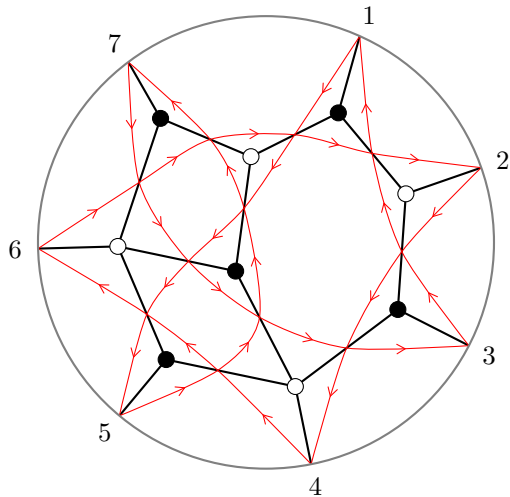

As can be deduced from the figure, it is possible to reverse this procedure to produce a collection of strands, also called \emph{zig-zag paths} in literature on dimer models, from any plabic graph.
These  will automatically satisfy conditions \ref{d:P1} and \ref{d:P2}, but \ref{d:P3} and \ref{d:P4} are additional (consistency) conditions on the plabic graph.
For this reason, we use the strand diagram as our primary definition.
Connectedness of \(D\) is equivalent to that of \(\Gamma_D\).

For the next definition, we use the fact that a Postnikov diagram is drawn in an oriented disc to give each strand a left-hand and right-hand side.
The left-hand side of a strand is bounded by the strand together with the anticlockwise segment of the boundary of the disc between the source and target of the strand, whereas the right-hand side uses the clockwise segment of the boundary instead.

\begin{defn}
Let \(D\) be a Postnikov diagram, and let \(j\) be an alternating region of \(D\).
The \emph{source label} \(\srclab{j}\) of \(j\) is the set of marked points \(i\in\ZZ_n\) such that the strand with source \(i\) has region \(j\) on its left-hand side.
The \emph{target label} \(\tgtlab{j}\) is the set of marked points \(i\in\ZZ_n\) such that the strand with target \(i\) has \(j\) on its left-hand side; that is, \(\tgtlab{j}=\pi_D(\srclab{j})\).
The \emph{source necklace} \(\srcneck_D\) of \(D\) is the set of source labels \(\srclab{j}\) for \(j\) an alternating region incident with the boundary of the disc, while the \emph{target necklace} \(\tgtneck_D\) is the set of target labels \(\tgtlab{j}\) for these regions.
See Figure~\ref{f:labels} for an example.

A \(k\)-element subset of \(\ZZ_n\) will be referred to simply as a \emph{\(k\)-subset}.
By comparing the labels of oriented regions either side of a crossing, one can see that all source and target labels for a fixed diagram \(D\) are \(k\)-subsets for some \(k\).
We say \(D\) has \emph{type} \((k,n)\) if \(k\) is this constant cardinality and \(n\) is the number of strands.
\end{defn}

\begin{figure}
\begin{minipage}{0.95\textwidth}
\begin{tikzpicture}[scale=3,baseline=(bb.base),yscale=-1]
\path (0,0) node (bb) {};

\foreach \n/\m/\a in {1/4/0, 2/3/0, 3/2/5, 4/1/10, 5/7/0, 6/6/-3, 7/5/0}
{ \coordinate (b\n) at (\bstart-\seventh*\n+\a:1.0);
  \draw (\bstart-\seventh*\n+\a:1.1) node {\(\m\)}; }

\foreach \n/\m in {8/1, 9/2, 10/3, 11/4, 14/5, 15/6, 16/7}
  {\coordinate (b\n) at ($0.65*(b\m)$);}

\coordinate (b13) at ($(b15) - (b16) + (b8)$);
\coordinate (b12) at ($(b14) - (b15) + (b13)$);

\foreach \n/\x/\y in {13/-0.03/-0.03, 12/-0.22/0.0, 14/-0.07/-0.03, 11/0.05/0.02, 16/-0.02/0.02}
  {\coordinate (b\n)  at ($(b\n) + (\x,\y)$); } 

\foreach \h/\t in {1/8, 2/9, 3/10, 4/11, 5/14, 6/15, 7/16, 
 8/9, 9/10, 10/11,11/12, 12/13, 13/8, 14/15, 15/16, 12/14, 13/15, 8/16}
{ \draw [bipedge] (b\h)--(b\t); }

\foreach \n in {8,10,12,15} 
  {\draw [\graphcolor] (b\n) circle(\dotrad) [fill=white];} \foreach \n in {9,11,13, 14,16}  
  {\draw [\graphcolor] (b\n) circle(\dotrad) [fill=\graphcolor];} 

\foreach \n/\m/\l/\a in {1/124/134/0, 2/234/123/-1, 3/345/127/1, 4/456/167/10, 5/256/367/5, 6/267/356/0, 7/127/345/0}
{ \draw [\frozcolor] (\bstart+\seventh/2-\seventh*\n+\a:0.9) node (q\m) {\scriptsize \(\mathbf{\l}\)}; }

\foreach \m/\l/\a/\r in {247/135/\bstart/0.42, 245/137/10/0.25, 257/357/210/0.36}
{ \draw [\quivcolor] (\a:\r) node (q\m) {\scriptsize \(\mathbf{\l}\)}; }

\foreach \e/\f/\t in {2/9/0.5, 4/11/0.5, 5/14/0.5, 7/16/0.5, 
 8/9/0.5, 9/10/0.5, 10/11/0.5,11/12/0.5, 12/13/0.45, 8/13/0.6, 
 14/15/0.5, 15/16/0.6, 12/14/0.45, 13/15/0.4, 8/16/0.6}
{\coordinate (a\e-\f) at ($(b\e) ! \t ! (b\f)$); }

\draw [strand] plot[smooth]
coordinates {(b1) (a8-16) (a15-16) (b6)}
[postaction=decorate, decoration={markings,
 mark= at position 0.2 with \strarrow,
 mark= at position 0.5 with \strarrow, 
 mark= at position 0.8 with \strarrow }];
 
\draw [strand] plot[smooth]
coordinates {(b6) (a14-15) (a12-14)(a11-12) (a10-11) (b3)}
[postaction=decorate, decoration={markings,
 mark= at position 0.15 with \strarrow, mark= at position 0.35 with \strarrow,
 mark= at position 0.53 with \strarrow, mark= at position 0.7 with \strarrow,
 mark= at position 0.87 with \strarrow }];
 
\draw [strand] plot[smooth]
coordinates {(b3) (a9-10) (a8-9) (b1)}
[postaction=decorate, decoration={markings,
 mark= at position 0.2 with \strarrow,
 mark= at position 0.5 with \strarrow, 
 mark= at position 0.8 with \strarrow }];

\draw [strand] plot[smooth]
coordinates {(b2) (a9-10) (a10-11) (b4)}
 [postaction=decorate, decoration={markings,
 mark= at position 0.2 with \strarrow,
 mark= at position 0.5 with \strarrow, 
 mark= at position 0.8 with \strarrow }];

\draw [strand] plot[smooth]
coordinates {(b4) (a11-12) (a12-13) (a13-15) (a15-16) (b7)}
[postaction=decorate, decoration={markings,
 mark= at position 0.15 with \strarrow, mark= at position 0.35 with \strarrow,
 mark= at position 0.55 with \strarrow, mark= at position 0.7 with \strarrow,
 mark= at position 0.87 with \strarrow }];

\draw [strand] plot[smooth]
coordinates {(b7) (a8-16) (a8-13) (a12-13) (a12-14) (b5)}
[postaction=decorate, decoration={markings,
 mark= at position 0.15 with \strarrow, mark= at position 0.315 with \strarrow,
 mark= at position 0.5 with \strarrow, mark= at position 0.7 with \strarrow,
 mark= at position 0.88 with \strarrow }];

\draw [strand] plot[smooth]
coordinates {(b5) (a14-15) (a13-15) (a8-13) (a8-9) (b2)}
[postaction=decorate, decoration={markings,
 mark= at position 0.13 with \strarrow, mark= at position 0.33 with \strarrow,
 mark= at position 0.5 with \strarrow, mark= at position 0.7 with \strarrow,
 mark= at position 0.88 with \strarrow }];
\end{tikzpicture}\hfill
\begin{tikzpicture}[scale=3,baseline=(bb.base),yscale=-1]
\path (0,0) node (bb) {};

\foreach \n/\m/\a in {1/4/0, 2/3/0, 3/2/5, 4/1/10, 5/7/0, 6/6/-3, 7/5/0}
{ \coordinate (b\n) at (\bstart-\seventh*\n+\a:1.0);
  \draw (\bstart-\seventh*\n+\a:1.1) node {\(\m\)}; }

\foreach \n/\m in {8/1, 9/2, 10/3, 11/4, 14/5, 15/6, 16/7}
  {\coordinate (b\n) at ($0.65*(b\m)$);}

\coordinate (b13) at ($(b15) - (b16) + (b8)$);
\coordinate (b12) at ($(b14) - (b15) + (b13)$);

\foreach \n/\x/\y in {13/-0.03/-0.03, 12/-0.22/0.0, 14/-0.07/-0.03, 11/0.05/0.02, 16/-0.02/0.02}
  {\coordinate (b\n)  at ($(b\n) + (\x,\y)$); } 

\foreach \h/\t in {1/8, 2/9, 3/10, 4/11, 5/14, 6/15, 7/16, 
 8/9, 9/10, 10/11,11/12, 12/13, 13/8, 14/15, 15/16, 12/14, 13/15, 8/16}
{ \draw [bipedge] (b\h)--(b\t); }

\foreach \n in {8,10,12,15} 
  {\draw [\graphcolor] (b\n) circle(\dotrad) [fill=white];} \foreach \n in {9,11,13, 14,16}  
  {\draw [\graphcolor] (b\n) circle(\dotrad) [fill=\graphcolor];} 

\foreach \n/\m/\l/\a in {1/124/156/0, 2/234/145/-1, 3/345/345/1, 4/456/235/10, 5/256/123/5, 6/267/127/0, 7/127/167/0}
{ \draw [\frozcolor] (\bstart+\seventh/2-\seventh*\n+\a:0.9) node (q\m) {\scriptsize \(\mathbf{\l}\)}; }

\foreach \m/\l/\a/\r in {247/157/\bstart/0.42, 245/135/10/0.25, 257/137/210/0.36}
{ \draw [\quivcolor] (\a:\r) node (q\m) {\scriptsize \(\mathbf{\l}\)}; }

\foreach \e/\f/\t in {2/9/0.5, 4/11/0.5, 5/14/0.5, 7/16/0.5, 
 8/9/0.5, 9/10/0.5, 10/11/0.5,11/12/0.5, 12/13/0.45, 8/13/0.6, 
 14/15/0.5, 15/16/0.6, 12/14/0.45, 13/15/0.4, 8/16/0.6}
{\coordinate (a\e-\f) at ($(b\e) ! \t ! (b\f)$); }

\draw [strand] plot[smooth]
coordinates {(b1) (a8-16) (a15-16) (b6)}
[postaction=decorate, decoration={markings,
 mark= at position 0.2 with \strarrow,
 mark= at position 0.5 with \strarrow, 
 mark= at position 0.8 with \strarrow }];
 
\draw [strand] plot[smooth]
coordinates {(b6) (a14-15) (a12-14)(a11-12) (a10-11) (b3)}
[postaction=decorate, decoration={markings,
 mark= at position 0.15 with \strarrow, mark= at position 0.35 with \strarrow,
 mark= at position 0.53 with \strarrow, mark= at position 0.7 with \strarrow,
 mark= at position 0.87 with \strarrow }];
 
\draw [strand] plot[smooth]
coordinates {(b3) (a9-10) (a8-9) (b1)}
[postaction=decorate, decoration={markings,
 mark= at position 0.2 with \strarrow,
 mark= at position 0.5 with \strarrow, 
 mark= at position 0.8 with \strarrow }];

\draw [strand] plot[smooth]
coordinates {(b2) (a9-10) (a10-11) (b4)}
 [postaction=decorate, decoration={markings,
 mark= at position 0.2 with \strarrow,
 mark= at position 0.5 with \strarrow, 
 mark= at position 0.8 with \strarrow }];

\draw [strand] plot[smooth]
coordinates {(b4) (a11-12) (a12-13) (a13-15) (a15-16) (b7)}
[postaction=decorate, decoration={markings,
 mark= at position 0.15 with \strarrow, mark= at position 0.35 with \strarrow,
 mark= at position 0.55 with \strarrow, mark= at position 0.7 with \strarrow,
 mark= at position 0.87 with \strarrow }];

\draw [strand] plot[smooth]
coordinates {(b7) (a8-16) (a8-13) (a12-13) (a12-14) (b5)}
[postaction=decorate, decoration={markings,
 mark= at position 0.15 with \strarrow, mark= at position 0.315 with \strarrow,
 mark= at position 0.5 with \strarrow, mark= at position 0.7 with \strarrow,
 mark= at position 0.88 with \strarrow }];

\draw [strand] plot[smooth]
coordinates {(b5) (a14-15) (a13-15) (a8-13) (a8-9) (b2)}
[postaction=decorate, decoration={markings,
 mark= at position 0.13 with \strarrow, mark= at position 0.33 with \strarrow,
 mark= at position 0.5 with \strarrow, mark= at position 0.7 with \strarrow,
 mark= at position 0.88 with \strarrow }];
\end{tikzpicture}
\end{minipage}
\caption{A plabic graph of type \((3,7)\). The left-hand figure shows the source labels, while the right-hand figure shows the target labels. The labels of boundary regions, which form the necklaces, are displayed in blue since these will later label the frozen Plücker coordinates in a cluster structure (see Theorem~\ref{t:GalLam}).}
\label{f:labels}
\end{figure}
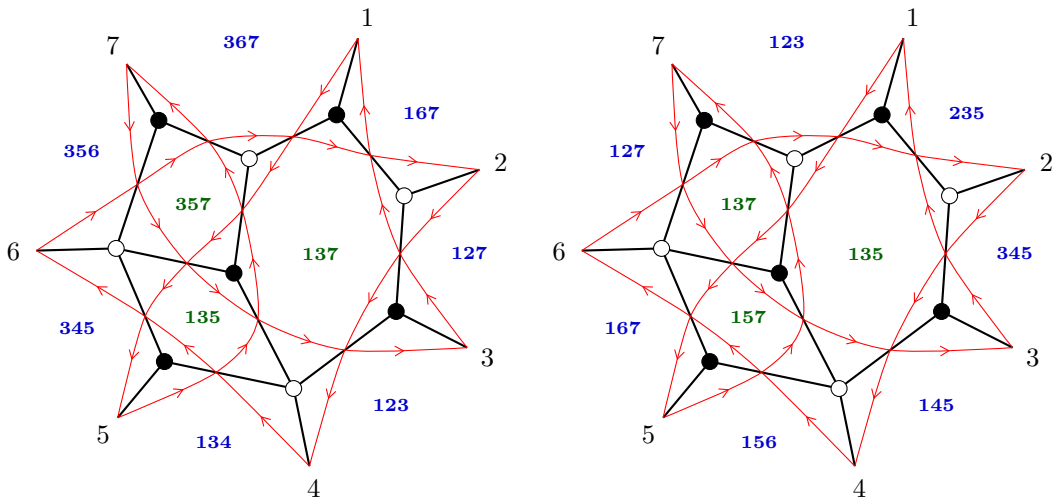

This definition of the type of a Postnikov diagram is equivalent to that given in \cite[Def.~2.5]{CKP} in terms of the combinatorics of the plabic graph, by \cite[Thm.~5.3]{MulSpe-Twist} (see also \cite[Prop.~8.2]{CKP}).
The target necklace is sometimes referred to simply as the necklace, and the source necklace as the reverse necklace (or upper necklace \cite{Oh}).
Strictly speaking, both \(\srcneck_D\) and \(\tgtneck_D\) should be regarded as a cyclically ordered sequence of \(n\) labels, possibly with repetition, by using the natural cyclic ordering of the boundary regions; when \(D\) is disconnected, some boundary regions meet the boundary in several distinct segments, and so their labels appear in \(\srcneck_D\) and \(\tgtneck_D\) multiple times.
However, since we will mainly use the necklaces in the case that \(D\) is connected, in which there are always exactly \(n\) boundary regions with distinct labels, and we will not explicitly use their cyclic ordering, we will ignore this subtlety from now on.

\begin{defn}
\label{d:pm}
Let \(\Gamma\) be a plabic graph.
A \emph{perfect matching} of \(\Gamma\) is a set \(\mu\) of edges (and half-edges) of \(\Gamma\) such that each node is incident with exactly one edge in \(\mu\).
The \emph{boundary value} \(\bdry\mu\subset\ZZ_n\) consists of the marked points \(i\) such that either
\begin{enumerate}
\item the half-edge at \(i\) is incident with a white node and included in \(\mu\), or
\item the half-edge at \(i\) is incident with a black node and not included in \(\mu\).
\end{enumerate}
If \(D\) is a Postnikov diagram, its associated \emph{positroid} \(\posit_D\) is the set of subsets of \(\ZZ_n\) which appear as boundary values of perfect matchings of \(\Gamma_D\) \cite[Thm.~3.1]{MulSpe-Twist}.
All of these subsets have the same cardinality, which may be computed by subtracting the number of black nodes of \(\Gamma_D\) from the number of white nodes, then adding the number of half-edges incident with a black node.
\end{defn}

\begin{rem}
The terminology `positroid' refers to the fact that \(\posit_D\) is a matroid represented by a totally positive matrix.
The diagram \(D\) is connected if and only if \(\posit_D\) is connected as a matroid, i.e.\ is not expressible as a non-trivial direct sum \cite[Prop.~5.8]{OPS-PGs}.
\end{rem}

\begin{defn}
Let \(D\) be a Postnikov diagram, and let \(e\) be an edge of the plabic graph \(\Gamma_D\), corresponding to a crossing of strands in \(D\) (or to a boundary marked point).
Consider the two strand segments obtained by following these strands forwards from the crossing until they terminate; by \ref{d:P4}, these segments do not themselves cross.
This pair of strand segments can be completed to a closed curve using either of the two boundary arcs in the disc between their endpoints.
One of these two closed curves bounds a region containing the angle between the two strands as they exit the crossing, and we call this the \emph{downstream wedge} at \(e\).

The \emph{upstream wedge} at \(e\) is defined similarly, instead following the strands backwards from the crossing and taking the region containing the angle between them as they enter the crossing.
Examples are shown in Figure~\ref{f:wedges}.

\begin{figure}
\begin{minipage}{0.6\textwidth}
\begin{tikzpicture}[baseline=0]
\draw [white,fill=magenta!20] (0,0) to (60:2) arc (60:120:2) (120:2) to (0,0);
\draw [white,fill=blue!20] (0,0) to (-60:2) arc (-60:-120:2) (-120:2) to (0,0);
\draw [bipedge] (-0.6,0)--(0.6,0);
\draw [\graphcolor] (-0.6,0) circle(2.5*\dotrad) [fill=black];
\draw [\graphcolor] (0.6,0) circle(2.5*\dotrad) [fill=white];
\draw [strand] plot[smooth]
coordinates {(-120:2) (0,0) (60:2)}
[postaction=decorate, decoration={markings,
 mark= at position 0.3 with \strarrow,
 mark= at position 0.8 with \strarrow }];
\draw [strand] plot[smooth]
 coordinates {(-60:2) (0,0) (120:2)}
 [postaction=decorate, decoration={markings,
  mark= at position 0.3 with \strarrow,
  mark= at position 0.8 with \strarrow }];
\draw [boundary] (0,0) circle(2.0);
\end{tikzpicture}\hfill
\begin{tikzpicture}[baseline=0]
\draw [fill=magenta, fill opacity=0.2, draw opacity=0] (120:2) to (0,0) plot[smooth]
coordinates {(0,0) (0.6,0.3) (-20:2)} arc (-20:120:2) (120:2);
\draw [fill=blue, fill opacity=0.2, draw opacity=0] (0,0) to (-120:2) arc (-120:20:2) (20:2) plot[smooth]
 coordinates {(20:2) (0.6,-0.3) (0,0)};
\draw [bipedge] (-0.6,0)--(0.6,0);
\draw [\graphcolor] (-0.6,0) circle(2.5*\dotrad) [fill=black];
\draw [\graphcolor] (0.6,0) circle(2.5*\dotrad) [fill=white];
\draw [strand] (-120:2) to (0,0)
[postaction=decorate, decoration={markings,
 mark= at position 0.6 with \strarrow }];
\draw [strand] (0,0) to (120:2)
[postaction=decorate, decoration={markings,
 mark= at position 0.6 with \strarrow }];
\draw [strand] plot[smooth]
coordinates {(0,0) (0.6,0.3) (-20:2)}
[postaction=decorate, decoration={markings,
 mark= at position 0.25 with \strarrow,
 mark= at position 0.8 with \strarrow }];
\draw [strand] plot[smooth]
 coordinates {(20:2) (0.6,-0.3) (0,0)}
 [postaction=decorate, decoration={markings,
  mark= at position 0.2 with \strarrow,
  mark= at position 0.70 with \strarrow }];
\draw [boundary] (0,0) circle(2.0);
\end{tikzpicture}
\end{minipage}
\caption{The downstream (magenta) and upstream (blue) wedge of an edge in a plabic graph \(\Gamma\). While \ref{d:P4} guarantees that these regions are indeed wedge-shaped, the right-hand figure demonstrates that they may not be disjoint.}
\label{f:wedges}
\end{figure}
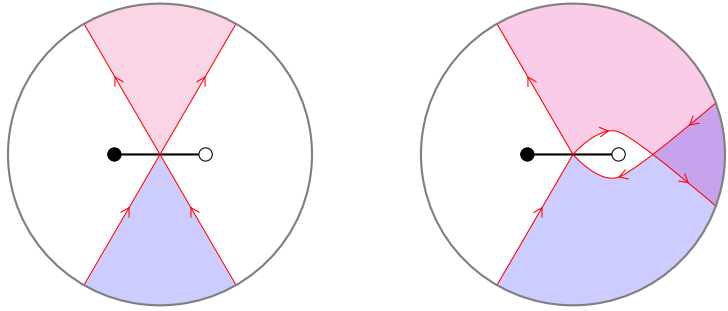
\end{defn}

\begin{rem}
\label{r:bdry-wedges}
If \(e\) is a half-edge, the crossing at \(e\) is itself the termination point of one of the two strands, and the start point of the other.
Thus, the downstream wedge at \(e\) is either the left-hand side of the strand beginning at \(e\), if \(e\) is incident with a white node, or the right-hand side of this strand, if \(e\) is incident with a black node.
The same holds for the upstream wedge, considering instead the strand ending at \(e\).
\end{rem}

\begin{prop}
\label{p:MSmats}
For each alternating region \(j\) of \(D\), let \(\MSsrc{j}\) be the set of edges \(e\in\Gamma_D\) such that \(j\) is contained in the downstream wedge of \(e\), and let \(\MStgt{j}\) be the set of edges \(e\in\Gamma_D\) such that \(j\) is contained in the upstream wedge of \(e\).
Then \(\MSsrc{j}\) and \(\MStgt{j}\) are perfect matchings of \(\Gamma_D\).
\end{prop}
\begin{proof}
This statement was originally proved by Muller and Speyer \cite[Thm.~5.3]{MulSpe-Twist}. A direct geometric argument is given in \cite[\S7]{CKP}.
\end{proof}

By the reasoning in Remark~\ref{r:bdry-wedges}, we see that
\begin{equation}
\label{eq:MSbdry}
\bdry\MSsrc{j}=\srclab{j},\quad
\bdry\MStgt{j}=\tgtlab{j},
\end{equation}
which has several immediate corollaries.

\begin{cor}
\label{c:bdry-card}
If \(D\) is a Postnikov diagram of type \((k,n)\), then the boundary value of any perfect matching of \(\Gamma_D\) is a \(k\)-subset.
Moreover, for any alternating region \(j\) of a Postnikov diagram \(d\), the source and target labels \(\srclab{j}\) and \(\tgtlab{j}\) are contained in the positroid \(\posit_D\).
In particular, there are inclusions \(\srcneck_D\subset\posit_D\) and \(\tgtneck_D\subset\posit_D\).
\end{cor}

\subsection{Positroid varieties and their cluster structures}

We consider the Grassmannian \(\Grass{k}{n}\) of \(k\)-dimensional subspaces of \(\CC^n\) (i.e.\ an \(n\)-dimensional complex vector space with a chosen ordered basis), in its Plücker embedding \(\Plueck{}\colon\Grass{k}{n}\to\PP(\Wedge{k}\CC^n)\) with Plücker coordinates \(\Plueck{I}\) indexed by \(k\)-subsets \(I\subset\ZZ_n\).
At this point, we are treating \(\ZZ_n\) as a linearly ordered set, to determine the sign for each Plücker coordinate \(\Plueck{I}\).
We will sometimes write \(\varPlueck{I}=\Plueck{I}\) when this improves readability (for example if \(I=\srclab{j}\) or \(I=\tgtlab{j}\)).

\begin{defn}
Let \(\posit=\posit_D\) be the positroid associated to a Postnikov diagram \(D\) of type \((k,n)\).
The \emph{positroid variety} \(\posvar_\posit\) is the subset of \(\Grass{k}{n}\) defined by the vanishing of the Plücker coordinates \(\Plueck{I}\) for \(I\notin\posit\) \cite[Thm.~5.15]{KLS}.
The \emph{open positroid variety} \(\openposvar_\posit\) is the subset of \(\posvar_\posit\) on which the Plücker coordinates \(\Plueck{I}\) with \(I\in\tgtneck_D\subset\posit\) do not vanish \cite[\S2.2]{FSB}.

We write \(\Grasscone{k}{n}\) for the affine cone on the Grassmannian (on which Plücker coordinates are actual functions, not just projective coordinates), which naturally contains the cones \(\posvarcone_\posit\) and \(\openposvarcone_\posit\) on \(\posvar_\posit\) and \(\openposvar_\posit\) respectively.
These cones are defined inside \(\Grasscone{k}{n}\) by the same vanishing and non-vanishing conditions on Plücker coordinates as for the original positroid varieties inside \(\Grass{k}{n}\).
\end{defn}

We remark that, despite what the notation may suggest, the necklaces \(\srcneck_D\) and \(\tgtneck_D\) may be determined entirely from the positroid \(\posit\), without reference to a Postnikov diagram \(D\) \cite[\S16]{Postnikov-PosGrass}; as a result, if \(D\) and \(D'\) are Postnikov diagrams such that \(\posit_D=\posit_{D'}\) (for example, if \(D\) and \(D'\) are related by a geometric exchange, also known as a square move \cite[\S12]{Postnikov-PosGrass} or a mutation), then \(\srcneck_D=\srcneck_{D'}\) and \(\tgtneck_D=\tgtneck_{D'}\) also.
In particular, the open positroid variety \(\openposvar_\posit\) depends only on \(\posit\), and not on \(D\).

\begin{prop}
\label{p:sourceneck}
If \(x\in\posvar_\posit\), then \(x\in\openposvar_{\posit}\) if and only if \(\Plueck{I}(x)\ne0\) for all \(I\in\srcneck_D\subseteq\posit\).
That is, \(\openposvar_{\posit}\) is also characterised inside \(\posvar_\posit\) by non-vanishing of the Plücker coordinates labelled by elements of the source necklace.
\end{prop}
\begin{proof}
Recall that \(\tgtlab{j}=\pi_D(\srclab{j})\) for every region \(j\) of \(D\), in particular for the boundary regions, whose labels give the elements of the necklaces.
Thus, it follows from a result of Muller--Speyer \cite[Prop.~7.13]{MulSpe-Twist} that when \(j\) is a boundary region, the Plücker coordinate \(\varPlueck{\srclab{j}}|_{\smash{\openposvarcone_\posit}}\) is a Laurent monomial in the Plücker coordinates \(\Plueck{I}\) for \(I\in\tgtneck_D\), and hence does not vanish on the open positroid variety.
While Muller--Speyer's more general result is stated with an automorphism (the square of the right twist) applied to \(\varPlueck{\srclab{j}}|_{\smash{\openposvarcone_\posit}}\), the argument opening the proof of \cite[Prop.~6.6]{MulSpe-Twist} shows that the right twist inverts \(\varPlueck{\srclab{j}}|_{\smash{\openposvarcone_\posit}}\) for each boundary region \(j\), and hence its square fixes this Plücker coordinate.

As Muller--Speyer point out, a completely analogous result holds with the roles of \(\srcneck_D\) and \(\tgtneck_D\) switched, and from this we deduce the converse implication.
\end{proof}

\begin{eg}
\label{eg:grass}
If \(\pi_D(i)=i+k\) for all \(i\in\ZZ_n\), then we call \(D\) \emph{uniform}. In this case, every \(k\)-subset appears in \(\posit=\posit_D\), and \(\srcneck_D=\tgtneck_D\) consists of the cyclic intervals \(\{i+1,\dotsc,i+k\}\) for \(i\in\{1,\dotsc,n\}\); in particular, \(D\) has type \((k,n)\).
(While the two necklaces have the same elements, their indexing by the boundary regions of \(D\) is different.)
Thus, \(\posvar_\posit=\Grass{k}{n}\) is the whole Grassmannian, while \(\openposvar_\posit\), defined by the non-vanishing of the Plücker coordinates with cyclic interval labels, is sometimes called the \emph{big cell}, being the unique top-dimensional piece of the positroid stratification.
\end{eg}

\begin{defn}
Let \(D\) be a Postnikov diagram.
The associated quiver \(Q_D\) has as vertices the alternating regions of \(D\), and as arrows the crossings of \(D\), with arrows oriented in the forward direction of the crossing strands (see Figure~\ref{f:quiv-arrow}).
As usual, boundary marked points are treated as crossings in this definition.

The \emph{frozen} subquiver \(F_D\subset Q_D\) has as vertices the alternating regions incident with the boundary of the disc, and as arrows the crossings at the boundary marked points.
\end{defn}

\begin{figure}
\begin{tikzpicture}
\draw [bipedge] (-1,0)--(1,0);
\draw [\graphcolor] (-1,0) circle(2.5*\dotrad) [fill=black];
\draw [\graphcolor] (1,0) circle(2.5*\dotrad) [fill=white];
\draw [strand] plot[smooth]
coordinates {(-1,-1) (-0.4,-0.6) (0,0) (0.4,0.6) (1,1)}
[postaction=decorate, decoration={markings,
 mark= at position 0.3 with \strarrow,
 mark= at position 0.8 with \strarrow }];
\draw [strand] plot[smooth]
 coordinates {(1,-1) (0.4,-0.6) (0,0) (-0.4,0.6) (-1,1)}
 [postaction=decorate, decoration={markings,
  mark= at position 0.3 with \strarrow,
  mark= at position 0.8 with \strarrow }];
\draw [\quivcolor] (0,-1) node (s) {\(\bullet\)};
\draw [\quivcolor] (0,1) node (t) {\(\bullet\)};
\draw [quivarrow] (s) -- (t);
\end{tikzpicture}
\caption{The arrow of \(Q_D\) at a crossing of \(D\), or equivalently at an edge of \(\Gamma_D\). The rules at half-edges are obtained by cutting the figure in half along the arrow.}
\label{f:quiv-arrow}
\end{figure}
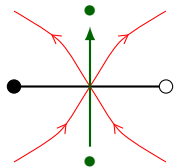

In terms of the plabic graph \(\Gamma_D\), the quiver \(Q_D\) has one vertex for each region cut out by \(\Gamma_D\), with frozen vertices corresponding to regions on the boundary, and one arrow for each (half-)edge of \(\Gamma_D\).
This arrow joins the two regions incident with the edge and is oriented so that the black node of the edge is on the left, and/or the white node of the edge is on the right.
\begin{defn}
For a Postnikov diagram \(D\), we write \(\clustalg{D}\) for the cluster algebra with initial ice quiver \(Q_D\), with invertible frozen variables and coefficients in \(\CC\).
The initial cluster variables are written \(x_j\), for \(j\) a vertex of \(Q_D\) (i.e.\ an alternating region of \(D\)).
For us, the term `cluster variable' includes frozen variables unless explicitly stated.
\end{defn}

The frozen arrows (or any other arrows between frozen vertices) play no role in the construction of \(\clustalg{D}\), but will be important later on in its categorification.
Since every element of \(\clustalg{D}\) may be written as a Laurent polynomial in the initial variables \(x_j\) by the Laurent phenomenon, any homomorphism \(\varphi\colon\clustalg{D}\to R\) of commutative \(\CC\)-algebras is entirely determined by its values on these variables.
This abstract cluster algebra is related to positroid varieties by the following theorem, due to Galashin and Lam.

\begin{thm}[{\cite[Thm.~3.5]{GalLam}}]
\label{t:GalLam}
Let \(D\) be a Postnikov diagram of type \((k,n)\) with associated positroid \(\posit=\posit_D\).
Then there are \(\CC\)-algebra isomorphisms
\[\GLsrc\colon\clustalg{D}\isoto\CC[\openposvarcone_\posit],\quad\GLtgt\colon\clustalg{D}\isoto\CC[\openposvarcone_\posit]\]
such that \(\GLsrc(x_j)=\varPlueck{\srclab{j}}|_{\openposvarcone_\posit}\) and \(\GLtgt(x_j)=\varPlueck{\tgtlab{j}}|_{\openposvarcone_\posit}\).
\end{thm}

This result has several precursors: it was proved by Serhiyenko, Sherman-Bennett and Williams \cite{SSBW} in the case that \(\openposvar_\posit\) is a Schubert variety, and by Scott \cite{Scott-Grass} for uniform diagrams, i.e.\ in the setting of Example~\ref{eg:grass}.
In this special case, Scott shows that the maps \(\GLsrc\) and \(\GLtgt\) extend to isomorphisms of the cluster algebra with non-invertible frozen variables associated to \(Q_D\) with the homogeneous coordinate ring \(\CC[\Grasscone{k}{n}]\) of the whole Grassmannian, but this stronger result does not have an analogue for more general positroid varieties.
From now on, we will usually drop the notation indicating that Plücker coordinates have been restricted to \(\openposvarcone_\posit\) when this is clear from the context.

Strictly speaking, Galashin and Lam only state the result of Theorem~\ref{t:GalLam} referring to \(\GLsrc\), but the statement concerning \(\GLtgt\) may be deduced from this in several ways: see \cite[Rem.~2.16, Thm.~5.17]{FSB}, or Proposition~\ref{p:op-GL} below.

In discussing cluster characters in general in Section~\ref{s:categorification}, we will need to refer to the upper cluster algebra generated by a seed, as defined by Berenstein, Fomin and Zelevinsky \cite[Def.~1.1]{BFZ-CA3}.
While in general this may be strictly larger than the ordinary cluster algebra, Muller and Speyer show that there is no difference in the case of positroid varieties.

\begin{prop}
\label{p:uca}
Let \(\clust{A}_D^+\) be the upper cluster algebra (with invertible frozen variables) defined from the same initial data as \(\clust{A}_D\). Then \(\clust{A}_D^+=\clust{A}_D\).
\end{prop}
\begin{proof}
This combines \cite[Prop.~2.6, Thm.~3.3]{MulSpe-LocAcycl}, showing that \(\clust{A}_D\) is locally acyclic, with \cite[Thm.~2]{Muller-AEqualsU}, showing that this property implies equality with the upper cluster algebra.
\end{proof}

\begin{defn}
Let \(R\) be a commutative \(\KK\)-algebra for some field \(\KK\).
A \emph{cluster structure} on \(R\) is a choice of \(\KK\)-algebra isomorphism \(\eta\colon\clust{A}\isoto R\) where \(\clust{A}\) is a cluster algebra over \(\KK\).
Given this data, we define the cluster variables, cluster monomials, clusters, etc.\ of \(R\) to be the images under \(\eta\) of the corresponding elements of \(\clust{A}\).
An \emph{upper cluster structure} is defined analogously to be a \(\KK\)-algebra isomorphism \(\eta\colon\clust{A}^+\isoto R\) for some upper cluster algebra \(\clust{A}^+\).
\end{defn}

By Theorem~\ref{t:GalLam}, the coordinate ring \(\CC[\openposvarcone_\posit]\) carries two cluster algebra structures; we call that arising from the isomorphism \(\GLsrc\) the \emph{source-labelled} structure and that from the isomorphism \(\GLtgt\) the \emph{target-labelled} structure.
While the two structures have the same underlying cluster combinatorics, since \(\GLsrc\) and \(\GLtgt\) have the same domain, they may not have the same cluster variables.
Indeed, while the two sets of cluster variables do coincide for uniform diagrams (for example, for a uniform diagram of type \((2,n)\) the cluster variables in either structure are precisely the Plücker coordinates), they are typically different.
By Proposition~\ref{p:uca}, both \(\GLsrc\) and \(\GLtgt\) are also upper cluster algebra structures on \(\CC[\openposvarcone_\posit]\).

\begin{eg}
\label{eg:running1}
Consider the plabic graph shown in Figure~\ref{f:eg}, together with its quiver \(Q_D\) and its source and target labellings.
Since there are \(7\) boundary marked points and each label has cardinality \(3\), the corresponding open positroid variety \(\openposvar_\posit\) is a subvariety of \(\Grass{3}{7}\).
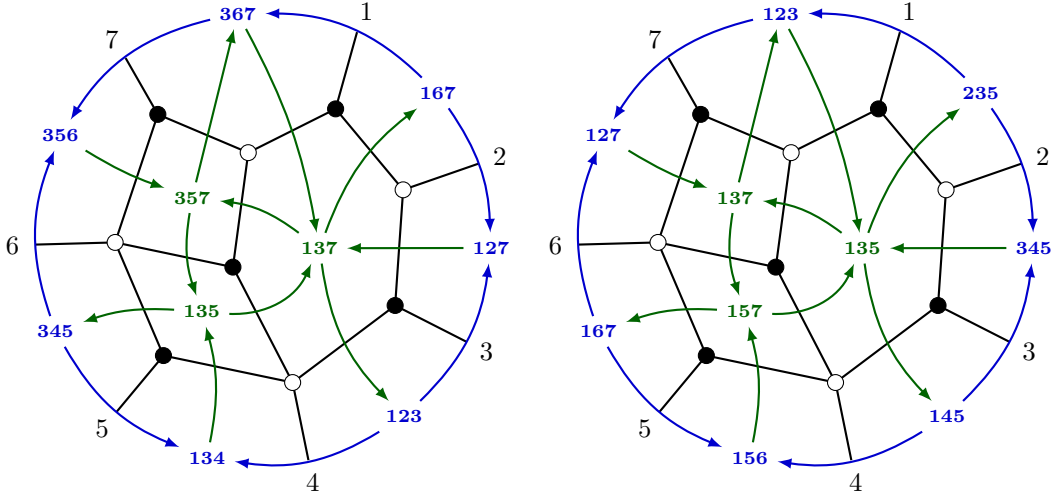
\begin{figure}
\begin{minipage}{0.95\textwidth}
\begin{tikzpicture}[scale=3,baseline=(bb.base),yscale=-1]
\path (0,0) node (bb) {};

\foreach \n/\m/\a in {1/4/0, 2/3/0, 3/2/5, 4/1/10, 5/7/0, 6/6/-3, 7/5/0}
{ \coordinate (b\n) at (\bstart-\seventh*\n+\a:1.0);
  \draw (\bstart-\seventh*\n+\a:1.1) node {\(\m\)}; }

\foreach \n/\m in {8/1, 9/2, 10/3, 11/4, 14/5, 15/6, 16/7}
  {\coordinate (b\n) at ($0.65*(b\m)$);}

\coordinate (b13) at ($(b15) - (b16) + (b8)$);
\coordinate (b12) at ($(b14) - (b15) + (b13)$);

\foreach \n/\x/\y in {13/-0.03/-0.03, 12/-0.22/0.0, 14/-0.07/-0.03, 11/0.05/0.02, 16/-0.02/0.02}
  {\coordinate (b\n)  at ($(b\n) + (\x,\y)$); } 

\foreach \h/\t in {1/8, 2/9, 3/10, 4/11, 5/14, 6/15, 7/16, 
 8/9, 9/10, 10/11,11/12, 12/13, 13/8, 14/15, 15/16, 12/14, 13/15, 8/16}
{ \draw [bipedge] (b\h)--(b\t); }

\foreach \n in {8,10,12,15} 
  {\draw [\graphcolor] (b\n) circle(\dotrad) [fill=white];} \foreach \n in {9,11,13, 14,16}  
  {\draw [\graphcolor] (b\n) circle(\dotrad) [fill=\graphcolor];} 

\foreach \e/\f/\t in {2/9/0.5, 4/11/0.5, 5/14/0.5, 7/16/0.5, 
 8/9/0.5, 9/10/0.5, 10/11/0.5,11/12/0.5, 12/13/0.45, 8/13/0.6, 
 14/15/0.5, 15/16/0.6, 12/14/0.45, 13/15/0.4, 8/16/0.6}
{\coordinate (a\e-\f) at ($(b\e) ! \t ! (b\f)$); }

\foreach \n/\m/\l/\a in {1/124/134/0, 2/234/123/-1, 3/345/127/1, 4/456/167/10, 5/256/367/5, 6/267/356/0, 7/127/345/0}
{ \draw [\frozcolor] (\bstart+\seventh/2-\seventh*\n+\a:1) node (q\m) {\scriptsize \(\mathbf{\l}\)}; }

\foreach \m/\l/\a/\r in {247/135/\bstart/0.42, 245/137/10/0.25, 257/357/210/0.36}
{ \draw [\quivcolor] (\a:\r) node (q\m) {\scriptsize \(\mathbf{\l}\)}; }

\foreach \t/\h/\a in {124/247/13, 247/127/11, 245/234/18, 247/245/34, 257/247/14, 245/257/15,
 267/257/6, 257/256/0, 256/245/-9, 345/245/0, 245/456/-16}
{ \draw [quivarrow]  (q\t) edge [bend left=\a] (q\h); }

\foreach \t/\h/\a in {234/124/-20, 234/345/19, 456/345/-16, 456/256/22, 256/267/22, 127/267/-20,
 127/124/19}
{ \draw [frozarrow]  (q\t) edge [bend left=\a] (q\h); }
\end{tikzpicture}\hfill
\begin{tikzpicture}[scale=3,baseline=(bb.base),yscale=-1]
\path (0,0) node (bb) {};

\foreach \n/\m/\a in {1/4/0, 2/3/0, 3/2/5, 4/1/10, 5/7/0, 6/6/-3, 7/5/0}
{ \coordinate (b\n) at (\bstart-\seventh*\n+\a:1.0);
  \draw (\bstart-\seventh*\n+\a:1.1) node {\(\m\)}; }

\foreach \n/\m in {8/1, 9/2, 10/3, 11/4, 14/5, 15/6, 16/7}
  {\coordinate (b\n) at ($0.65*(b\m)$);}

\coordinate (b13) at ($(b15) - (b16) + (b8)$);
\coordinate (b12) at ($(b14) - (b15) + (b13)$);

\foreach \n/\x/\y in {13/-0.03/-0.03, 12/-0.22/0.0, 14/-0.07/-0.03, 11/0.05/0.02, 16/-0.02/0.02}
  {\coordinate (b\n)  at ($(b\n) + (\x,\y)$); } 

\foreach \h/\t in {1/8, 2/9, 3/10, 4/11, 5/14, 6/15, 7/16, 
 8/9, 9/10, 10/11,11/12, 12/13, 13/8, 14/15, 15/16, 12/14, 13/15, 8/16}
{ \draw [bipedge] (b\h)--(b\t); }

\foreach \n in {8,10,12,15} 
  {\draw [\graphcolor] (b\n) circle(\dotrad) [fill=white];} \foreach \n in {9,11,13, 14,16}  
  {\draw [\graphcolor] (b\n) circle(\dotrad) [fill=\graphcolor];} 

\foreach \e/\f/\t in {2/9/0.5, 4/11/0.5, 5/14/0.5, 7/16/0.5, 
 8/9/0.5, 9/10/0.5, 10/11/0.5,11/12/0.5, 12/13/0.45, 8/13/0.6, 
 14/15/0.5, 15/16/0.6, 12/14/0.45, 13/15/0.4, 8/16/0.6}
{\coordinate (a\e-\f) at ($(b\e) ! \t ! (b\f)$); }

\foreach \n/\m/\l/\a in {1/124/156/0, 2/234/145/-1, 3/345/345/1, 4/456/235/10, 5/256/123/5, 6/267/127/0, 7/127/167/0}
{ \draw [\frozcolor] (\bstart+\seventh/2-\seventh*\n+\a:1) node (q\m) {\scriptsize \(\mathbf{\l}\)}; }

\foreach \m/\l/\a/\r in {247/157/\bstart/0.42, 245/135/10/0.25, 257/137/210/0.36}
{ \draw [\quivcolor] (\a:\r) node (q\m) {\scriptsize \(\mathbf{\l}\)}; }

\foreach \t/\h/\a in {124/247/13, 247/127/11, 245/234/18, 247/245/34, 257/247/14, 245/257/15,
 267/257/6, 257/256/0, 256/245/-9, 345/245/0, 245/456/-16}
{ \draw [quivarrow]  (q\t) edge [bend left=\a] (q\h); }

\foreach \t/\h/\a in {234/124/-20, 234/345/19, 456/345/-16, 456/256/22, 256/267/22, 127/267/-20,
 127/124/19}
{ \draw [frozarrow]  (q\t) edge [bend left=\a] (q\h); }
\end{tikzpicture}
\end{minipage}
\caption{A plabic graph \(D\) for a positroid variety in \(\Grass{3}{7}\) with its dual quiver. The quiver vertices are given their source labels on the left, and their target labels on the right (cf.~Figure~\ref{f:labels}).}
\label{f:eg}
\end{figure}
By computing perfect matchings of \(\Gamma_D\), we see that it is determined by the equations
\[\Plueck{234}=\Plueck{456}=\Plueck{457}=\Plueck{467}=\Plueck{567}=0,\]
together with the non-vanishing of the \(7\) frozen variables in either of the two cluster structures.
As a sample calculation, note that the Plücker relation
\[\Plueck{167}\Plueck{345}-\Plueck{367}\Plueck{145}+\Plueck{467}\Plueck{135}-\Plueck{567}\Plueck{134}=0\]
on \(\Grasscone{3}{7}\) implies that \(\Plueck{167}\Plueck{345}=\Plueck{367}\Plueck{145}\) on \(\openposvarcone_\posit\), where \(\Plueck{467}=\Plueck{567}=0\).
Thus, if the frozen variables \(\Plueck{367}\), \(\Plueck{167}\) and \(\Plueck{345}\) in the source-labelled structure are all non-zero, so is the frozen variable
\begin{equation}
\label{eq:froz-ident}
\Plueck{145}=\frac{\Plueck{167}\Plueck{345}}{\Plueck{367}}
\end{equation}
in the target-labelled structure.

Since the full subquiver of \(Q_D\) on the mutable vertices is an oriented \(3\)-cycle, the cluster algebra \(\clustalg{D}\) is of finite type \(\type{A}_3\), so we may compute all of the cluster variables.
For the source-labelled structure, these are shown in Table~\ref{tab:src-eg}, whereas for the target-labelled structure they are shown in Table~\ref{tab:tgt-eg}.
Degrees are computed with respect to the grading in which all Plücker coordinates have degree \(1\).
By direct comparison, we see that these two sets of cluster variables are different; this also follows from \eqref{eq:froz-ident}, which would contradict the linear independence of cluster monomials \cite[Cor.~3.4]{CIKLFP} if \(\Plueck{145}\) were a cluster variable in the source-labelled structure.

We also note that, in contrast to the cluster structure on the Grassmannian itself, on more general positroid varieties it is possible for a non-trivial product of Plücker coordinates to be a cluster variable, e.g.\ \(\Plueck{125}\Plueck{367}\) in the source-labelled structure in this example.
While this structure has \(\Plueck{367}\) as a frozen variable, \(\Plueck{125}\) is not a cluster variable at all, allowing \(\Plueck{125}\Plueck{367}\) to be a cluster variable without violating the linear independence of cluster monomials.
\begin{table}
\caption{The cluster variables for the source-labelled structure on \(\CC[\openposvarcone_\posit]\) for \(\posit\) as in Example~\ref{eg:running1}.}
\label{tab:src-eg}
\begin{tabular}{ll}\toprule
Frozen&\(\Plueck{167},\Plueck{127},\Plueck{123},\Plueck{134},\Plueck{345},\Plueck{356},\Plueck{367}\)\\\midrule
Mutable, degree \(1\)&\(\Plueck{357},\Plueck{347},\Plueck{137},\Plueck{346},\Plueck{136},\Plueck{126},\Plueck{135}\)\\\midrule
Mutable, degree \(2\)&\(\Plueck{125}\Plueck{367},\Plueck{124}\Plueck{367}\)\\\bottomrule
\end{tabular}
\end{table}
\begin{table}
\caption{The cluster variables for the target-labelled structure on \(\CC[\openposvarcone_\posit]\) for \(\posit\) as in Example~\ref{eg:running1}.}
\label{tab:tgt-eg}
\begin{tabular}{ll}\toprule
Frozen&\(\Plueck{123},\Plueck{235},\Plueck{345},\Plueck{145},\Plueck{156},\Plueck{167},\Plueck{127}\)\\\midrule
Mutable, degree \(1\)&\(\Plueck{137},\Plueck{136},\Plueck{135},\Plueck{126},\Plueck{125},\Plueck{245},\Plueck{157}\)\\\midrule
Mutable, degree \(2\)&\(\Plueck{147}\Plueck{235},\Plueck{145}\Plueck{236}\)\\\bottomrule
\end{tabular}
\end{table}
\end{eg}

\subsection{Opposite diagrams}

One technique we will apply repeatedly in this paper is to deduce statements about one Postnikov diagram \(D\) by applying a theorem to its opposite diagram \(D^\op\), obtained by reversing the directions of the strands.
In this subsection, we collect some useful information concerning the relationship between \(D\) and \(D^\op\).
While the sets of alternating regions in these two diagrams are the same, their labels are typically different, and so we write \(\srclab{j}(D)\) and \(\srclab{j}(D^\op)\) to distinguish them.
We use similar notation for target labels, the Muller--Speyer matchings from Proposition~\ref{p:MSmats}, and so on.

For \(I\subset\ZZ_n\), we write \(I^\compl=\ZZ_n\setminus I\) for the complement of \(I\).
The opposite \(\Gamma^\op\) of a plabic graph \(\Gamma\) is obtained by reversing the colours of the nodes; in particular, this means that \(\Gamma\) and \(\Gamma^\op\) have the same edge set.
The next proposition collects statements which follow directly from the definitions.

\begin{prop}
\label{p:op}
Let \(D\) be a Postnikov diagram and \(\Gamma\) a plabic graph.
\begin{enumerate}
\item\label{p:op-plabic} \(\Gamma_{D^\op}=\Gamma_D^\op\).
\item\label{p:op-labs} If \(j\) is an alternating region of \(D\) and \(D^\op\), then \(\srclab{j}(D^\op)=\tgtlab{j}(D)^\compl\) and \(\tgtlab{j}(D^\op)=\srclab{j}(D)^\compl\).
\item If \(D\) has type \((k,n)\), then \(D^\op\) has type \((n-k,n)\).
\item A subset \(\mu\) of the common edge set of \(\Gamma\) and \(\Gamma^\op\) is a perfect matching of \(\Gamma\) if and only if it is a perfect matching of \(\Gamma^\op\).
\item For a perfect matching \(\mu\) of \(\Gamma\) and \(\Gamma^\op\), we have \(\bdry\mu(\Gamma^\op)=(\bdry\mu(\Gamma))^\compl\).
\item\label{p:op-posit} \(\posit_{D^\op}=\{I^\compl:I\in\posit_D\}\), \(\srcneck_{D^\op}=\{I^\compl:I\in\tgtneck_{D}\}\) and \(\tgtneck_{D^\op}=\{I^\compl:I\in\srcneck_D\}\).
\item\label{p:op-MS} If $j$ is an alternating region of \(D\) and \(D^\op\), then \(\MSsrc{j}(D^\op)=\MStgt{j}(D)\) and \(\MStgt{j}(D^\op)=\MSsrc{j}(D)\).
\item\label{p:op-quiv}\(Q_{D^\op}=Q_D^\op\).
\end{enumerate}
\end{prop}

In \ref{p:op-labs} and \ref{p:op-MS}, the second statement follows from the first by swapping the roles of \(D\) and \(D^\op\), since \((D^\op)^\op=D\).

Recall that there is an isomorphism \(\op\colon \Grass{k}{n}\isoto\Grass{n-k}{n}\) given on coordinates by \(\op^*\Plueck{I}=\Plueck{I^\compl}\).
It follows from Proposition~\ref{p:op}\ref{p:op-posit} that this isomorphism restricts to isomorphisms \(\op\colon\posvar_{\posit}\isoto\posvar_{\posit^\op}\) and \(\op\colon\openposvar_{\posit}\isoto\openposvar_{\posit^\op}\), where \(\posit=\posit_D\) and \(\posit^\op=\posit_{D^\op}\).
Here the statement concerning open positroid varieties also depends on Proposition~\ref{p:sourceneck}.

By Proposition~\ref{p:op}\ref{p:op-quiv}, the cluster algebras \(\clustalg{D}\) and \(\clustalg{D^\op}\) differ only by reversing the quivers of their seeds.
In particular, they coincide as subrings of the Laurent polynomial ring in the initial variables \(x_j\) for \(j\in Q_0\) (this vertex set being common to both quivers), and they have the same cluster variables and clusters.

\begin{prop}
\label{p:op-GL}
Let \(D\) be a Postnikov diagram with positroid \(\posit=\posit_D\), and write \(\posit^\op=\posit_{D^\op}\).
Then the target-labelled cluster structure \(\GLtgt\colon\clustalg{D}\isoto\CC[\openposvar_\posit]\) (see Theorem~\ref{t:GalLam}) may be expressed as the composition
\[\begin{tikzcd}
\clustalg{D}\arrow{r}{\mathrm{id}}&\clustalg{D^\op}\arrow{r}{\GLsrc}&\CC[\openposvarcone_{\posit^\op}]\arrow{r}{\op^*}&\CC[\openposvarcone_{\posit}]
\end{tikzcd}\]
of the source-labelled cluster structure on \(\CC[\openposvarcone_{\posit^\op}]\) with the isomorphism \(\op^*\).
\end{prop}
\begin{proof}
To check this, it is enough to compare the values of the isomorphisms on the initial cluster variables of \(\clustalg{D}\).
By definition, we have \(\GLtgt(x_j)=\varPlueck{\tgtlab{j}(D)}\) for each \(j\in Q_0\). On the other hand, viewing \(x_j\) as an initial variable of \(\clustalg{D^\op}\) and applying the isomorphism \(\GLsrc\) (for \(D^\op\)) yields
\[\GLsrc(x_j)=\varPlueck{\srclab{j}(D^\op)}=\varPlueck{\tgtlab{j}(D)^\compl}\]
by Proposition~\ref{p:op}\ref{p:op-labs}.
But then \(\op^*\varPlueck{\tgtlab{j}(D)^\compl}=\varPlueck{\tgtlab{j}(D)}\), and so the two isomorphisms coincide.
\end{proof}

\section{Quasi-equivalences and quasi-coincidence}
\label{s:quasi-equivalence}

The relationship between the source and target-labelled cluster structures on the coordinate ring \(\CC[\openposvarcone_\posit]\) is phrased in terms of quasi-equivalences of cluster algebras in the sense of Fraser \cite{Fraser-Quasi}.
In this section, we will recall the relevant definitions in the context of cluster algebras of geometric type defined by quivers.

We first introduce shorthand notation for cluster monomials.
If \(Q=Q(s)\) is the quiver of a seed in a cluster algebra \(\clust{A}\), we may write the cluster (Laurent) monomials of \(s\) as
\[x^v=\prod_{j\in Q_0}x_j^{v_j}\]
for \(v\in\ZZ^{Q_0}\), where \(x_j\) is the cluster variable associated to vertex \(j\).
If \(x^v\) is a Laurent monomial in frozen variables, then the support of \(v\) is contained in the set \(F_0=F_0(\clust{A})\) of frozen vertices, common to all seeds, and we will typically view \(v\) as an element of the smaller lattice \(\ZZ^{F_0}\).
If \(\eta\colon\clust{A}\isoto R\) is a cluster structure on a commutative \(\KK\)-algebra \(R\), we abbreviate \(\eta_j=\eta(x_j)\) and \(\eta^v=\eta(x^v)\) for the cluster variables and monomials in \(R\) of some fixed seed \(s\).

\begin{defn}
Let \(\clust{A}\) be a cluster algebra with invertible frozen variables.
We denote by \[\froz(\clust{A})=\{x^v:v\in\ZZ^{F_0}\}\]
the group of Laurent monomials in the frozen variables. If \(\eta\colon\clust{A}\isoto R\) is a cluster structure on \(R\), we write \(\froz(R)=\eta(\froz(\clust{A}))\).
If \(\eta\colon\clust{A}\isoto R\) and \(\nu\colon\clust{B}\isoto S\) are two cluster structures and \(f\colon R\to S\) is an algebra homomorphism, we say that \(f\) is a \emph{cluster isomorphism} if \(\nu^{-1}\circ f\circ\eta\colon\clust{A}\to\clust{B}\) is a strong isomorphism of cluster algebras in the sense of \cite{FomZel-CA2}.
\end{defn}

\begin{defn}
Let \(\clust{A}\) be the cluster algebra generated by an initial seed \(s\) with ice quiver \((Q(s),F(s))\).
Then \(\stab{\clust{A}}\) is the cluster algebra (without frozen variables) generated by the initial seed \(\stab{s}\) with quiver \(\stab{Q}(s)\), the full subquiver of \(Q(s)\) on the mutable vertices.
\end{defn}

We can also obtain \(\stab{\clust{A}}\) from \(\clust{A}\) by setting all frozen variables equal to \(1\); in particular, it does not depend on the choice of initial seed.
Indeed, any seed \(s\) of \(\clust{A}\), with quiver \(Q(s)\), determines a seed \(\underline{s}\) of \(\stab{\clust{A}}\) by setting the frozen variables to \(1\) and taking \(Q(\stab{s})=\stab{Q}(s)\).
From this point of view, we obtain a quotient map \(\pi_{\clust{A}}\colon\clust{A}\to\stab{\clust{A}}\), the kernel of which is the two-sided ideal generated by \(\{x_j-1:j\in F_0\}\) (cf.~\cite[Lem.~A.1]{KelWu}).
Similarly, for a cluster structure \(\eta\colon\clust{A}\isoto R\), there is a quotient map \(\pi_R\colon R\to\stab{R}\) with kernel generated by \(\{\eta_j-1:j\in F_0\}\).
The following statements are then immediate.

\begin{prop}
\label{p:stab-map}
Let \(\eta\colon\clust{A}\isoto R\) and \(\nu\colon\clust{B}\isoto S\) be cluster structures, and let \(f\colon R\to S\) be a (unital) algebra homomorphism such that \(f(\froz(R))\subset\froz(S)\).
Then there is a unique algebra homomorphism \(\stab{f}\colon\stab{R}\to\stab{S}\) such that the diagram
\[\begin{tikzcd}
R\arrow{r}{f}\arrow{d}[swap]{\pi_{R}}&S\arrow{d}{\pi_{S}}\\
\stab{R}\arrow{r}{\stab{f}}&\stab{S}
\end{tikzcd}\]
commutes.
In particular, there is an induced cluster structure \(\stab{\eta}\colon\stab{\clust{A}}\isoto\stab{R}\) such that
\[\begin{tikzcd}
\clust{A}\arrow{r}{\eta}\arrow{d}[swap]{\pi_{\clust{A}}}&R\arrow{d}{\pi_{R}}\\
\stab{\clust{A}}\arrow{r}{\stab{\eta}}&\stab{R}
\end{tikzcd}\]
commutes.
\end{prop}

Let \(\eta\colon \clust{A}\isoto R\) be a cluster structure and let \(s\) be a seed of \(\clust{A}\).
We define the exchange matrix \(B=B(s)\) of the quiver \(Q=Q(s)\) to have \((i,j)\)-th entry
\[b_{ij}=|\{(j\to i)\in Q_1\}|-|\{(i\to j)\in Q_1\}|\]
for \(i\in Q_0\) and \(j\in \stab{Q}_0\), and then for each \(j\in \stab{Q}_0\) define
\[\yhat_j=\yhat_j(s)=\prod_{i\in Q_0}\eta_i^{b_{ij}}\in\Frac(R),\]
where \(\eta_i=\eta(x_i(s))\) is the cluster variable of \(s\) attached to vertex \(i\) (cf.~\cite[Eq.~3.7]{FomZel-CA4}).
We take the fraction field \(\Frac(R)\) as the codomain here since the cluster variables \(\eta_i\) may not be invertible in \(R\).

In the situation of Proposition~\ref{p:stab-map}, assume that \(\stab{f}\) is a cluster isomorphism.
In this case, there is an induced bijection \(s\mapsto f(s)\) between seeds of \(R\) and seeds of \(S\), and a further induced bijection \(f\colon \stab{Q}_0(s)\isoto\stab{Q}_0(f(s))\) between the sets of mutable vertices of the quivers.
Moreover, this bijection of quiver vertices may be extended (non-uniquely, if \(\stab{Q}(s)\) has parallel arrows) to a quiver isomorphism \(\stab{Q}\isoto\stab{Q}(s)\).
Even if \(f\) is the identity map, as in our main application, the induced bijections on seeds and vertex sets are typically not; indeed, these sets are usually not equal.

In general, the map \(f\) also extends uniquely to a function \(f\colon\Frac(R)\to\Frac(S)\).
All of these abuses of notation appear in part \ref{d:qcm-yhat} of the following definition.

\begin{defn}[{\cite[Def.~3.1]{Fraser-Quasi}, see also \cite[\S2.6]{FSB}}]
\label{d:qcm}
Let \(\eta\colon\clust{A}\isoto R\) and \(\nu\colon\clust{B}\isoto S\) be cluster structures, in which frozen variables are invertible.
Then an algebra homomorphism \(f\colon R\to S\) is a \emph{quasi-cluster morphism} from \(\eta\) to \(\nu\) if
\begin{enumerate}
\item\label{d:qcm-froz} \(f(\froz(R))\subset\froz(S)\) and for each non-frozen cluster variable \(x\in R\), there is a non-frozen cluster variable \(x'\in S\) and \(p\in\froz(S)\) such that \(f(x)=x'p\),
\item\label{d:qcm-stable} the map \(\stab{f}\) as in Proposition~\ref{p:stab-map} is a cluster isomorphism with respect to the cluster structures \(\stab{\eta}\) and \(\stab{\nu}\), and
\item\label{d:qcm-yhat} for any seed \(s\) of \(\clust{A}\) and any \(j\in\stab{Q}_0(s)\), we have \(f(\yhat_j(s))=\yhat_{f(j)}(f(s))\).
\end{enumerate}
\end{defn}

\begin{rem}
\label{r:seed-by-seed}
As in \cite[Prop.~3.2]{Fraser-Quasi}, the condition in Definition~\ref{d:qcm}\ref{d:qcm-yhat} is stable under mutations, and so it suffices to check it on a single seed \(s\) of \(\clust{A}\).
Similarly, if \(\underline{f}\) is an algebra isomorphism and there exists a seed \(s\) of \(R\) such that \(f(s)\) is a seed of \(S\) and \(f\colon\stab{Q}_0(s)\isoto\stab{Q}_0(f(s))\) may be extended to a quiver isomorphism, then this property in fact holds for all seeds of \(R\), and so \(f\) is a cluster isomorphism.

It follows directly from the definition that if \(f\) is a quasi-cluster morphism from \(\eta\) to \(\nu\) and also an isomorphism of algebras, then \(f^{-1}\) is a quasi-cluster morphism from \(\nu\) to \(\eta\).
\end{rem}

\begin{defn}
\label{d:quasi-coincide}
Let \(\eta\colon\clust{A}\isoto R\) and \(\nu\colon\clust{B}\isoto R\) be cluster algebra structures on a single algebra \(R\).
We say that these cluster structures \emph{quasi-coincide} if the identity map on \(R\) is a quasi-cluster morphism from \(\eta\) to \(\nu\) (and hence also from \(\nu\) to \(\eta\)).
\end{defn}

Having established this language, the quasi-coincidence conjecture which we will ultimately prove (see Theorem~\ref{t:main-thm}) may be stated as follows.

\begin{conj}[{\cite[Rem.~4.7]{MulSpe-Twist}, \cite[Conj.~1.1]{FSB}}]
\label{conj:qc-conj}
For any Postnikov diagram \(D\) with positroid \(\posit=\posit_D\), the cluster structures \(\GLsrc\colon\clustalg{D}\isoto\CC[\openposvarcone_\posit]\) and \(\GLtgt\colon\clustalg{D}\isoto\CC[\openposvarcone_\posit]\) quasi-coincide.
\end{conj}

\begin{rem}
Quasi-coincidence of \(\eta\colon\clust{A}\isoto R\) and \(\nu\colon\clust{B}\isoto R\) does not imply that \(\clust{A}\) and \(\clust{B}\) are strongly isomorphic cluster algebras; while \(\stab{\clust{A}}\) and \(\stab{\clust{B}}\) must be, the systems of frozen variables (and even their number \cite[\S7]{Fraser-GrassBraid}) may differ.
Indeed, Fraser and Sherman-Bennett \cite[Thm.~4.21]{FSB} have exhibited yet more cluster structures on \(\CC[\openposvarcone_\posit]\), some of which come from cluster algebras not isomorphic to \(\clustalg{D}\), and they prove in some cases \cite[Cor.~6.8]{FSB} (and conjecture in general \cite[Conj.~1.3]{FSB}) that these also quasi-coincide with the source-labelled and target-labelled structures.
At present, our methods for proving Conjecture~\ref{conj:qc-conj} depend implicitly on the fact that both \(\GLsrc\) and \(\GLtgt\) have domain \(\clustalg{D}\) (see Proposition~\ref{p:homot-lift}, for example), and so do not apply to Fraser--Sherman-Bennett's stronger conjecture.
\end{rem}

\section{Reduction to connected positroids}
\label{s:connected}

In this section, we show how to reduce Conjecture~\ref{conj:qc-conj} to the case of connected positroids, for which we have access to the categorical tools to be discussed in the next section.
The intuition is relatively straightforward---given a disconnected positroid \(\posit=\posit_1\times\posit_2\), there is a natural isomorphism \(\openposvar_\posit\iso\openposvar_{\posit_1}\times\openposvar_{\posit_2}\), and we may view a Plücker coordinate on \(\openposvar_\posit\), up to sign, as a product of a Plücker coordinate on \(\openposvar_{\posit_1}\) with one on \(\openposvar_{\posit_2}\).
This extends to a relationship between the cluster structures, allowing us to deduce quasi-coincidence of the source-labelled and target-labelled cluster structures on \(\CC[\openposvarcone_\posit]\) from their quasi-coincidence on \(\CC[\openposvarcone_{\posit_1}]\) and \(\CC[\openposvarcone_{\posit_2}]\).
Unfortunately, making this precise requires a somewhat technical detour with some heavy notation, not least because of the need to keep track of the signs.

Let \(D\) be a Postnikov diagram, and let \(\gamma\) be an arc in the disc with endpoints on the boundary and which is disjoint from \(D\).
Cutting along \(\gamma\), we obtain a Postnikov diagram \(D_1\) in a disc with marked points \(S_1\subseteq\ZZ_n\) on the boundary, and a second Postnikov diagram \(D_2\) in a disc with marked points \(S_2\subseteq\ZZ_n\) on the boundary.
Moreover, \(S_1\) and \(S_2\) are cyclic intervals partitioning \(\ZZ_n\), and we fix their labelling so that \(1\in S_1\).
The results in this section are only of interest in the case that \(S_1\) and \(S_2\) are both non-empty, which we will assume throughout.

For the remainder of the section, we minimise double subscripts by writing \(\Gamma_{D_i}=\Gamma_i\), \(\posit_{D_i}=\posit_i\), and so on. For a similar reason, we also write \(\epsilon(n)=(-1)^n\) for \(n\in\ZZ\).

\begin{prop}
\label{p:posit-decomp}
If \(D_i\) has type \((k_i,|S_i|)\) and \(D\) has type \((k,n)\), then \(k=k_1+k_2\) and the map \(\disunion\colon\binom{S_1}{k_1}\times\binom{S_2}{k_2}\to\binom{n}{k}\) with \((I_1,I_2)\mapsto I_1\disunion I_2\), restricts to a bijection \(\disunion\colon\posit_1\times\posit_2\isoto\posit_D\).
\end{prop}
\begin{proof}
Let \(\mu\) be a perfect matching of \(\Gamma_D\), so that by Corollary~\ref{c:bdry-card} we have \(|\bdry\mu|=k\).
Let \(\mu_i\) be the intersection of \(\mu\) with the edges of the connected component \(\Gamma_i\) of \(\Gamma_D\).
Directly from the definition, we see that \(\mu_i\) is a perfect matching of \(\Gamma_{D_i}\), hence \(|\bdry\mu_i|=k_i\) by Corollary~\ref{c:bdry-card} again, and that \(\bdry\mu=\bdry\mu_1\disunion\bdry\mu_2\).
It follows that \(k=k_1+k_2\).

Similarly, if \(\mu_i\) is a perfect matching of \(\Gamma_i\), then \(\mu=\mu_1\disunion\mu_2\) is a perfect matching of \(\Gamma_D\), and \(\bdry\mu=\bdry\mu_1\disunion\bdry\mu_2\) as above.
This shows that \(\disunion\) restricts to a map \(\posit_1\times\posit_2\to\posit_D\) as claimed.

The map \(\disunion\) is injective since \(I_1=(I_1\disunion I_2)\cap S_1\), and similarly for \(I_2\). The argument from the first paragraph demonstrates that if \(I\in\posit_D\), then \(I\cap S_i\in\posit_i\). Thus, \(I=(I\cap S_1)\disunion(I\cap S_2)\) is in the image of \(\disunion\colon\posit_1\times\posit_2\to\posit_D\), and so this restricted map is surjective.
\end{proof}

The partition \(\ZZ_n=S_1\disunion S_2\) determines a direct sum decomposition \(\CC^n=\CC^{S_1}\dsum\CC^{S_2}\), and hence a decomposition
\begin{equation}
\label{eq:bigwedge}
\Wedge{k}\CC^n=\bigdsum_{k_1+k_2=k}\Wedge{k_1}\CC^{S_1}\tensor\Wedge{k_2}\CC^{S_2}.
\end{equation}
Given \(S\subseteq\ZZ_n\) and \(0<k<|S|\), we write \(\Grass{k}{S}\) for the Grassmannian of \(k\)-dimensional subspaces of \(\CC^S\).
The following geometric fact appears to be well known (see e.g.\ \cite{HavZan}).

\begin{prop}
\label{p:Grass-product}
Given a partition \(\ZZ_n=S_1\disunion S_2\) and natural numbers \(k=k_1+k_2\) with \(0<k_i<|S_i|\), consider the map \(\dsum\colon\Grass{k_1}{S_1}\times\Grass{k_2}{S_2}\to\Grass{k}{n}\) given by \((U_1,U_2)\mapsto U_1\oplus U_2\). Then there is a commutative diagram
\[\begin{tikzcd}
\Grass{k}{n}\arrow{r}{\Plueck{}}&\PP(\Wedge{k}\CC^n)\\
\Grass{k_1}{S_1}\times\Grass{k_2}{S_2}\arrow{u}{\dsum}\arrow{r}{\Plueck{}\times\Plueck{}}&\PP(\Wedge{k_1}\CC^{S_1})\times\PP(\Wedge{k_2}\CC^{S_2}),\arrow{u}{\bsigma}
\end{tikzcd}\]
where \(\bsigma\) is obtained by composing the Segre embedding \[\sigma\colon\PP(\Wedge{k_1}\CC^{S_1})\times\PP(\Wedge{k_2}\CC^{S_2})\to\PP(\Wedge{k_1}\CC^{S_1}\otimes\Wedge{k_2}\CC^{S_2})\]
with the projectivisation of the inclusion \(\Wedge{k_1}\CC^{S_1}\otimes\Wedge{k_2}\CC^{S_2}\to\Wedge{k}\CC^n\) from the decomposition \eqref{eq:bigwedge}.
\end{prop}

Recall that to define Plücker embeddings as in Proposition~\ref{p:Grass-product}, we break the cyclic order on \(\ZZ_n=\{1,\dotsc,n\}\) to its usual linear order.
Defining
\begin{align*}
S_1^-&=\{i\in\ZZ_n:\text{\(i<j\) for all \(j\in S_2\)}\},\\
S_1^+&=\{i\in\ZZ_n:\text{\(i>j\) for all \(j\in S_2\)}\},
\end{align*}
the fact that \(S_1\) and \(S_2\) are cyclic intervals means that \(S_1^\pm\) and \(S_2\) are linear intervals, and \(\ZZ_n\) is the concatenation of \(S_1^-\), followed by \(S_2\), followed by \(S_1^+\), recalling that we assume \(1\in S_1\).
(In particular, \(S_1^-\) is non-empty, although \(S_1^+\) need not be.)
Then for \(I\in\binom{S_1}{k_1}\) and \(J\in\binom{S_2}{k_2}\), we have
\[\bigwedge_{k\in I\disunion J}v_k=\epsilon(k_2\cdot|I\cap S_1^+|)\bigl({\bigwedge}_{i\in I}v_i\bigr)\wedge\bigl({\bigwedge}_{j\in J}v_j\bigr),\]
where \(v_k\) denotes the \(k\)-th standard basis vector of \(\CC^n\) and the wedge products over \(I\), \(J\) and \(I\disunion J\) are taken in the linear order on these subsets of \(\ZZ_n\).
Recall that we write \(\epsilon(n)=(-1)^n\).

\begin{cor}
\label{c:Plueck-prod}
For \((U_1,U_2)\in\Grass{k_1}{S_1}\times\Grass{k_2}{S_2}\) and \(I\in\binom{n}{k}\), we have
\[\Plueck{I}(U_1\oplus U_2)=
\begin{cases}
\epsilon(k_2\cdot|I\cap S_1^+|)\Plueck{I\cap S_1}(U_1)\Plueck{I\cap S_2}(U_2),&|I\cap S_i|=k_i,\\
0,&\text{otherwise.}
\end{cases}\]
In particular, the map \(\oplus\) embeds \(\Grass{k_1}{S_1}\times\Grass{k_2}{S_2}\) into \(\Grass{k}{n}\) as the closed positroid variety \(\posvar_\posit\) for \(\posit=\binom{S_1}{k_1}\times\binom{S_2}{k_2}\).
\end{cor}
\begin{proof}
Using Proposition~\ref{p:Grass-product}, the value of \(\Delta_I(U_1\oplus U_2)\) is given by the \(I\)-th coordinate of \(\bsigma(\Delta(U_1),\Delta(U_2))\), which is as claimed by the definition of the Segre embedding \(\sigma\).
\end{proof}

Alternatively, by representing points in \(\Grasscone{k}{n}\) as \(k\times n\) matrices, with Plücker coordinates computed as minors, one may prove Corollary~\ref{c:Plueck-prod} directly and then deduce Proposition~\ref{p:Grass-product} as a consequence.

\begin{prop}
\label{p:posvar-decomp}
In the setting of Proposition~\ref{p:posit-decomp}, the map \(\dsum\) from Proposition~\ref{p:Grass-product} restricts to isomorphisms
\begin{align*}
\dsum\colon\posvar(\posit_1)\times\posvar(\posit_2)&\isoto\posvar(\posit_D),\\
\dsum\colon\openposvar(\posit_1)\times\openposvar(\posit_2)&\isoto\openposvar(\posit_D).
\end{align*}
\end{prop}
\begin{proof}
By Proposition~\ref{p:posit-decomp}, for \(I\in\binom{n}{k}\) we have \(I\in\posit_D\) if and only if \(I\cap S_1\in\posit_1\) and \(I\cap S_2\in\posit_2\).
In particular, this means that \(|I\cap S_1|=k_1\) and \(|I\cap S_2|=k_2\), so \(\posvar(\posit_D)\) is contained in \(\posvar_{\posit}\) for \(\posit=\binom{S_1}{k_1}\times\binom{S_2}{k_2}\).
Thus, if \(U\in\posvar(\posit_D)\), then \(U=U_1\dsum U_2\) for some \(U_1\in\Grass{k_1}{S_1}\) and \(U_2\in\Grass{k_2}{S_2}\).

Now if \(I\notin\posit_D\), then either \(I\cap S_1\notin\posit_1\) or \(I\cap S_2\notin\posit_2\), so by Corollary~\ref{c:Plueck-prod} we have \(\Plueck{I}(U_1\dsum U_2)=0\) for any \(U_1\in\posvar(\posit_1)\) and \(U_2\in\posvar(\posit_2)\).
Thus, the inclusion \(\dsum\) takes \(\posvar(\posit_1)\times\posvar(\posit_2)\) into \(\posvar(\posit_D)\) as claimed.

On the other hand, assume \(U=U_1\dsum U_2\in\posvar(\posit_D)\) and choose \(I\in\binom{S_1}{k_1}\) with \(I\notin\posit_1\). Then \(I\disunion J\notin\posit_D\) for any \(J\in\binom{S_2}{k_2}\). Since \(k_2\ne0\), we may choose \(J\in\binom{S_2}{k_2}\) such that \(\Delta_J(U_2)\ne0\). It then follows from Corollary~\ref{c:Plueck-prod} that
\[\Plueck{I}(U_1)\Plueck{J}(U_2)=\pm\Plueck{I\disunion J}(U)=0,\]
and so \(\Plueck{I}(U_1)=0\). Thus, \(U_1\in\posvar(\posit_1)\). The symmetric argument shows that \(U_2\in\posvar(\posit_2)\), and it follows that \(\dsum\) takes \(\posvar(\posit_1)\times\posvar(\posit_2)\) onto \(\posvar(\posit_D)\).

By the definition of the labelling rules for a Postnikov diagram, if \(I\in\tgtneck_D\), then \(I\cap S_1\in\tgtneck_1\) and \(I\cap S_2\in\tgtneck_2\).
Moreover, if \(I\in\tgtneck_1\), then there is \(J\in\tgtneck_2\) such that \(I\disunion J\in\tgtneck_D\); indeed, there is a unique such \(J\), given by the target label of the boundary face of \(D_2\) coming from the face of \(D\) that we cut through.
The analogous statement holds swapping the roles of \(\tgtneck_1\) and \(\tgtneck_2\).
Using Corollary~\ref{c:Plueck-prod} to see that \(\Plueck{I}(U_1\oplus U_2)\ne0\) if and only if \(\Plueck{I\cap S_1}(U_1)\ne0\) and \(\Plueck{I\cap S_2}(U_2)\ne0\), the claim on open positroids follows.
\end{proof}

Given \(\ZZ\)-graded commutative \(\CC\)-algebras \(R\) and \(S\), we denote by
\[R\Segprod S=\bigdsum_{d\in\ZZ}R_d\otimes_\CC S_d\]
their Segre product, with the property that \(\Proj(R\Segprod S)=\Proj(R)\times\Proj(S)\).
We consider \(R\Segprod S\) as a \(\ZZ\)-graded algebra in the way suggested by this direct sum decomposition: a homogeneous element of degree \(d\) is a sum of pure tensors \(x\otimes y\) for which \(x\in R_d\) and \(y\in S_d\) each have degree \(d\) in their respective algebras.
In the next result, we grade the homogeneous coordinate ring of an open positroid variety in the usual way, by putting the Plücker coordinates in degree \(1\).

\begin{cor}
\label{c:Segre-product}
The map \(\partial\colon\CC[\openposvarcone(\posit_D)]\to\CC[\openposvarcone(\posit_1)]\Segprod\CC[\openposvarcone(\posit_2)]\), defined on Plücker coordinates by \(\partial\colon\Plueck{I}\mapsto\epsilon(k_2\cdot|I\cap S_1^+|)\Plueck{I\cap S_1}\otimes\Plueck{I\cap S_2}\), is an isomorphism of graded \(\CC\)-algebras.
\end{cor}
\begin{proof}
This is the second isomorphism from Proposition~\ref{p:posvar-decomp}, reinterpreted as a map of homogeneous coordinate rings.
\end{proof}

Now we relate cluster algebra structures on \(\CC[\openposvarcone(\posit_1)]\) and \(\CC[\openposvarcone(\posit_2)]\) with those on \(\CC[\openposvarcone(\posit)]\).
The next proposition is illustrated by Figure~\ref{f:Postcut}.

\begin{prop}
\label{p:quivsplit}
Let \(Q_0^i\) be the set of vertices of \(Q=Q_D\) corresponding to faces of \(D_i\), for \(i=1,2\). These sets satisfy
\begin{enumerate}
\item\label{p:quivsplit-union} \(Q_0^1\cup Q_0^2=Q_0\),
\item\label{p:quivsplit-intersection} \(Q_0^1\cap Q_0^2=\{*\}\) is a single frozen vertex,
\item\label{p:quivsplit-subquiv} the full subquiver of \(Q\) on \(Q_0^i\) is the quiver associated to \(D_i\), and
\item\label{p:quivsplit-noarrows} if \(\alpha\colon j\to j'\) is an arrow, then either \(j,j'\in Q_0^1\) or \(j,j'\in Q_0^2\).
\end{enumerate}
In particular, \(Q_0^1\) and \(Q_0^2\) have no mutable vertices in common, and there are no arrows between a mutable vertex of \(Q_0^1\) and a mutable vertex of \(Q_0^2\).
\end{prop}
\begin{proof}
Recall that \(D_1\) and \(D_2\) are obtained from \(D\) by cutting through a boundary region along a curve \(\gamma\), and we take \(*\) to be the frozen vertex of \(Q\) corresponding to this region.
The remaining vertices of \(Q\) are on one side of \(\gamma\) or the other, and \(Q_0^1\) consists of \(*\) together with those vertices on the side of \(\gamma\) containing the boundary marked point \(1\), while \(Q_0^2\) consists of \(*\) together with the vertices on the other side of \(\gamma\).
Properties \ref{p:quivsplit-union}--\ref{p:quivsplit-subquiv} are then immediate from the construction.

Moreover, removing the vertex \(*\) and all adjacent arrows disconnects the quiver, corresponding to the fact that \(\gamma\) cuts the disc into two components.
Thus, every (directed or undirected) path from \(Q_0^1\) to \(Q_0^2\) must pass through \(*\), and \ref{p:quivsplit-noarrows} follows.
\end{proof}

\begin{eg}
\label{eg:disconnected}
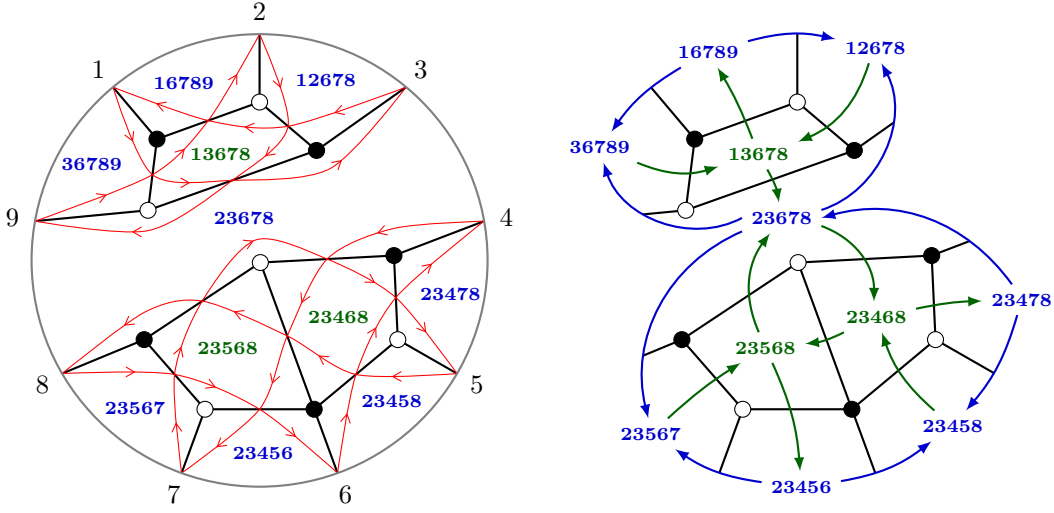
\begin{figure}
\begin{minipage}{0.95\textwidth}
\begin{tikzpicture}[scale=3,baseline=(bb.base)]
\path (0,0) node (bb) {};

\foreach \n/\m/\a in {9/1/0, 8/2/0, 7/3/0, 6/4/0, 5/5/0, 4/6/0, 3/7/0, 2/8/0, 1/9/0}
{ \coordinate (b\m) at (\bstart+\ninth*\n+\a:1.0);
  \draw (\bstart+\ninth*\n+\a:1.1) node {\(\m\)}; }

\foreach \n/\m/\x/\y in {10/1/0/0, 11/2/0/0, 12/3/-0.2/-0.05, 13/4/-0.1/-0.1, 14/5/0/0, 15/6/0/0, 16/7/0/0, 17/8/0.1/0, 18/9/0.2/0.1}
  {\coordinate (b\n) at ($0.7*(b\m)+(\x,\y)$);}

\coordinate (b19) at ($0.01*(b6)$);

\foreach \h/\t in {1/10, 2/11, 3/12, 4/13, 5/14, 6/15, 7/16, 8/17, 9/18, 18/10, 10/11, 11/12, 12/18, 13/14, 14/15, 15/16, 16/17, 17/19, 19/15, 19/13}
{ \draw [bipedge] (b\h)--(b\t); }

\foreach \n in {11,14,16,18,19} 
  {\draw [\graphcolor] (b\n) circle(\dotrad) [fill=white];} \foreach \n in {10,12,13,15,17}  
  {\draw [\graphcolor] (b\n) circle(\dotrad) [fill=\graphcolor];} 

\foreach \e/\f/\t in {10/18/0.5, 10/11/0.5, 11/12/0.5, 12/18/0.5, 13/14/0.5, 14/15/0.5, 15/16/0.5, 16/17/0.5, 17/19/0.5, 15/19/0.5, 13/19/0.5}
{\coordinate (a\e-\f) at ($(b\e) ! \t ! (b\f)$); }

\draw [strand] plot[smooth]
coordinates {(b1) (a10-18) (a12-18) ($(b12)+(0.06,-0.08)$) (b3)}
[postaction=decorate, decoration={markings,
 mark= at position 0.15 with \strarrow,
 mark= at position 0.35 with \strarrow, 
 mark= at position 0.75 with \strarrow }];
 
 \draw [strand] plot[smooth]
 coordinates {(b2) (a11-12) (a12-18) ($(b18)+(0,-0.09)$) (b9)}
 [postaction=decorate, decoration={markings,
  mark= at position 0.15 with \strarrow,
  mark= at position 0.35 with \strarrow, 
  mark= at position 0.75 with \strarrow }];
  
 \draw [strand] plot[smooth]
 coordinates {(b3) (a11-12) (a10-11) (b1)}
 [postaction=decorate, decoration={markings,
  mark= at position 0.25 with \strarrow,
  mark= at position 0.55 with \strarrow, 
  mark= at position 0.85 with \strarrow }];
  
 \draw [strand] plot[smooth]
 coordinates {(b4) ($(b13)+(0,0.12)$) (a13-19) (a15-19) (a15-16) (b7)}
 [postaction=decorate, decoration={markings,
  mark= at position 0.25 with \strarrow,
  mark= at position 0.5 with \strarrow, 
  mark= at position 0.7 with \strarrow,
  mark= at position 0.9 with \strarrow }];
  
 \draw [strand] plot[smooth]
 coordinates {(b5) (a14-15) (a15-19) (a17-19) ($(b17)+(0,0.12)$) (b8)}
 [postaction=decorate, decoration={markings,
  mark= at position 0.15 with \strarrow,
  mark= at position 0.33 with \strarrow, 
  mark= at position 0.55 with \strarrow,
  mark= at position 0.82 with \strarrow }];
  
 \draw [strand] plot[smooth]
 coordinates {(b6) (a14-15) (a13-14) (b4)}
 [postaction=decorate, decoration={markings,
  mark= at position 0.2 with \strarrow,
  mark= at position 0.5 with \strarrow,
  mark= at position 0.78 with \strarrow }];
  
 \draw [strand] plot[smooth]
 coordinates {(b7) (a16-17) (a17-19) ($(b19)+(-0.02,0.1)$) (a13-19) (a13-14) (b5)}
 [postaction=decorate, decoration={markings,
  mark= at position 0.1 with \strarrow,
  mark= at position 0.25 with \strarrow, 
  mark= at position 0.505 with \strarrow,
  mark= at position 0.76 with \strarrow,
  mark= at position 0.9 with \strarrow }];
  
 \draw [strand] plot[smooth]
 coordinates {(b8) (a16-17) (a15-16) (b6)}
 [postaction=decorate, decoration={markings,
  mark= at position 0.24 with \strarrow,
  mark= at position 0.51 with \strarrow,
  mark= at position 0.82 with \strarrow }];
  
 \draw [strand] plot[smooth]
 coordinates {(b9) (a10-18) (a10-11) (b2)}
 [postaction=decorate, decoration={markings,
  mark= at position 0.25 with \strarrow,
  mark= at position 0.55 with \strarrow, 
  mark= at position 0.85 with \strarrow }];
 
 \foreach \n/\l/\a in {1/16789/3, 2/12678/0, 4/23478/0, 5/23458/3, 6/23456/1, 7/23567/0, 9/36789/0}
 { \draw [\frozcolor] (\bstart+\ninth/2-\ninth*\n+\a:0.85) node (q\l) {\scriptsize \(\mathbf{\l}\)}; }
 \draw [\frozcolor] ($-0.2*(b6)$) node (q23678) {\scriptsize \(\mathbf{23678}\)};
 \draw [\quivcolor] ((\bstart-\ninth/2:0.5) node (q13678) {\scriptsize \(\mathbf{13678}\)};
 \draw [\quivcolor] ($0.4*(b5)-(0,0.05)$) node (q23468) {\scriptsize \(\mathbf{23468}\)};
 \draw [\quivcolor] ($0.41*(b7)$) node (q23568) {\scriptsize \(\mathbf{23568}\)};
 
 \draw [boundary] (0,0) circle(1.0);
\end{tikzpicture}
\hfill
\begin{tikzpicture}[scale=3,baseline=(bb.base),tips=proper]
\path (0,0) node (bb) {};

\foreach \n/\m/\a in {9/1/0, 8/2/0, 7/3/0, 6/4/0, 5/5/0, 4/6/0, 3/7/0, 2/8/0, 1/9/0}
{ \coordinate (b\m) at (\bstart+\ninth*\n+\a:1.0);}

\foreach \n/\m/\x/\y in {10/1/0/0, 11/2/0/0, 12/3/-0.2/-0.05, 13/4/-0.1/-0.1, 14/5/0/0, 15/6/0/0, 16/7/0/0, 17/8/0.1/0, 18/9/0.2/0.1}
  {\coordinate (b\n) at ($0.7*(b\m)+(\x,\y)$);}

\coordinate (b19) at ($0.01*(b6)$);

\foreach \e/\f/\t in {10/18/0.5, 10/11/0.5, 11/12/0.5, 12/18/0.5, 13/14/0.5, 14/15/0.5, 15/16/0.5, 16/17/0.5, 17/19/0.5, 15/19/0.5, 13/19/0.5}
{\coordinate (a\e-\f) at ($(b\e) ! \t ! (b\f)$); }

\foreach \e/\f/\t in {3/12/0.56, 4/13/0.58, 8/17/0.53, 9/18/0.62}
{\coordinate (b'\e) at ($(b\e) ! \t ! (b\f)$);}

 \foreach \n/\l/\a in {1/16789/3, 2/12678/0, 4/23478/0, 5/23458/3, 6/23456/1, 7/23567/0, 9/36789/0}
 { \draw [\frozcolor] (\bstart+\ninth/2-\ninth*\n+\a:1) node (q\l) {\scriptsize \(\mathbf{\l}\)}; }
 \draw [\frozcolor] ($-0.2*(b6)$) node (q23678) {\scriptsize \(\mathbf{23678}\)};
 \draw [\quivcolor] ((\bstart-\ninth/2:0.5) node (q13678) {\scriptsize \(\mathbf{13678}\)};
 \draw [\quivcolor] ($0.4*(b5)-(0,0.05)$) node (q23468) {\scriptsize \(\mathbf{23468}\)};
 \draw [\quivcolor] ($0.41*(b7)$) node (q23568) {\scriptsize \(\mathbf{23568}\)};

\foreach \h/\t in {1/10, 2/11, 5/14, 6/15, 7/16, 18/10, 10/11, 11/12, 12/18, 13/14, 14/15, 15/16, 16/17, 17/19, 19/15, 19/13}
{ \draw [bipedge] (b\h)--(b\t); }
\foreach \h/\t in {3/12, 4/13, 8/17, 9/18}
{ \draw [bipedge] (b'\h)--(b\t); }

\foreach \n in {11,14,16,18,19} 
  {\draw [\graphcolor] (b\n) circle(\dotrad) [fill=white];} \foreach \n in {10,12,13,15,17}  
  {\draw [\graphcolor] (b\n) circle(\dotrad) [fill=\graphcolor];} 

 \foreach \t/\h/\a in {36789/13678/-20, 13678/16789/-5, 12678/13678/25, 13678/23678/10, 23468/23568/10, 23568/23678/30,
 23568/23456/10, 23567/23568/5, 23468/23478/5, 23458/23468/15, 23678/23468/35}
 { \draw [quivarrow]  (q\t) edge [bend left=\a] (q\h); }
 
 \foreach \t/\h/\a in {16789/36789/-12, 16789/12678/14, 23478/23458/12, 23456/23458/-14, 23456/23567/13, 23678.west/36789/50, 23678/12678/-50, 23678/23567/-40, 23478/23678.east/-35}
 { \draw [frozarrow]  (q\t) edge [bend left=\a] (q\h); }
\end{tikzpicture}
\end{minipage}
\caption{A disconnected Postnikov diagram and dual plabic graph (left), with the associated quiver (right) and source-labelling.}
\label{f:Postcut}
\end{figure}
Figure~\ref{f:Postcut} shows a disconnected Postnikov diagram \(D\), of type \((5,9)\), and its associated quiver \(Q_D\).
In this case, we have \(S_1^-=\{1,2,3\}\), \(S_2=\{4,5,6,7,8\}\) and \(S_1^+=\{9\}\).
The vertices in the upper component, with boundary marked points \(S_1\), form \(Q_0^1\), while those in the lower component form \(Q_0^2\), with the vertex \(*\) labelled by \(23678\) in both subsets.
The quiver \(Q_D\) is obtained by gluing together at a frozen vertex the quivers of the two component Postnikov diagrams \(D_1\) and \(D_2\); cf.\ Proposition~\ref{p:quivsplit}.
In this case, the two component Postnikov diagrams are both uniform, of types \((2,4)\) and \((3,5)\) respectively.

Observe that \(\{6,7,8\}=\srclab{*}\cap S_2\) appears in the label of every vertex from \(Q_0^1\), while \(\{2,3\}=\srclab{*}\cap S_1\) appears in the label of every vertex from \(Q_0^2\).
The same phenomenon occurs for the target labelling.
\end{eg}

\begin{cor}
\label{c:mut-comm}
For \(i\in Q_0^1\) and \(j\in Q_0^2\) mutable vertices, mutation at \(i\) commutes with mutation at \(j\).
\end{cor}
\begin{proof}
This is a direct consequence of Proposition~\ref{p:quivsplit}\ref{p:quivsplit-noarrows}, which implies that there are no arrows between \(i\) and \(j\).
\end{proof}

A consequence of Corollary~\ref{c:mut-comm} is that the local bijections on vertices induced by applying a sequence of mutations to \(Q\) respect the division of these vertices into the sets \(Q_0^1\) and \(Q_0^2\).
Thus, if \(j\) is a vertex of a quiver \(Q'\) mutation equivalent to \(Q\), we can make sense of the statement \(j\in Q_0^1\) or \(j\in Q_0^2\) despite the lack of a preferred bijection between \(Q_0\) and \(Q'_0\).
We use this repeatedly below; for example, it is needed for part \ref{mut-partition} of the next proposition to make sense.

\begin{prop}
\label{p:cvar-partition}
The set \(\cvars\) of cluster variables of \(\clustalg{D}\) has the form \(\cvars=\cvars_1\cup\cvars_2\), where
\begin{enumerate}
\item the initial variable \(x_j\) is in \(\cvars_i\) if and only if \(j\in Q_0^i\), and
\item\label{mut-partition} a non-initial variable lies in \(\cvars_i\) if and only if it can be obtained from the initial seed by a sequence of mutations using only vertices from \(Q_0^i\).
\end{enumerate}
In particular, the set \(\cvars_i\) of cluster variables of \(\clustalg{D}\) is in natural bijection, respecting mutations, with the set of cluster variables of \(\clustalg{D_i}\), and \(\cvars_1\cap\cvars_2=\{x_*\}\) consists of a single frozen variable.
\end{prop}
\begin{proof}
This is a direct consequence of Corollary~\ref{c:mut-comm}; given a sequence of mutations leading to a cluster variable \(x\) at some vertex \(j\in Q_0^i\) of its quiver, one obtains the same cluster variable by performing only the mutations in this sequence from \(Q_0^i\).
In particular this means that no variable other than \(x_*\) lies in \(\cvars_1\cap\cvars_2\), since then it would appear twice in the same seed, violating the fact that the cluster variables in a seed are algebraically independent.
\end{proof}

Given \(x\in\cvars_i\), we write \(x|_i\) for the cluster variable of \(\clustalg{D_i}\) corresponding to \(x\) under the bijection from Proposition~\ref{p:cvar-partition}.
If \(x_j\) is the initial cluster variable of \(\clustalg{D}\) at the vertex \(j\in Q_0^i\), then \(x_j|_i\) is the initial variable of \(\clustalg{D_i}\) attached to the same vertex.
In particular, this means that \(x_j|_i\) is frozen if and only if \(x_j\) is frozen.
With this notation, we may state the following immediate corollary of Proposition~\ref{p:cvar-partition}.

\begin{cor}
\label{c:Segre-comp}
In the notation of Proposition~\ref{p:cvar-partition}, any cluster variable from \(\cvars_1\) is compatible with any cluster variable from \(\cvars_2\).
Moreover, cluster variables \(x,y\in\cvars_i\) are compatible if and only if \(x|_i\) and \(y|_i\) are compatible in \(\clustalg{D_i}\).
\end{cor}

As a result, if we write \(\cmons_i\) for the cluster monomials of \(\clustalg{D}\) obtained as products of compatible cluster variables from \(\cvars_i\), the map \(x\mapsto x|_i\) extends multiplicatively to a bijection between \(\cmons_i\) and the cluster monomials of \(\clustalg{D_i}\).
We extend the notation \(p\mapsto p|_i\) to describe this bijection.

For a fixed (source or target) labelling convention, we write \(I_{j}|_i=I_j\cap S_i\in\binom{S_i}{k_i}\) for the label of \(j\in Q_0^i\) in the diagram \(D_i\).
As in the proof of Proposition~\ref{p:posvar-decomp} (cf.~Example~\ref{eg:disconnected}), it follows from the definition of the labelling rules that the label \(I_j\) of \(j\in Q_0\) in the full Postnikov diagram \(D\) is given by
\begin{equation}
\label{eq:label-decomp}
I_j=\begin{cases}
I_{j}|_1\disunion I_{*}|_2,&j\in Q_0^1,\\
I_{*}|_1\disunion I_{j}|_2,&j\in Q_0^2.
\end{cases}
\end{equation}

Recall from \cite{Grabowski-GCAs} that a grading on a cluster algebra \(\clustalg{}\) is a grading of the underlying \(\CC\)-algebra with the property that every cluster variable is homogeneous.
We now introduce several useful gradings on the cluster algebra \(\clustalg{D}\) associated to a Postnikov diagram, all derived from the standard \(\ZZ^n\)-grading on \(\CC[\Grasscone{k}{n}]\), called the \(\GL_n(\CC)\)-weight grading in \cite{JKS}.
This is given on Plücker coordinates by
\[\degvec{\Plueck{I}}=\sum_{i\in I}\alpha_i\]
for \(\alpha_i\) the \(i\)-th standard basis vector in \(\ZZ^n\).
This descends to a grading on \(\CC[\openposvarcone_\posit]\) for any positroid \(\posit\), since this subvariety is defined by vanishing and non-vanishing conditions on Plücker coordinates, all of which are homogeneous.

\begin{prop}
\label{p:grading}
For \(x\in\clustalg{D}\), define
\[
\degvec^{\src}(x)=\degvec\GLsrc(x),\quad\degvec^{\tgt}(x)=\degvec\GLtgt(x).
\]
That is, \(\degvec^{\src}\) and \(\degvec^{\tgt}\) are induced by pulling back \(\degvec\) along \(\GLsrc\) and \(\GLtgt\) respectively.
Then both \(\degvec^{\src}\) and \(\degvec^{\tgt}\) are gradings of \(\clustalg{D}\) as a cluster algebra.
\end{prop}
\begin{proof}
As in \cite[Prop.~3.2]{Grabowski-GCAs}, to see that \(\degvec^{\src}\) is a cluster algebra grading of \(\clustalg{D}\), it is sufficient to check that each initial variable \(x_j\) is homogeneous, which is immediate since \(\GLsrc\) sends these variables to Plücker coordinates, and that
\[\sum_{u\to j}\degvec^{\src}(x_u)=\sum_{j\to v}\degvec^{\src}(x_v)\]
for each mutable vertex \(j\) of \(Q\).
Translating to labels, this means that we need
\[\sum_{u\to j}|\srclab{u}\cap\{i\}|=\sum_{j\to v}|\srclab{v}\cap\{i\}|\]
for each \(i\in\ZZ_n\).
This fact, well-known to experts (cf.~\cite[Proof of Thm.~5.17]{FSB}), follows by considering the local strand configuration at the face \(j\), and how the labels change when passing through one of the crossings around this face.
For example, when \(j\) is \(4\)-valent, the statement can be read off from \cite[Fig.~17]{Scott-Grass}, and the general case is proved similarly.
\end{proof}

The standard \(\ZZ\)-grading on \(\CC[\openposvarcone_\posit]\), which we have already used above, is that with \(\deg(\Plueck{I})=1\) for any (non-zero) Plücker coordinate \(\Plueck{I}\), and which satisfies
\[\deg(f)=\tfrac{1}{k}(\mathbf{1}\cdot\degvec{f})\]
for any \(\degvec\)-homogeneous element \(f\). Here \(\mathbf{1}\) denotes the all \(1\)s vector.
Pulling back to \(\clustalg{D}\) (with either \(\GLsrc\) or \(\GLtgt\)) gives the grading \(\deg\) of this cluster algebra with \(\deg(x_j)=1\) for all \(j\in Q_0\); the fact that this is a grading either follows from Proposition~\ref{p:grading} and the previous formula, or directly from the fact that every mutable vertex of \(Q\) has the same number of incoming and outgoing arrows.

\begin{defn}
For \(S\subseteq\ZZ_n\), write \(\mathbf{1}_S\) for the indicator function of \(S\).
Define a \(\ZZ\)-grading \(\deg_S\) by
\[\deg_S(f)=\mathbf{1}_S\cdot\degvec(f)=\sum_{i\in S}\degvec(f)_i\]
for  \(f\in\CC[\openposvarcone_\posit]\) a \(\degvec\)-homogeneous element, and write \(\deg_S^{\src}\) and \(\deg_S^{\tgt}\) for the pullbacks of \(\deg_S\) to \(\clustalg{D}\) along \(\GLsrc\) and \(\GLtgt\) respectively.
We abbreviate \(\deg_{+}=\deg_{S_1^+}\), and similarly for \(\deg_+^{\src}\) and \(\deg_+^{\tgt}\).
\end{defn}

It follows from Proposition~\ref{p:grading} (or can be proved directly via an analogous argument) that \(\deg_S^{\src}\) and \(\deg_S^{\tgt}\) are cluster algebra gradings of \(\clustalg{D}\) for any \(S\subseteq\ZZ_n\).
We observe that if \(p\in\cmons_1\), then we have equalities
\[\deg_+^{\src}(p)=\deg_+^{\src}(p|_1),\quad\deg_+^{\tgt}(p)=\deg_+^{\tgt}(p|_1)\]
since these hold for the initial variables, noting that \(\deg_+\Plueck{I}=|I\cap S_1^+|=\deg_+\Plueck{I\cap S_1}\).
We similarly have \(\deg(p)=\deg(p|_i)\) for \(p\in\cmons_i\), \(i=1,2\), and we will use these identities in the proof of Theorem~\ref{t:qco-conn} below.

These gradings provide a useful language for dealing with the signs appearing in Corollary~\ref{c:Plueck-prod}.
Indeed, we may rewrite the map \(\partial\) from Corollary~\ref{c:Segre-product} as
\[\partial(\Plueck{I})=\epsilon(k_2\deg_+(\Plueck{I}))\Plueck{I\cap S_1}\otimes\Plueck{I\cap S_2}.\]

\begin{prop}
\label{p:clust-Segre}
Define a map \(\delta^{\src}\colon\clustalg{D}\to\clustalg{D_1}\Segprod\clustalg{D_2}\) on initial variables by
\[\delta^{\src}(x_j)=
\begin{cases}
\epsilon(k_2\deg_+^{\src}(x_j))x_j|_1\otimes x_*|_2,&j\in Q_0^1,\\
\epsilon(k_2\deg_+^{\src}(x_*))x_*|_1\otimes x_j|_2,&j\in Q_0^2.
\end{cases}\]
Then there is a commutative diagram
\[\begin{tikzcd}[column sep=3.5em]
\clustalg{D}\arrow{r}{\GLsrc}\arrow{d}{\delta^{\src}}&\CC[\openposvarcone(\posit)]\arrow{d}{\partial}\\
\clustalg{D_1}\Segprod\clustalg{D_2}\arrow{r}{\GLsrc_1\Segprod\GLsrc_2}&\CC[\openposvarcone(\posit_1)]\Segprod\CC[\openposvarcone(\posit_2)],
\end{tikzcd}\]
In particular, \(\delta^{\src}\) is an isomorphism (and hence a cluster structure).
The analogous statement holds if we replace source-labelling by target-labelling throughout.
\end{prop}

\begin{proof}
It suffices to show that the diagram commutes beginning with an initial variable \(x_j\).
For \(j\in Q_0^1\), we may compute directly that
\begin{align*}
(\GLsrc_1\Segprod\GLsrc_2)(\delta^{\src}(x_j))
&=\epsilon(k_2\deg_+^{\src}(x_j))\GLsrc_1(x_j|_1)\otimes\GLtgt_2(x_*|_2)\\
&=\epsilon(k_2\deg_+^{\src}(x_j))\varPlueck{\srclab{j}|_1}\otimes\varPlueck{\srclab{*}|_2}.
\end{align*}
On the other hand, since \(\srclab{j}=\srclab{j}|_1\disunion\srclab{*}|_2\) with \(\srclab{j}|_1\subseteq S_1\) and \(\srclab{*}|_2\subseteq S_2\) by construction, we also have
\[\partial\GLsrc(x_j)=\partial(\varPlueck{\srclab{j}})=\epsilon(k_2\deg_+\varPlueck{\srclab{j}})\varPlueck{\srclab{j}|_1}\otimes\varPlueck{\srclab{*}|_2}.\]
We have \(\deg_+^{\src}(x_j)=\deg_+(\varPlueck{\srclab{j}})\) by definition, and so the diagram commutes starting with \(x_j\).
The proof for \(j\in Q_0^2\) is essentially the same, with \(\deg_+^{\src}(x_*)=\deg_+(\varPlueck{\srclab{*}})\) appearing in the expressions for the signs.

Since \(\GLsrc\) and \(\GLsrc_1\Segprod\GLsrc_2\) are isomorphisms by Theorem~\ref{t:GalLam} and \(\partial\) is an isomorphism by Corollary~\ref{c:Segre-product}, it follows that \(\delta^{\src}\) is also an isomorphism.
The proof of the analogous statement for the target-labelling is identical.
\end{proof}

\begin{rem}
Based partly on this example, a more general study of Segre products of graded cluster algebras has been conducted by Grabowski and Hindmarch \cite{GraHin}.
\end{rem}

\begin{eg}
Proposition~\ref{p:clust-Segre} gives two equivalent ways of using the data from Figure~\ref{f:Postcut} to describe a cluster algebra structure on \(\CC\bigl[\openposvarcone\binom{2}{4}\bigr]\Segprod\CC\bigl[\openposvarcone\binom{3}{5}\bigr]\), where each of the positroid varieties involved is the big cell in the relevant Grassmannian, which is identified with the source-labelled cluster structure for \(\openposvar\bigl(\binom{2}{4}\times\binom{3}{5}\bigr)\subseteq\Grass{5}{9}\) using the isomorphism \(\partial\) from Corollary~\ref{c:Segre-product}.
An initial seed for this cluster structure is shown in Figure~\ref{f:prodstruct}; each initial variable is a product of Plücker coordinates up to sign, but some of the signs need to be negative to obtain the correct identification.
\begin{figure}

\begin{tikzpicture}[xscale=4,yscale=3.5,baseline=(bb.base),tips=proper]
\path (0,0) node (bb) {};

\foreach \n/\m/\a in {9/1/0, 8/2/0, 7/3/0, 6/4/0, 5/5/0, 4/6/0, 3/7/0, 2/8/0, 1/9/0}
{ \coordinate (b\m) at (\bstart+\ninth*\n+\a:1.0);}

\foreach \n/\m/\x/\y in {10/1/0/0, 11/2/0/0, 12/3/-0.2/-0.05, 13/4/-0.1/-0.1, 14/5/0/0, 15/6/0/0, 16/7/0/0, 17/8/0.1/0, 18/9/0.2/0.1}
  {\coordinate (b\n) at ($0.7*(b\m)+(\x,\y)$);}

\coordinate (b19) at ($0.01*(b6)$);

\foreach \e/\f/\t in {10/18/0.5, 10/11/0.5, 11/12/0.5, 12/18/0.5, 13/14/0.5, 14/15/0.5, 15/16/0.5, 16/17/0.5, 17/19/0.5, 15/19/0.5, 13/19/0.5}
{\coordinate (a\e-\f) at ($(b\e) ! \t ! (b\f)$); }

\foreach \e/\f/\t in {3/12/0.56, 4/13/0.58, 8/17/0.53, 9/18/0.62}
{\coordinate (b'\e) at ($(b\e) ! \t ! (b\f)$);}

 \foreach \n/\l/\m/\s/\a in {1/19/678/-/3, 2/12/678//0, 4/23/478//0, 5/23/458//3, 6/23/456//1, 7/23/567//0, 9/39/678/-/0}
 { \draw [\frozcolor] (\bstart+\ninth/2-\ninth*\n+\a:1) node (q\l\m) { \(\s\Plueck{\l}\Plueck{\m}\)}; }
 \draw [\frozcolor] ($-0.2*(b6)$) node (q23678) { \(\Plueck{23}\Plueck{678}\)};
 \draw [\quivcolor] ((\bstart-\ninth/2:0.5) node (q13678) { \(\Plueck{13}\Plueck{678}\)};
 \draw [\quivcolor] ($0.4*(b5)-(0,0.05)$) node (q23468) { \(\Plueck{23}\Plueck{468}\)};
 \draw [\quivcolor] ($0.41*(b7)$) node (q23568) { \(\Plueck{23}\Plueck{568}\)};

 \foreach \t/\h/\a in {39678/13678/-20, 13678/19678/-5, 12678/13678/25, 13678/23678/10, 23468/23568/10, 23568/23678/30,
 23568/23456/10, 23567/23568/5, 23468/23478/5, 23458/23468/15, 23678/23468/35}
 { \draw [quivarrow]  (q\t) edge [bend left=\a] (q\h); }
 
 \foreach \t/\h/\a in {19678/39678/-12, 19678/12678/14, 23478/23458/12, 23456/23458/-14, 23456/23567/13, 23678.west/39678/50, 23678/12678/-50, 23678/23567/-40, 23478/23678.east/-35}
 { \draw [frozarrow]  (q\t) edge [bend left=\a] (q\h); }
\end{tikzpicture}
\caption{An initial seed for a cluster structure on \(\CC\bigl[\openposvarcone{\binom{2}{4}}\bigr]\Segprod\CC\bigl[\openposvarcone{\binom{3}{5}}\bigr]\), which identifies with the source-labelled structure on \(\CC\bigl[\openposvar\bigl(\binom{2}{4}\times\binom{3}{5}\bigr)\bigr]\) (cf.~Figure~\ref{f:Postcut}) under the isomorphism \(\partial\).}
\label{f:prodstruct}
\end{figure}
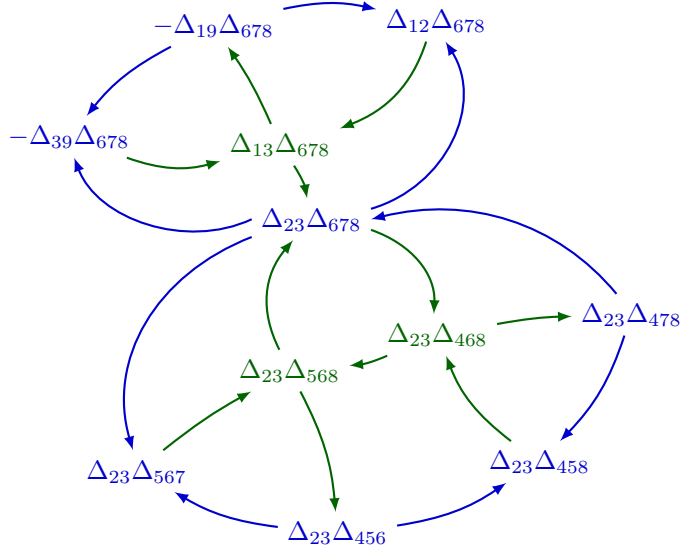
\end{eg}

\begin{prop}
\label{p:delta-formula}
If \(p\in\cmons_1\), then
\begin{equation}
\label{eq:delta-src-1}
\delta^{\src}(p)=\epsilon(k_2\deg_+^{\src}(p))p|_1\otimes (x_*|_2)^{\deg(p)},
\end{equation}
while if \(p\in\cmons_2\), then
\[\delta^{\src}(p)=\epsilon(k_2\deg(p)\deg_+^{\src}(x_*))(x_{*}|_1)^{\deg(p)}\otimes p|_2.\]
The analogous formulae hold for \(\delta^{\tgt}\).
\end{prop}
\begin{proof}
We give the proof for \(p\in\cmons_1\) and the source-labelling convention, the other cases being completely analogous.
First observe that the right-hand side of \eqref{eq:delta-src-1} is multiplicative in \(p\), so this formula holds for the cluster monomials of a given cluster if and only if it holds for the cluster variables of that cluster.

For the initial variables, the claim is just the definition of \(\delta^{\src}\) from Proposition~\ref{p:clust-Segre}.
Assuming that \eqref{eq:delta-src-1} holds for all cluster monomials from a given cluster, consider an exchange relation
\[x'=\frac{a+b}{x}\]
for this cluster; here \(x\in\cvars_1\) is a cluster variable associated to a vertex from \(Q_0^1\), while \(a\) and \(b\) are cluster monomials.
The fact that \(\deg\) and \(\deg_+^{\src}\) are gradings of \(\clustalg{D}\) means that \(\deg(a)=\deg(b)\) and \(\deg_+^{\src}(a)=\deg_+^{\src}(b)\), and we abbreviate these values to \(d\) and \(d_+\).
This means in particular that \(\deg(x)+\deg(x')=d\) and \(\deg_+(x)+\deg_+(x')=d_+\).
Moreover, by Proposition~\ref{p:quivsplit}\ref{p:quivsplit-noarrows}, we have \(a,b\in\cmons_1\).
Thus, we may compute that
\begin{align*}
\delta^{\src}(x')&=\frac{\delta^{\src}(a)+\delta^{\src}(b)}{\delta^{\src}(x)}\\
&=\frac{\epsilon(k_2d_+)a|_1\otimes (x_*|_2)^d+\epsilon(k_2d_+)b|_1\otimes (x_*|_2)^d}{\epsilon(k_2\deg_+(x))x|_1\otimes (x_*|_2)^{\deg(x)}}\\
&=\epsilon(k_2(d_+-\deg_+x))\frac{a|_1+b|_1}{x|_1}\otimes (x_*|_2)^{d-\deg(x)}\\
&=\epsilon(k_2\deg_+(x'|_1))x'|_1\otimes (x_*|_2)^{\deg(x')},
\end{align*}
as required.
The argument for a mutation at \(x\in\cvars_2\) is analogous, and completes the proof.
\end{proof}

We are now ready to prove the main result of this section.

\begin{thm}
\label{t:qco-conn}
If the cluster structures \(\GLsrc_i,\GLtgt_i\colon\clustalg{D_i}\isoto\CC[\openposvarcone(\posit_i)]\) quasi-coincide for \(i=1,2\), then the cluster structures \(\GLsrc,\GLtgt\colon\clustalg{D}\isoto\CC[\openposvarcone(\posit)]\) also quasi-coincide.
\end{thm}

\begin{proof}
We already know that any frozen variable for the cluster structure \(\GLtgt\) is a Laurent monomial in the frozen variables for \(\GLsrc\) by \cite[Prop.~7.13]{MulSpe-Twist} (see Proposition~\ref{p:sourceneck}), so let \(x\in\cvars_1\) be a non-frozen cluster variable.
By Propositions~\ref{p:clust-Segre} and \ref{p:delta-formula}, we calculate
\begin{equation}
\label{eq:GLtgt-val2}
\partial\GLtgt(x)=(\GLtgt_1\Segprod\GLtgt_1)(\delta^{\tgt}(x))
=\epsilon(k_2\deg_+^{\tgt}(x))\GLtgt_1(x|_1)\otimes\GLtgt_2(x_*|_2)^{\deg(x)}.
\end{equation}
Since \(\GLtgt_1\) quasi-coincides with \(\GLsrc_1\), there is a non-frozen cluster variable \(x'\in\cvars_1\) and Laurent monomials \(p\in\cmons_1\) and \(q\in\cmons_2\) in frozen variables such that \(\GLtgt_1(x|_1)=\GLsrc_1(x'p|_1)\) and \(\GLtgt_2(x_*|_2)=\GLsrc_2(q|_2)\).
It follows from these identities that \(\deg(x'p)=\deg(x)\), \(\deg_+^{\src}(x'p)=\deg_+^{\tgt}(x)\) and \(\deg(q)=1\), recalling that \(\deg(x)=\deg(x|_1)\) and so on.
In particular, \(\deg(\frac{q}{x_*})=0\).

Now we calculate
\begin{align*}
\delta^{\src}\Bigl(x'p\Big(\frac{q}{x_*}\Big)^{\deg(x)}\Bigr)
&=\epsilon( k_2\deg_+^{\src}(x'p))\bigl(x'p|_1\otimes(x_*|_2)^{\deg(x)}\bigr)\Bigl(1\otimes\frac{q}{x_*}\Big|_2\Bigr)^{\deg(x)} \\
&=\epsilon(k_2\deg_+^{\src}(x'p))x'p|_1\otimes (q|_2)^{\deg(x)}.
\end{align*}
Thus, by Proposition~\ref{p:clust-Segre} and \eqref{eq:GLtgt-val2} we have
\begin{align*}
\partial\GLsrc\Bigl(x'p\Bigl(\frac{q}{x_*}\Bigr)^{\deg(x)}\Bigr)
&=(\GLsrc_1\Segprod\GLsrc_2)\delta^{\src}\Bigl(x'p\Bigl(\frac{q}{x_*}\Bigr)^{\deg(x)}\Bigr)\\
&=\epsilon( k_2\deg_+^{\src}(x'p))\GLsrc_1(x'p|_1)\otimes \GLsrc_2(q|_2)^{\deg(x)}\\
&=\epsilon( k_2\deg_+^{\tgt}(x))\GLtgt_1(x|_1)\otimes \GLtgt_2(x_*|_2)^{\deg(x)}\\
&=\partial\GLtgt(x).
\end{align*}
Therefore \(\GLsrc(x'pq^{\deg(x)}/x_*^{\deg(x)})=\GLtgt(x)\) which, since \(x'\) is a cluster variable of \(\clustalg{D}\) and \(pq^{\deg(x)}/x_*^{\deg(x)}\) is a Laurent monomial in frozen variables, is what we needed to show.
Once again, the argument in the case that \(x\in\cvars_2\) is very similar.

Since removing frozen vertices disconnects \(Q\), setting frozen variables to \(1\) in \(\clustalg{D}\) leads to a disconnected cluster algebra \(\stab{\clust{A}}_D=\stab{\clust{A}}_{D_1}\times\stab{\clust{A}}_{D_2}\).
The requirement of Definition~\ref{d:qcm}\ref{d:qcm-stable} for the quasi-coincidence of \(\GLsrc\) and \(\GLtgt\) thus follows immediately from the corresponding property for \(\GLsrc_i\) and \(\GLtgt_i\).

It remains to check Definition~\ref{d:qcm}\ref{d:qcm-yhat}, and as pointed out in Remark~\ref{r:seed-by-seed} it is enough to do this on the initial seed.
Pick a mutable vertex \(j\in Q_0\) and consider
\[\yhat_j^{\tgt}=\prod_{i\in Q_0}\GLtgt(x_i)^{b_{ij}}=
\begin{cases}
\prod_{i\in Q_0^1}\GLtgt(x_i)^{b_{ij}},&j\in Q_0^1,\\
\prod_{i\in Q_0^2}\GLtgt(x_i)^{b_{ij}},&j\in Q_0^2
\end{cases}\]
since any arrow of \(Q\) has either both endpoints in \(Q_0^1\) or both in \(Q_0^2\).
In particular, \(\deg(\yhat_j^{\tgt})=\sum_{i\in Q_0^1}b_{ij}\) if \(j\in Q_0^1\), with the analogous formula for \(j\in Q_0^2\).

If \(j\in Q_0^1\), then we compute that
\begin{align*}
\partial(\yhat_j^{\tgt})
&=\prod_{i\in Q_0^1}\epsilon( k_2b_{ij}\deg_+(x_i))(\GLtgt_1(x_i|_1)\otimes \GLtgt(x_*|_2))^{b_{ij}}\\
&=\epsilon(k_2\deg_+(\yhat_j^{\tgt}))\yhat^{\tgt}_j|_1\otimes \varPlueck{\tgtlab{*}|_2}^{\deg(\yhat^{\tgt}_{j})},
\end{align*}
where
\[\yhat_{j}^{\tgt}|_1=\prod_{i\in Q_0^1}\GLtgt_1(x_i|_1)^{b_{ij}}\]
is the \(\yhat\)-variable attached to vertex \(j\) in the cluster structure \(\GLtgt_1\colon\clustalg{D_1}\isoto\CC[\openposvarcone_1]\).
We use Proposition~\ref{p:quivsplit}\ref{p:quivsplit-subquiv} here to see that the exponents \(b_{ij}\) match those in the calculation of \(\yhat_j^{\tgt}\).
In the same way, for \(j\in Q_0^2\) we have
\[\partial(\yhat_j^{\tgt})=\epsilon(k_2\deg(\yhat_j^{\tgt})\deg_+^{\tgt}(x_*))\varPlueck{\tgtlab{*}|_1}^{\deg(\yhat_j^{\tgt})}\otimes\yhat_j^{\tgt}|_2.\]

Now since \(\GLtgt_1\) quasi-coincides with \(\GLsrc_1\), there is a seed \(s_1\) for \(\clustalg{D_1}\) such that \(\yhat_{j}^{\tgt}|_1=\yhat_{j'}^{\src}(s_1)\) for all mutable \(j\in Q_0^1\).
Here \(j\mapsto j'\) is a bijection between the mutable vertices of \(Q_0^1\) and the mutable vertices of the quiver \(Q(s_1)\) of \(s_1\).
Similarly, there is a seed \(s_2\) for \(\clustalg{D_2}\) such that \(\yhat_{j}^{\tgt}|_2=\yhat_{j'}^{\src}(s_2)\) for all \(j\in Q_0^2\).
Choose a sequence of mutations from the initial seed of \(\clustalg{D_1}\) to \(s_1\), and a sequence of mutations from the initial seed of \(\clustalg{D_2}\) to \(s_2\).
Applying these sequences of mutations to \(Q\)---which we may do in any order by Corollary~\ref{c:mut-comm}---we obtain a seed \(s\) of \(\clustalg{D}\).
By Proposition~\ref{p:quivsplit}\ref{p:quivsplit-noarrows}, the full subquiver of \(Q(s)\) on \(Q_0^i\) is \(Q(s_i)\) for \(i=1,2\).

Combining the bijections \(j\mapsto j'\) for \(j\in Q_0^1\) and \(j\in Q_0^2\), we get a bijection \(j\mapsto j'\) from the mutable vertices of \(Q_D\) to those of \(Q(s)\). For \(j\in Q_0^1\), we calculate
\begin{align}
\label{eq:yhat-calc}
\begin{split}
\partial\yhat_{j'}^{\src}(s)
&=\prod_{i\in Q_0^1}\partial\GLsrc(x_{i'})^{b_{i'j'}}\\
&=\prod_{i\in Q_0^1}\epsilon(k_2b_{i'j'}\deg_+^{\src}(x_i'))\GLsrc_1(x_{i'}|_1)^{b_{i'j'}}\otimes \varPlueck{\tgtlab{*}|_2}^{b_{i'j'}\deg(x_{i'})}\\
&=\epsilon\bigl(k_2\deg_+(\yhat_{j'}^{\src}(s_1))\bigr)\yhat_{j'}^{\src}(s_1)\otimes \varPlueck{\tgtlab{*}|_2}^{\deg(\yhat_{j'}^{\src}(s_1))}\\
&=\epsilon(k_2\deg_+(\yhat_j^{\tgt}))\yhat_j^{\tgt}|_1\otimes\varPlueck{\tgtlab{*}|_2}^{\deg(\yhat_j^{\tgt})}\\
&=\partial\yhat_j^{\tgt},
\end{split}
\end{align}
and so \(\yhat_{j'}^{\src}(s)=\yhat_j^{\tgt}\).
As usual, a completely analogous calculation shows that \(\yhat_{j'}^{\src}(s)=\yhat_j^{\tgt}\) when \(j\in Q_0^2\), completing the proof.
\end{proof}

\begin{cor}
\label{c:conn}
If the source-labelled and target-labelled cluster structures quasi-coincide for connected positroids, then they quasi-coincide for all positroids.
\end{cor}
\begin{proof}
This follows by induction on Theorem~\ref{t:qco-conn}.
\end{proof}
 
\begin{eg}
Consider again the example from Figure~\ref{f:Postcut}.
One of the initial target-labelled variables from a vertex in \(Q_0^2\) is \(\Plueck{14679}\), which satisfies
\[\partial(\Plueck{14679})=-\Plueck{19}\otimes\Plueck{467}.\]
Since in this case both connected components give rise to uniform positroid varieties, for which the source-labelled and target-labelled cluster structures coincide exactly, both \(\Plueck{19}\) and \(\Plueck{467}\) are source-labelled cluster variables as well as target-labelled variables, and \(\Plueck{19}\) is moreover frozen.
Indeed, one can obtain \(\Plueck{467}\) from the initial source-labelled seed by mutating at the vertex of \(Q_{D_2}\) labelled by \(568\) (i.e.\ that labelled by \(23568\) in \(Q_D\)).
Thus, the recipe of Theorem~\ref{t:qco-conn} claims that in \(\CC[\openposvarcone_\posit]\), for \(\posit=\binom{2}{4}\times\binom{3}{5}\), we have
\[\Plueck{14679}=\Plueck{23467}\frac{\Plueck{16789}}{\Plueck{23678}},\]
with \(\Plueck{23678}\) being the frozen variable at the vertex \(*\).
Indeed, applying the isomorphism \(\partial\), we find that
\[\partial\Bigl(\Plueck{23467}\frac{\Plueck{16789}}{\Plueck{23678}}\Bigr)=(\Plueck{23}\otimes\Plueck{467})\frac{-\Plueck{19}\otimes\Plueck{678}}{\Plueck{23}\otimes\Plueck{678}}=-\Plueck{19}\otimes\Plueck{467}=\partial(\Plueck{14679}),\]
and so this claim is correct. Thus, the target-labelled variable \(\Plueck{14679}\) agrees with the product of the source-labelled variable \(\Plueck{23467}\) and the Laurent monomial \(\Plueck{16789}/\Plueck{23678}\) in frozen source-labelled variables, as required for quasi-coincidence.
\end{eg}

We close this section by giving similar results for Muller--Speyer's left twist automorphism of \(\openposvarcone_\posit\) \cite[\S6.1,Thm.~6.7]{MulSpe-Twist}, denoted by \(\ltwist\); we omit the definition, since we will mostly not need it.
We will also write \(\ltwist\colon\CC[\openposvarcone_\posit]\to\CC[\openposvarcone_\posit]\) for the induced map on functions.

\begin{prop}
\label{p:twist-comm}
Consider a disconnected positroid \(\posit=\posit_1\times\posit_2\), and write \(\ltwist\), \(\ltwist_1\) and \(\ltwist_2\) for the left twist automorphisms of \(\posit\), \(\posit_1\) and \(\posit_2\) respectively.
Then the diagram
\[\begin{tikzcd}[column sep=3em]
\openposvar(\posit)&\openposvar(\posit)\arrow{l}[swap]{\ltwist}\\
\openposvar(\posit_1)\times\openposvar(\posit_2)\arrow{u}{\dsum}&\openposvar(\posit_1)\times\openposvar(\posit_2)\arrow{l}[swap]{\ltwist_1\times\ltwist_2}\arrow{u}{\dsum}
\end{tikzcd}\]
commutes.
\end{prop}
\begin{proof}
By Proposition~\ref{p:posvar-decomp}, given \(A\in\openposvar(\posit)\) we have \(A=A^1\oplus A^2\) for \(A^i\in\openposvar(\posit_i)\), and we need to show that \(\ltwist_1A^1\oplus\ltwist_2A^2=\ltwist A\).
Via the decomposition \(\CC^n=\CC^{S_1}\oplus\CC^{S_2}\), we may view \(A=A^1\oplus A^2\) as a \(k\times n\) matrix in which the \(k_1\times n_1\) submatrix in rows \(\{1,\dotsc,k_1\}\) and columns \(S_1\) is \(A^1\), the \(k_2\times n_2\) submatrix in rows \(\{k_1+1,\dotsc,k\}\) and columns \(S_2\) is \(A^2\), and all other entries are zero.
The matrix \(\ltwist_1A^1\oplus\ltwist_2A^2\) has the same form.

Denote the ordinary Euclidean inner product on \(\CC^n\) by \(\langle\blank,\blank\rangle\), and the \(a\)-th column of a matrix \(M\) by \(M_a\).
Then for indices \(a\in S_i\) and \(b\in S_j\), we have
\[\langle(\ltwist_1A^1\dsum\ltwist_2A^2)_a,A_b\rangle=
\begin{cases}
\langle\ltwist_iA^i_a,A^i_b\rangle,&i=j,\\
0,&i\ne j,
\end{cases}\]
as a consequences of the description of these matrices above.
Now a direct comparison to the definition of \(\ltwist A\) given in \cite[\S6.1]{MulSpe-Twist} shows that \(\ltwist_1A^1\oplus\ltwist_2A^2=\ltwist A\) as required.
\end{proof}

\begin{cor}
\label{c:twist-comm}
In the setting of Proposition~\ref{p:twist-comm}, the diagram
\[\begin{tikzcd}[column sep=3em]
\CC[\openposvarcone(\posit)]\arrow{r}{\ltwist}\arrow{d}{\partial}&\CC[\openposvarcone(\posit)]\arrow{d}{\partial}\\
\CC[\openposvarcone(\posit_1)]\Segprod\CC[\openposvarcone(\posit_2)]\arrow{r}{\ltwist_1\Segprod\ltwist_2}&\CC[\openposvarcone(\posit_1)]\Segprod\CC[\openposvarcone(\posit_2)]
\end{tikzcd}\]
commutes, where \(\partial\) is the isomorphism from Corollary~\ref{c:Segre-product}.
\end{cor}

\begin{thm}
\label{t:twist-conn}
If each \(\ltwist_i\colon\CC[\openposvarcone(\posit_i)]\to\CC[\openposvarcone(\posit_i)]\) is a quasi-cluster morphism from \(\GLtgt_i\) to \(\GLsrc_i\), then \(\ltwist\colon\CC[\openposvarcone(\posit)]\to\CC[\openposvarcone(\posit)]\) is a quasi-cluster morphism from \(\GLtgt\) to \(\GLsrc\).
\end{thm}
\begin{proof}
The proof is along much the same lines as that of Theorem~\ref{t:qco-conn}, and so we concentrate mainly on the slight differences.

The argument opening the proof of \cite[Prop.~6.6]{MulSpe-Twist}, together with \cite[Thm.~6.7]{MulSpe-Twist}, shows that \(\ltwist\) inverts frozen variables for \(\GLtgt\), and so takes these frozen variables to Laurent monomials in the frozen variables for \(\GLsrc\) as in Proposition~\ref{p:sourceneck}.

Assume \(x\in\cvars_1\) is non-frozen, the proof for \(x\in\cvars_2\) being analogous. Applying \(\ltwist_1\Segprod\ltwist_2\) to \eqref{eq:GLtgt-val2} and using Corollary~\ref{c:twist-comm}, we see that
\begin{equation}
\label{eq:twist-calc-1}
\partial\ltwist\GLtgt(x)
=\epsilon(k_2\deg_+^{\tgt}(x))\ltwist_1\GLtgt_1(x|_1)\otimes\ltwist_2\GLtgt_2(x_*|_2)^{\deg(x)}.
\end{equation}
By assumption, we may write \(\ltwist_1\GLtgt_1(x|_1)=\GLsrc_1(x'p|_1)\) and \(\ltwist_2\GLtgt(x_*)=\GLsrc_2(q|_2)\), where \(x'\in\cvars_1\) is non-frozen, and \(p\in\cmons_1\) and \(q\in\cmons_2\) are Laurent monomials in frozen variables.
By \cite[Prop.~6.1, Rem.~6.2]{MulSpe-Twist}, we have \(\degvec(\ltwist(f))=-\degvec(\ltwist)\) for any \(\degvec\)-homogeneous element \(f\in\CC[\openposvarcone_\posit]\).
It follows that \(\deg(x'p)=-\deg(x)\), \(\deg_+^{\src}(x'p)=-\deg_+^{\tgt}(x)\) and \(\deg(q)=-\deg(x_*)=-1\), so that \(\deg(qx_*)=0\).
As in the proof of Theorem~\ref{t:qco-conn}, we may thus calculate
\[\partial\GLsrc(x'p(qx_*)^{\deg(x)})=\epsilon(k_2\deg_+^{\src}(x'p))\ltwist_1\GLtgt_1(x|_1)\otimes\ltwist_2\GLtgt_2(x_*|_2)^{\deg(x)}.\]
Since \(\deg_+^{\src}(x'p)=-\deg_+^{\tgt}(x)\), these two degrees have the same parity, and thus we see by comparing to \eqref{eq:twist-calc-1} that \(\GLsrc(x'p(qx_*)^{\deg(x)})=\ltwist\GLtgt(x)\), as required.

The condition in Definition~\ref{d:qcm}\ref{d:qcm-stable} follows from the disconnectedness of  \(\clustalg{D}=\stab{\mathscr{A}}_D\) exactly as in the proof of Theorem~\ref{t:qco-conn}.
To establish Definition~\ref{d:qcm}\ref{d:qcm-yhat}, our assumption on \(\ltwist_1\) and \(\ltwist_2\) means that there are seeds \(s_i\) for \(\clustalg{D_i}\), \(i=1,2\), such that \(\ltwist(\yhat_j^{\tgt}|_i)=\yhat_{j'}^{\src}(s_i)\), for a bijection \(j\mapsto j'\) between the mutable vertices of \(Q_0^i\) and the mutable vertices of \(Q(s_i)\).
For \(s\) the seed of \(\clustalg{D}\) obtained from \(s_1\) and \(s_2\) as in the proof of Theorem~\ref{t:qco-conn}, a calculation parallel to \eqref{eq:yhat-calc} shows that \(\yhat^{\src}_{j'}(s)=\ltwist\yhat_j^{\tgt}\) for each mutable \(j\in Q_0^1\), with an analogous calculation giving the same result for \(j\in Q_0^2\).
This completes the proof.
\end{proof}

\begin{rem}
\label{r:choices}
While we made particular choices for concreteness, the only property of the source-labelled and target-labelled cluster structures used in proving Theorem~\ref{t:twist-conn} was the commutativity from Proposition~\ref{p:clust-Segre}.
Thus, we can show by an identical argument that \(\ltwist\) is a quasi-cluster morphism from \(\GLsrc\) to \(\GLsrc\) provided \(\ltwist_i\) is a quasi-cluster morphism from \(\GLsrc_i\) to \(\GLsrc_i\), and so on for other combinations of labelling rules.
All of these statements turn out to be equivalent to each other; see Proposition~\ref{p:qcms} below.
\end{rem}

\section{Categorification}
\label{s:categorification}

Given Corollary~\ref{c:conn}, it remains to show that the source-labelled and target-labelled cluster structures on an open positroid variety \(\CC[\openposvarcone_{\posit}]\) quasi-coincide when \(\posit\) is connected.
To do this, we will use categorifications of the combinatorial ingredients introduced in Sections~\ref{s:positroids} and \ref{s:quasi-equivalence}, and in this section we recall the necessary details of these constructions.
For a connected Postnikov diagram \(D\), the cluster algebra \(\clustalg{D}\) was categorified by the author \cite{Pressland-Postnikov}, with the link to the cluster structures on \(\CC[\openposvarcone_\posit]\) clarified in work with Çanakçı and King \cite{CKP}.
Perfect matchings and Muller--Speyer's left twist automorphism of \(\openposvarcone_\posit\) are also categorified in \cite{CKP}.
The categorification of quasi-cluster morphisms is due to Fraser and Keller \cite[Appendix]{KelWu}.

\subsection{Categorification of positroids and positroid cluster algebras}
\label{s:categorification-posit}

Fix a connected Postnikov diagram \(D\), with plabic graph \(\Gamma=\Gamma_D\) and quiver \(Q=Q_D\).
Each node \(v\in\Gamma\) determines a cycle \(c_v\) in \(Q\) (up to cyclic equivalence) by following the arrows of \(Q\) crossing the edges of \(\Gamma\) incident with \(v\).
If \(v\) is black, then \(c_v\) is anticlockwise, whereas if \(v\) is white, then \(c_v\) is clockwise.

If \(a\in Q_1\) crosses a full edge of \(\Gamma\), incident with both a black node \(v\) and white node \(w\), then it is contained in both of the cycles \(c_v\) and \(c_w\).
Up to rotation, we may write these cycles as
\[c_v=a p^\black_a,\quad c_w=ap^\white_a\]
for paths \(p^\black_a,p^\white_a\colon t(a)\to s(a)\) from the target of \(a\) to its source.

\begin{defn}
\label{d:dimer-alg}
The \emph{dimer algebra} of \(D\) is
\[A_D=\cpa{\CC}{Q_D}/\close{\ideal{p^\black_a-p^\white_a:a\in Q_1\setminus F_1}},\]
where \(\cpa{\CC}{Q_D}\) denotes the complete path algebra of \(Q\), a topological algebra in which it makes sense to take closures of ideals.
Equivalently, \(A_D\) is the frozen Jacobian algebra \cite{Pressland-FJAs} of \((Q_D,F_D)\) with potential \(W_D=\sum_{v\in\Gamma^\black_0}c_v-\sum_{w\in\Gamma^\white_0}c_w\), for \(\Gamma^\black_0\) the set of black nodes of \(\Gamma_D\), and \(\Gamma^\white_0\) the set of white nodes.

The boundary algebra of \(D\) is
\[B_D=eA_De,\]
where \(e=\sum_{i\in F_0}\idemp{i}\) is the sum of vertex idempotents at the frozen vertices of \(Q_D\).
\end{defn}

Concretely, the algebra \(A_D\) is spanned (topologically) by the set of paths in the quiver \(Q_D\), subject to the relations indicated, while \(B_D\) is the subalgebra spanned by those paths which start and end on the frozen vertices (at the boundary of the disc).

\begin{rem}
Bearing in mind parts \ref{p:op-plabic} and \ref{p:op-quiv} of Proposition~\ref{p:op}, we see that \(A_{D^\op}=A_D^\op\) and \(B_{D^\op}=B_D^\op\).
\end{rem}

Write \(Z=\powser{\CC}{t}\).
Then both \(A_D\) and \(B_D\) can be given the structure of \(Z\)-algebras, and are free and finitely generated as \(Z\)-modules \cite[Prop.~2.15]{CKP}.
At a vertex \(j\) of \(Q_D\), the generator \(t\in Z\) acts on \(A_D\) by multiplication by a cycle \(c_j\) cyclically equivalent to one of the cycles \(c_v\) passing through \(j\); the choice of node \(v\) is irrelevant because of the defining relations of \(A_D\) and the fact that \(D\) is connected.
We typically abuse notation by identifying \(t\in Z\) both with
\[t=\sum_{i\in Q_0}c_i\in A_D\]
and with
\[ete=\sum_{i\in F_0}c_i\in B_D.\]

\begin{rem}
\label{r:disconnected}
If \(D\) is not connected, then \(A_D\) as defined above is not finitely generated as a \(Z\)-module.
This can be fixed by imposing additional relations to recover the property that all distinguished cycles \(c_v\) at a given vertex are equal in \(A_D\).
While these extra relations are rather natural, the resulting algebra is not a frozen Jacobian algebra, and so we lose access to several technical ingredients in the proof of Theorem~\ref{t:categorify} below (most notably the results from \cite[\S5]{Pressland-iCY} and \cite{Pressland-FJAs}).
On the other hand, we expect that the conclusion of this theorem does in fact remain true in the disconnected case, as long as \(A_D\) is modified in this way.

Results of Keller and Wu \cite{KelWu} (building on \cite{Wu-Higgs} by removing a Jacobi-finiteness assumption) may be applied to give an extriangulated categorification of the cluster algebra \(\clust{A}_D\) in all cases, using the ice quiver with potential \((Q_D,F_D,W_D)\) (and hence implicitly only the ordinary Jacobian relations from Definition~\ref{d:dimer-alg}).
However, this categorification will not satisfy the conclusion of Proposition~\ref{p:embed} below, and so is not suitable from the point of view of our methods.
Once again, we expect that Proposition~\ref{p:embed} remains true for \(D\) disconnected, when \(A_D\) is modified as above.

\end{rem}

\begin{defn}[{\cite{JKS}}]
Fix \(0<k<n\). 
We write \(Q_n\) for the quiver whose vertex set is \(\ZZ_n+\frac{1}{2}=\{\frac{1}{2},\frac{3}{2},\dotsc,n-\frac{1}{2}\}\), again viewed as a cyclically ordered set of \(n\) elements, and arrows \(x_i\colon (i-\frac{1}{2})\to (i+\frac{1}{2})\) and \(y_i\colon (i+\frac{1}{2})\to (i-\frac{1}{2})\) for each \(i\in\ZZ_n\).
Writing \(x=\sum_{i\in\ZZ_n}x_i\) and \(y=\sum_{i\in\ZZ_n}y_i\), we define the \emph{circle algebra} by
\[C_{k,n}=\cpa{\CC}{Q_n}/\close{\ideal{xy-yx,y^k-x^{n-k}}}.\]
\end{defn}

The algebra \(C_{k,n}\) is once again a free and finitely-generated \(Z\)-algebra with \(t\) acting as multiplication by \(xy=yx\) \cite[\S3]{JKS}, and we continue to abuse notation by identifying \(t\) with this element.

\begin{defn}
\label{d:CM}
Let \(\Lambda\) be a free and finitely generated \(Z\)-algebra. A \(\Lambda\)-module is \emph{Cohen--Macaulay} if it is free and finitely generated as a \(Z\)-module.
The full subcategory of \(\fgmod{\Lambda}\) on the Cohen--Macaulay modules is denoted by \(\CM(\Lambda)\).
The rank of \(M\in\CM(\Lambda)\) is defined by
\[\rank(M)=\dim_{\powrat{\CC}{t}}\powrat{\CC}{t}\otimes_ZM,\]
noting that \(\powrat{\CC}{t}=\Frac(Z)\).
For \(\Lambda=A_D\), this means that \(\rank(M)=\frac{1}{\#Q_0}\rank_Z(M)\).
\end{defn}

For \(\Lambda\) as in Definition~\ref{d:CM}, the category \(\CM(\Lambda)\) has enough projective objects, given by \(\add(\Lambda)\), and enough injective objects, given by \(\add(\Zdual{\Lambda})\), where \(\Zdual{(\blank)}=\Hom_Z(\blank,Z)\).
For \(C=C_{k,n}\), the category \(\CM(C)\) provides a categorification of Scott's cluster algebra structure on the homogeneous coordinate ring \(\CC[\Grasscone{k}{n}]\) of the Grassmannian, as demonstrated by Jensen, King and Su \cite{JKS}.
This category is related to the boundary algebras of Postnikov diagrams as follows.

\begin{prop}[{\cite[Prop.~3.6]{CKP}}]
\label{p:embed}
For \(D\) a connected Postnikov diagram of type \((k,n)\), there is a canonical map \(C_{k,n}\to B_D\) such that the induced restriction functor \(\CM(B_D)\to\CM(C_{k,n})\) is fully faithful.
\end{prop}

When \(D\) is uniform, Baur--King--Marsh show \cite[Thm.~11.2]{BKM} (cf.~\cite[Rem.~8.5]{CKP}) that the canonical map \(C_{k,n}\to B_D\) is an isomorphism, and so the restriction functor is an equivalence, but this is not the case in general.
In this paper, unlike in \cite{CKP}, we will mostly think of \(\CM(B_D)\) as a full subcategory of \(\CM(C_{k,n})\), treating the restriction functor as an inclusion rather than introducing extra notation for it.

\begin{defn}
\label{d:rank1}
Let \(I\subset\ZZ_n\) be a \(k\)-subset.
We write \(M_I\) for the representation of \(Q_n\) with \(Z\) at each vertex, \(x_i\) acting as multiplication by \(t\) if \(i\in I\) and as \(\id_Z\) otherwise, and \(y_i\) acting as \(\id_Z\) if \(i\in I\) and as multiplication by \(t\) otherwise. 
\end{defn}

One can show \cite[\S5]{JKS} that \(M_I\) is a Cohen--Macaulay \(C_{k,n}\)-module with \(\rank(M_I)=1\).
Moreover, any \(M\in\CM(C_{k,n})\) with \(\rank(M)=1\) is isomorphic to \(M_I\) for a unique \(k\)-subset \(I\) \cite[Prop.~5.2]{JKS}.
We view these modules as the categorical counterpart of the Plücker coordinates on \(\Grass{k}{n}\).
Indeed, by \cite[Eq.~9.4]{JKS}, the category \(\CM(C_{k,n})\) admits a cluster character \(\Psi\) such that \(\Psi(M_I)=\Plueck{I}\).
As with Plücker coordinates, we will sometimes write \(M(I)=M_I\) if this aids legibility.

\begin{prop}[{\cite[Prop.~8.6]{CKP}}]
The module \(M_I\in\CM(C_{k,n})\) is in the subcategory \(\CM(B_D)\) if and only if \(I\in\posit_D\).
\end{prop}

In contrast to the case of the full Grassmannian (or its big cell), the category of Cohen--Macaulay modules \(\CM(B_D)\) is not suitable for categorifying the cluster algebra \(\clustalg{D}\); for example, it is typically not a Frobenius exact category, since its projectives and injectives do not coincide.
Instead, we need to restrict ourselves to a Frobenius exact subcategory of it, and there are two natural choices.

\begin{defn}
Let \(\Lambda\) be a free and finitely generated \(Z\)-algebra.
We say that \(\Lambda\) is Iwanaga--Gorenstein if it has finite injective dimension as a module over itself on both the left and on the right; these injective dimensions then necessarily coincide.
We say that \(M\in\CM(\Lambda)\) is \emph{Gorenstein projective} if
\(\Ext^i_\Lambda(M,\Lambda)=0\)
for all \(i>0\), and that \(M\in\CM(\Lambda)\) is \emph{Gorenstein injective} if
\(\Ext^i_\Lambda(\Zdual{\Lambda},M)=0\)
for all \(i>0\).
The full subcategories of \(\CM(\Lambda)\) on the Gorenstein projective and Gorenstein injective modules are denoted by \(\gproj\CM(\Lambda)\) and \(\ginj\CM(\Lambda)\) respectively.
\end{defn}

By construction, \(\gproj\CM(\Lambda)\) is a Frobenius exact category whose projective-injective objects are those in \(\add(\Lambda)\), and \(\ginj\CM(\Lambda)\) is a Frobenius exact category whose projective-injective objects are those in \(\add(\Zdual{\Lambda})\).
By \cite[Lem.~3.6]{JKS}, for \(C=C_{k,n}\) we have \(\Zdual{C}\iso C\) and hence \(\gproj\CM(C)=\ginj\CM(C)=\CM(C)\), but in general all three categories will be different.

\begin{rem}
There is a slight asymmetry to these definitions arising from the fact that the Cohen--Macaulay property is defined with reference to free (i.e.\ projective) \(Z\)-modules.
The Iwanaga--Gorenstein property of \(\Lambda\) makes no reference to its \(Z\)-module structure (except that we omit the requirement that \(\Lambda\) is Noetherian from the usual definition, this being automatic for finitely generated \(Z\)-algebras), and given this property one would usually define Gorenstein projective\footnote{Confusingly, the modules we call Gorenstein projective are often, e.g.\ in \cite{Buchweitz-MCM}, called (maximal) Cohen--Macaulay.
The definition of Cohen--Macaulay we use here is more classical, and more closely related to commutative algebra, but is weaker than Gorenstein projectivity when the injective dimension of \(\Lambda\) is greater than the Krull dimension of the ground ring \(Z\), as will usually be the case for us.}
\(\Lambda\)-modules as the objects of the full subcategory
\[\GP(\Lambda)=\{M\in\fgmod{\Lambda}:\text{\(\Ext^i_\Lambda(M,\Lambda)=0\) for all \(i>0\)}\}.\]
It is well known that this coincides with the category of \(d\)-th syzygy modules, where \(d\) is the injective dimension of \(\Lambda\).
In our setting, the fact that \(Z\) is \(1\)-dimensional means that \(d\geq1\), so any \(M\in\GP(\Lambda)\) is a syzygy module, i.e.\ a submodule of an object in \(\add(\Lambda)\).
Since \(Z\) is a principal ideal domain, and \(\Lambda\) is free and finitely generated over \(Z\), the same is therefore true of \(M\), and hence
\(\GP(\Lambda)\subset\CM(\Lambda)\)
coincides with \(\gproj\CM(\Lambda)\); that is, a \(\Lambda\)-module is Gorenstein projective if and only if it both lies in \(\CM(\Lambda)\) and is Gorenstein projective in this subcategory.

However, this is not true for Gorenstein injectives.
The Gorenstein injective objects in \(\fgmod{\Lambda}\) are those in the full subcategory
\[\GI(\Lambda)=\{M\in\fgmod{\Lambda}:\text{\(\Ext^i_\Lambda(Q,M)=0\) for all \(i>0\)}\},\]
where \(Q\) is an injective generator of \(\fgmod{\Lambda}\), and these are not typically Cohen--Macaulay; indeed, \(Q\in\GI(\Lambda)\) is itself not free over \(Z\) in general. The Gorenstein injective objects in \(\CM(\Lambda)\) are defined instead with reference to the injective generator \(\Zdual{\Lambda}=\Hom_Z(\Lambda,Z)\) of this subcategory, which is not injective in \(\fgmod{\Lambda}\).
\end{rem}

By the following well-known result (see e.g.\ \cite[Prop.~7.2]{Auslander-Phil}), we may also describe Gorenstein injective Cohen--Macaulay modules in terms of Gorenstein projective Cohen--Macaulay modules for the opposite algebra.

\begin{prop}
\label{p:Zdual}
Let \(\Lambda\) be a free and finitely generated \(Z\)-algebra. Then \(\Lambda^\op\) is another such algebra, and the functor
\[\Zdual{(\blank)}\colon\CM(\Lambda^\op)\isoto(\CM(\Lambda))^\op\]
is an equivalence.
If \(\Lambda\) is Iwanaga--Gorenstein, then this restricts to an equivalence
\[\Zdual{(\blank)}\colon\gproj\CM(\Lambda^\op)\isoto(\ginj\CM(\Lambda))^\op.\]
\end{prop}

It turns out that the category \(\gproj\CM(B_D)\) is most suitable for categorifying the source-labelled cluster structure on \(\CC[\openposvarcone_\posit]\), whereas \(\ginj\CM(B_D)\) categorifies the target-labelled structure, as we will now explain.
For the source-labelled cluster structure, we have the following theorem.

\begin{thm}
\label{t:src-cat}
The category \(\gproj\CM(B_D)\) is a stably \(2\)-Calabi--Yau Frobenius exact category admitting a cluster-tilting object \(\Tsrc=eA_D\) such that the natural maps \(A_D\to\Endalg{B_D}{\Tsrc}\) and \(A_D/A_DeA_D\to\stabEndalg{B_D}{\Tsrc}\) are isomorphisms.
For each \(j\in Q_0\), the indecomposable summand \(eA_De_j\) of \(\Tsrc\) is isomorphic to \(M(\srclab{j})\) as an object of \(\CM(C_{k,n})\), so in particular we have
\[\Tsrc\iso\bigdsum_{j\in Q_0}M(\srclab{j}).\]
\end{thm}
\begin{proof}
The first statement is \cite[Thm.~4.5]{Pressland-Postnikov}, while the second is \cite[Prop.~8.2]{CKP}.
\end{proof}

Here \(\stabEndalg{B_D}{\Tsrc}\) denotes the stable endomorphism algebra of \(\Tsrc\), that is, its endomorphism algebra in the stable category \(\stabgproj{\CM(B_D)}=\gproj{\CM(B_D)}/\add(B_D)\).
This result motivates the shorthand \(\stab{A}_D\defeq A_D/A_DeA_D\). By applying Theorem~\ref{t:src-cat} to the opposite diagram \(D^\op\) and using Proposition~\ref{p:Zdual}, we obtain the corresponding result for \(\ginj\CM(B_D)\).
The statement identifying the relevant cluster-tilting object in \(\CM(C_{k,n})\) can also be found in \cite[Rem.~8.4]{CKP}.

\begin{thm}
\label{t:tgt-cat}
The category \(\ginj\CM(B_D)\) is a stably \(2\)-Calabi--Yau Frobenius exact category admitting a cluster-tilting object \(\Ttgt=\Zdual{(A_De)}\) such that the natural maps \(A_D\to\Endalg{B_D}{\Ttgt}\) and \(\stab{A}_D\to\stabEndalg{B_D}{\Ttgt}\) are isomorphisms. For each \(j\in Q_0\), the indecomposable summand \(\Zdual{(e_jA_De)}\) of \(\Ttgt\) is isomorphic to \(M(\tgtlab{j})\) as an object of \(\CM(C_{k,n})\), so in particular we have
\[\Ttgt\iso\bigdsum_{i\in Q_0}M(\tgtlab{j}).\]
\end{thm}

We emphasise that in Theorem~\ref{t:tgt-cat} the notation \(\stabEndalg{B_D}{\Ttgt}\) refers to the endomorphism algebra of \(\Ttgt\) in the stable category \(\stabginj{\CM(B_D)}=\ginj{\CM(B_D)}/\add({\Zdual{B}_D})\), and not to the projectively stable endomorphism algebra of \(\Ttgt\) as a \(B_D\)-module.

\begin{rem}
Because \(A_D\) is Noetherian (being free and finitely generated over the Noetherian ring \(Z\)) and of finite global dimension \cite[Thm.~3.7]{Pressland-Postnikov}, it follows from Theorems~\ref{t:src-cat} and \ref{t:tgt-cat} together with a result of Kalck, Iyama, Wemyss and Yang \cite[Thm.~2.7]{KIWY} that the categories \(\gproj\CM(B_D)\) and \(\ginj\CM(B_D)\) are equivalent, via an equivalence that takes \(\Tsrc\) to \(\Ttgt\).
However, this is not the relationship between these two categories which will ultimately lead to quasi-equivalence of the two cluster structures.
\end{rem}

By Theorems~\ref{t:src-cat} and \ref{t:tgt-cat}, both \(\gproj\CM(B_D)\) and \(\ginj\CM(B_D)\) serve as categorifications of the abstract cluster algebra \(\clustalg{D}\).
Cluster-tilting objects, which both of these categories have, play the role of seeds or clusters, the stably \(2\)-Calabi--Yau property ensuring that these objects admit a mutation theory \cite{IyaYos}.
The fact that the cluster-tilting objects \(\Tsrc\) and \(\Ttgt\) have endomorphism algebra \(A_D\), which is defined from the same quiver \(Q_D\) as the cluster algebra \(\clustalg{D}\), means that the mutation classes of these cluster-tilting objects exhibit the same combinatorics as the mutation class of seeds in \(\clustalg{D}\).
More precisely, we have the following.

\begin{thm}[{\cite[Thm.~6.11]{Pressland-Postnikov}}]
\label{t:categorify}
Both \((\gproj\CM(B_D),\Tsrc)\) and \((\ginj\CM(B_D),\Ttgt)\) are Frobenius \(2\)-Calabi--Yau realisations of the cluster algebra \(\clustalg{D}\) \cite[Def.~5.1]{FuKel} (see also \cite[Def.~6.9]{Pressland-Postnikov}).

In particular, both categories carry a Fu--Keller cluster character \cite[\S3]{FuKel}, providing a bijection between the mutation class of the given cluster-tilting object and the set of seeds of \(\clustalg{D}\), inducing a bijection between the indecomposable summands of the objects in this mutation class and the cluster variables, compatible with the combinatorics of mutation on each set.
\end{thm}

Cluster-tilting objects in \(\gproj\CM(B_D)\) in the mutation class of \(\Tsrc\), and objects of their additive closures, are called \emph{reachable}.
We use the same terminology for the mutation class of \(\Ttgt\) in \(\ginj\CM(B_D)\).

Recall that \(\CM(C_{k,n})\) carries a cluster character \(\Psi\) with values in \(\CC[\Grasscone{k}{n}]\), such that \(\Psi(M_I)=\Plueck{I}\) for each \(k\)-subset \(I\) \cite[\S9]{JKS}.
We may restrict the domain of \(\Psi\) to either \(\gproj\CM(B_D)\) and \(\ginj\CM(B_D)\), and restrict its values to functions on \(\openposvarcone_\posit\).
Since these subcategories are extension-closed by \cite[Prop~7.2]{Pressland-Postnikov} and its dual, each such modification is a cluster character with values in \(\CC[\openposvarcone_\posit]\), which we will again denote by \(\Psi\).
Additionally, we have the Fu--Keller cluster characters on \(\gproj\CM(B_D)\) and \(\ginj\CM(B_D)\) from Theorem~\ref{t:categorify}, which take values in \(\clust{A}_D^+=\clust{A}_D\) by \cite[Thm.~1.3]{Plamondon-Generic} and Proposition~\ref{p:uca}, and thus can be composed with either of Galashin--Lam's isomorphisms \(\GLsrc\) and \(\GLtgt\) from Theorem~\ref{t:GalLam}; we denote the resulting cluster characters with values in \(\CC[\openposvarcone_\posit]\) by \(\CCsrc\) and \(\CCtgt\) respectively.

An object \(X\) of an exact category \(\cat{E}\) is called \emph{rigid} if \(\Ext^1_\cat{E}(X,X)=0\).

\begin{thm}
\label{t:CC-values}
If \(X\in\gproj\CM(B_D)\) is reachable and rigid, then \(\CCsrc(X)=\Psi(X)\), whereas if \(X\in\ginj\CM(B_D)\) is reachable and rigid, then \(\CCtgt(X)=\Psi(X)\).
In particular, if \(I\) is a \(k\)-subset and \(M_I\in\gproj\CM(B_D)\), then \(\CCsrc(M_I)=\Plueck{I}\), whereas if \(M_I\in\ginj\CM(B_D)\), then \(\CCtgt(M_I)=\Plueck{I}\).
\end{thm}
\begin{proof}
For \(\gproj\CM(B_D)\), this is \cite[Prop.~7.5]{Pressland-Postnikov} and \cite[Thm.~7.6]{Pressland-Postnikov}, the latter result showing that any \(M_I\in\gproj\CM(B_D)\) is reachable.
The statements for \(\ginj\CM(B_D)\) can be proved similarly, or deduced from those for \(\gproj\CM(B_{D^\op})\).
\end{proof}

\begin{rem}
We do not claim that the cluster characters \(\CCsrc\) and \(\Psi\) coincide on every object of \(\gproj\CM(B_D)\) (nor make the analogous claim for \(\ginj\CM(B_D)\)), and to the best of our knowledge this remains an open question in this generality.
The two functions are constructed rather differently: the cluster character \(\CCsrc\) is obtained from a formula of Fu and Keller \cite{FuKel}, building on earlier work of Palu \cite{Palu-CC} and Caldero--Chapoton \cite{CalCha}, whose expression depends on the choice of cluster-tilting object \(\Tsrc\), whereas \(\Psi\) is based on Lusztig's construction of the dual semicanonical basis in terms of modules for a preprojective algebra \cite{Lusztig-Quivers,Lusztig-Semican}, and is explicitly independent of any choice of cluster-tilting object.
However, in the absence of any explicit examples to demonstrate otherwise, it remains possible that the two functions are in fact equal.
Some evidence for this comes from the case that $D$ is uniform, for which $\gproj\CM(B_D)=\CM(C_{k,n})=\ginj\CM(B_D)$; here \(\CCtgt=\Psi\) by \cite[Thm.~9.11]{JKS3}.
\end{rem}

\begin{eg}
\label{eg:running2}
For the Postnikov diagram \(D\) from Example~\ref{eg:running1}, the Auslander--Reiten quiver of \(\gproj\CM(B_D)\) is shown in Figure~\ref{f:gproj}, and that of \(\ginj\CM(B_D)\) is shown in Figure~\ref{f:ginj}.
We use Proposition~\ref{p:embed} to identify these categories with their embeddings into the Grassmannian cluster category \(\CM(C_{3,7})\).
The notation \(M(\frac{257}{136})\) in Figure~\ref{f:gproj} refers to an indecomposable rank \(2\) module with a filtration whose composition factors are \(M_{136}\) and \(M_{257}\), as in \cite[\S6]{JKS}, and similarly for the other rank \(2\) modules.
One can also view the label of a module as notation for its profile (see \cite[\S6]{JKS} again), although our conventions differ from those of \cite{JKS} in such a way that the label records the upward, rather than downward, steps in this profile.

\begin{figure}
\begin{tikzpicture}[xscale=1.7,yscale=1.3]
\foreach \x/\y/\n/\l in {-3/-1/357l/357, -2/0/347/347, -1/1/137/137, -1/-1/346/346, 0/0/136/136, 1/1/126/126, 1/-1/135/135, 3/1/357r/357}
{\draw (\x,\y) node (\n) {\(M_{\l}\)};}

\foreach \x/\y/\n/\l in {-3/1/247136l/\frac{247}{136}, 2/0/257136/\frac{257}{136}, 3/-1/247136r/\frac{247}{136}}
{\draw (\x,\y) node (\n) {\(M(\l)\)};}

\foreach \x/\y/\n in {-2/2.25/167, -2/1.75/123, -2/-2/356, 0/2/127, 0/-2/345, 1/0.35/367, 2/-2/134}
{\draw (\x,\y) node (\n) {\(\boxed{M_{\n}}\)};}

\foreach \s/\t in {247136l/167.west, 247136l/123.west, 247136l/347, 167.east/137, 123.east/137, 347/137, 357l/347, 357l/356, 347/346, 356/346, 137/127, 137/136, 346/136, 346/345, 127/126, 136/126, 136/367, 136/135, 345/135, 126/257136, 367/257136, 135/257136, 135/134, 257136/357r, 257136/247136r, 134/247136r}
{\path [-angle 90] (\s) edge (\t);}

\draw (-3,0) node (l) {};
\draw (3,0) node (r) {};
\foreach \l/\r in {247136l/137, l/347, 357l/346, 347/136, 137/126, 346/135, 136/257136, 126/357r, 135/247136r, 257136/r}
{\path [dashed] (\l) edge (\r);}

\draw [dotted] (-3,-2.5) -- (357l) -- (247136l) -- (-3,2.5);
\draw [dotted] (3,-2.5) -- (247136r) -- (357r) -- (3,2.5);
\end{tikzpicture}
\caption{The Auslander--Reiten quiver of the category \(\gproj\CM(B_D)\), for \(B_D\) the boundary algebra of the plabic graph in Figure~\ref{f:eg}. Objects displayed in boxes are projective-injective, and the left-hand and right-hand ends are identified via a Möbius twist. The category is embedded in the Grassmannian cluster category \(\CM(C_{3,7})\), and objects named using their profiles as \(C_{3,7}\)-modules.}
\label{f:gproj}
\end{figure}
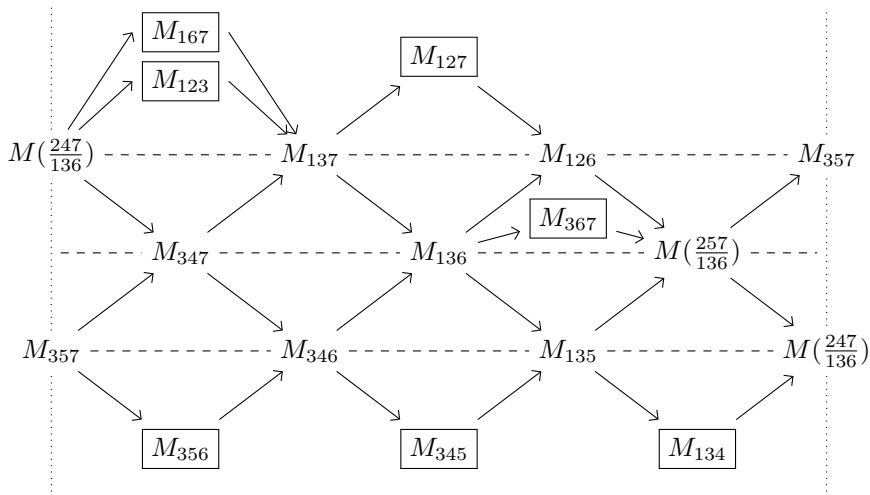

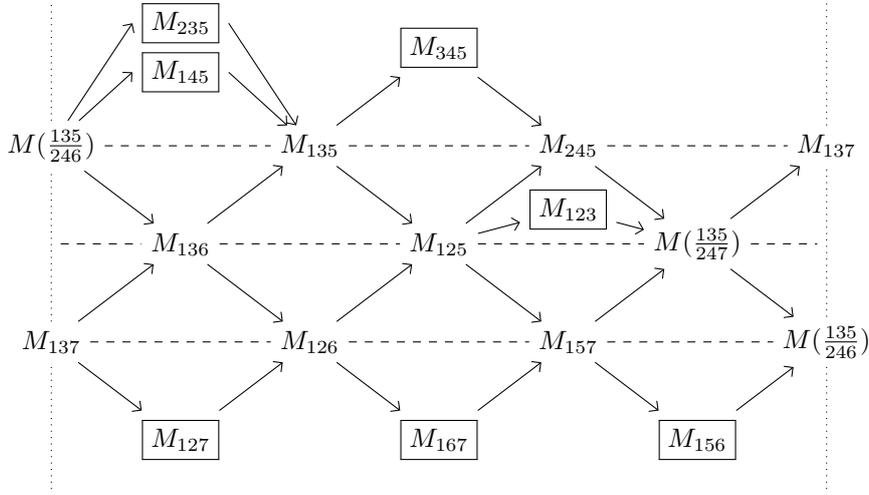
\begin{figure}
\begin{tikzpicture}[xscale=1.7,yscale=1.3]
\foreach \x/\y/\n/\l in {-3/-1/357l/137, -2/0/347/136, -1/1/137/135, -1/-1/346/126, 0/0/136/125, 1/1/126/245, 1/-1/135/157, 3/1/357r/137}
{\draw (\x,\y) node (\n) {\(M_{\l}\)};}

\foreach \x/\y/\n/\l in {-3/1/247136l/\frac{135}{246}, 2/0/257136/\frac{135}{247}, 3/-1/247136r/\frac{135}{246}}
{\draw (\x,\y) node (\n) {\(M(\l)\)};}

\foreach \x/\y/\n/\l in {-2/2.25/167/235, -2/1.75/123/145, -2/-2/356/127, 0/2/127/345, 0/-2/345/167, 1/0.35/367/123, 2/-2/134/156}
{\draw (\x,\y) node (\n) {\(\boxed{M_{\l}}\)};}

\foreach \s/\t in {247136l/167.west, 247136l/123.west, 247136l/347, 167.east/137, 123.east/137, 347/137, 357l/347, 357l/356, 347/346, 356/346, 137/127, 137/136, 346/136, 346/345, 127/126, 136/126, 136/367, 136/135, 345/135, 126/257136, 367/257136, 135/257136, 135/134, 257136/357r, 257136/247136r, 134/247136r}
{\path [-angle 90] (\s) edge (\t);}

\draw (-3,0) node (l) {};
\draw (3,0) node (r) {};
\foreach \l/\r in {247136l/137, l/347, 357l/346, 347/136, 137/126, 346/135, 136/257136, 126/357r, 135/247136r, 257136/r}
{\path [dashed] (\l) edge (\r);}

\draw [dotted] (-3,-2.5) -- (357l) -- (247136l) -- (-3,2.5);
\draw [dotted] (3,-2.5) -- (247136r) -- (357r) -- (3,2.5);
\end{tikzpicture}
\caption{The Auslander--Reiten quiver for \(\ginj\CM(B_D)\), with conventions as in Figure~\ref{f:gproj}.}
\label{f:ginj}
\end{figure}

Comparing to Tables~\ref{tab:src-eg} and \ref{tab:tgt-eg}, we see that the rank \(1\) modules in each category have the same labels as the Plücker coordinates which are cluster variables in the corresponding cluster structure, with the projective-injective objects sharing their labels with the frozen variables.
Each category also contains a pair of rank \(2\) modules, corresponding to the pair of degree \(2\) cluster variables.

As a sample computation, observe that there are short exact sequences
\[\begin{tikzcd}
M_{136}\arrow{r}&\begin{matrix}M_{126}\\\dsum\\M_{367}\\\dsum\\M_{135}\end{matrix}\arrow{r}&M(\frac{257}{136}),
\end{tikzcd}\quad
\begin{tikzcd}
M(\frac{257}{136})\arrow{r}&\begin{matrix}M_{167}\\\dsum\\M_{123}\\\dsum\\M_{356}\end{matrix}\arrow{r}&M_{136}
\end{tikzcd}\]
in \(\gproj\CM(B_D)\), in which all objects except \(M(\frac{257}{136})\) have summands from the cluster-tilting object \(M_{136}\dsum M_{126}\dsum M_{135}\dsum B_D\).
This is different from the initial cluster-tilting object \(\Tsrc=eA_D=M_{137}\oplus M_{135}\oplus M_{357}\oplus B_D\), but is more convenient for our computation.

It then follows that
\begin{align*}
\CCsrc\bigl(M(\tfrac{257}{136})\bigr)&=\frac{\CCsrc(M_{126}\dsum M_{367}\dsum M_{135})+\CCsrc(M_{167}\dsum M_{123}\dsum M_{356})}{\CCsrc(M_{136})}\\
&=\frac{\Plueck{126}\Plueck{367}\Plueck{135}+\Plueck{167}\Plueck{123}\Plueck{356}}{\Plueck{136}}
\end{align*}
by Theorem~\ref{t:CC-values} and the multiplication formula for cluster characters \cite[Thm.~3.3(d)]{FuKel}.
Now using the Plücker relations
\begin{align*}
\Plueck{126}\Plueck{135}+\Plueck{123}\Plueck{156}&=\Plueck{125}\Plueck{136},\\
\Plueck{167}\Plueck{356}-\Plueck{367}\Plueck{156}&=\Plueck{136}\Plueck{567}=0
\end{align*}
on \(\openposvar_{\posit}\), where \(\Plueck{567}=0\), we see that this expression simplifies to
\[\CCsrc\bigl(M(\tfrac{257}{136})\bigr)=\Plueck{125}\Plueck{367},\]
cf. Table~\ref{tab:src-eg}.
\end{eg}

\subsection{Categorification of perfect matchings}
\label{s:categorification-pms}

Recall the notion of a perfect matching of a plabic graph from Definition~\ref{d:pm}, defined as a subset of the set of edges and half-edges.
For the plabic graph \(\Gamma_D\) associated to a Postnikov diagram \(D\), these edges are in natural bijection with the arrows of the quiver \(Q=Q_D\), and so we may also view a perfect matching as a subset of this set of arrows.
In terms of \(Q_D\), a perfect matching \(\mu\) is defined by the property that each of the distinguished cycles \(c_v\) contains exactly one arrow from \(\mu\).

\begin{defn}[{\cite[Def.~4.3]{CKP}}]
Let \(\mu\) be a perfect matching of \(Q=Q_D\). The \emph{perfect matching module} \(N_\mu\) is the \(A_D\)-module given by the following representation of \(Q\).
At each vertex in \(Q_0\), we take the vector space \(Z=\powser{C}{t}\).
The linear map \(Z\to Z\) associated to an arrow \(\alpha\in Q_1\) is multiplication by \(t\) if \(\alpha\in\mu\), and \(\id_Z\) otherwise.
\end{defn}

Once again, we will sometimes write \(N(\mu)\) instead of \(N_\mu\).
In a perfect matching module \(N_\mu\), any representative of a cycle \(c_v\) acts as multiplication by \(t\) (on the relevant copy of \(Z\) at a vertex).
This applies in particular to the cycles \(ap_a^\black\) and \(ap_a^\white\) associated to an internal arrow \(a\), and so we see that either \(p_a^\black\) and \(p_a^\white\) both act on \(N_\mu\) as \(\id_Z\), if \(\alpha\in\mu\), or they both act as multiplication by \(t\) if \(\alpha\not\in\mu\).
As a result, \(N_\mu\) does indeed describe an \(A_D\)-module as claimed.

The perfect matching modules are the analogue for \(A_D\) of the rank \(1\) modules \(M_I\) for \(C_{k,n}\) from Definition~\ref{d:rank1}.
Indeed, in the construction of \(M_I\), the role of the distinguished cycles \(c_v\) is played by the terms of the product \(xy\), each of which is a \(2\)-cycle in which exactly one of the two arrows acts as multiplication by \(t\) in \(M_I\), and the other by \(\id_Z\).
Analogous to the \(M_I\), the modules \(N_\mu\) are Cohen--Macaulay \(A_D\)-modules with \(\rank(N_\mu)=1\), and any \(N\in\CM(A)\) with \(\rank(N)=1\) is isomorphic to \(N_\mu\) for a unique perfect matching \(\mu\) \cite[Cor.~4.6]{CKP}.
A further relationship between the \(N_\mu\) and the \(M_I\) is the following.

\begin{prop}[{\cite[Prop.~4.9]{CKP}}]
\label{p:bdry-value}
Let \(\mu\) be a perfect matching.
Then the \(B_D\)-module \(eN_\mu\) obtained by restricting \(N_\mu\) to the boundary coincides, when viewed in \(\CM(C_{k,n})\), with the rank \(1\) module \(M_{\bdry\mu}\).
\end{prop}

Because the indecomposable projective \(A_D\)-modules are Cohen--Macaulay with rank \(1\) by \cite[Prop.~2.15]{CKP}, they are isomorphic to perfect matching modules.
The same is true for the indecomposable injective objects in \(\CM(A_D)\), and in both cases it turns out that we have seen the corresponding matchings already in Proposition~\ref{p:MSmats}.

\begin{thm}[{\cite[Cor.~7.6--7.7]{CKP}}]
\label{t:MSmats}
Let \(j\) be a vertex of \(Q_D\).
Then the indecomposable projective \(A\)-module \(Ae_j\) is isomorphic to the perfect matching module \(N(\MSsrc{j})\) associated to the downstream wedge matching \(\MSsrc{j}\).
Dually, the module \(\Zdual{(e_jA)}\), which is indecomposable injective in \(\CM(A)\), is isomorphic to \(N(\MStgt{j})\), for \(\MStgt{j}\) the upstream wedge matching.
\end{thm}

\begin{cor}
\label{c:bdry-value}
For any vertex \(j\) of \(Q_D\), we have isomorphisms \(eAe_j\iso M(\srclab{j})\) and \(\Zdual{(e_jAe)}\iso M(\tgtlab{j})\).
\end{cor}
\begin{proof}
Combine Theorem~\ref{t:MSmats} with Proposition~\ref{p:bdry-value} and \eqref{eq:MSbdry}.
\end{proof}

Because the algebra \(A_D\) has finite global dimension \cite[Thm.~3.7]{Pressland-Postnikov}, we may associate to each \(N\in\fgmod{A_D}\) the class \([N]\in\Kgp_0(\projcat{A_D})\) of a projective resolution of \(N\).
When \(N=N_\mu\) is a perfect matching module, \cite[Thm.~6.7]{CKP} gives an explicit projective resolution of \(N_\mu\), leading to an expression for \([N_\mu]\) in terms of the classes of indecomposable projective \(A_D\)-modules \(P_j=A_D\idemp{j}\) for \(j\in Q_0\) \cite[Prop.~6.9]{CKP}.
In this paper, it will suffice to know the image of this class under the natural projection \(\Kgp_0(\projcat{A_D})\to\Kgp_0(\projcat{\stab{A}_D})\) with kernel spanned by the classes \([P_j]\) for \(j\in F_0\), recalling that \(\stab{A}_D=A_D/A_DeA_D\).

\begin{prop}[{\cite[Prop.~6.9]{CKP}}]
\label{p:mat-res}
Let \(\mu\) be a perfect matching of \(D\). Then in \(\Kgp_0(\projcat{\stab{A}_D})\), we have
\begin{equation}
\label{eq:white-class}
[N_{\mu}]=\sum_{j\in Q_0\setminus F_0}[Ae_j]-\wt[\white](\mu),
\end{equation}
where
\[\wt[\white](\mu)=\sum_{\gamma\in\mu}\wt[\white](\gamma),\quad\wt[\white](\gamma)=\sum_{j\in\pos_0(\gamma)}[Ae_j],\]
and \(\pos_0(\gamma)\) is obtained from the set of vertices appearing in the positively oriented face of \(Q\) with \(\gamma\) in its boundary by removing the two vertices \(s(\gamma)\) and \(t(\gamma)\) incident with \(\gamma\).
We also have
\begin{equation}
\label{eq:black-class}
[N_{\mu}]=\sum_{j\in Q_0\setminus F_0}[Ae_j]-\wt[\black](\mu)
\end{equation}
where
\[\wt[\black](\mu)=\sum_{\gamma\in\mu}\wt[\black](\gamma),\quad\wt[\black](\gamma)=\sum_{j\in\neg_0(\gamma)}[Ae_j]\]
and \(\neg_0(\gamma)\) is defined analogously to \(\pos_0(\gamma)\) but using the negatively oriented face containing \(\gamma\).
\end{prop}

\begin{rem}
While \cite[Prop.~6.9]{CKP} is stated for standardised diagrams, standardisation only changes \(Q_D\) by adding arrows between frozen vertices, and so has no effect on the formulae for \([N_\mu]\in\Kgp_0(\projcat{\stab{A}_D})\), in which contributions from frozen vertices are discounted.
A consequence of Proposition~\ref{p:mat-res} is that \(\wt[\black](\mu)=\wt[\white](\mu)\) in \(\Kgp_0(\projcat{\stab{A}_D})\) for any matching \(\mu\), although the corresponding identity does not hold in \(\Kgp_0(\projcat{A_D})\) when the contributions from frozen vertices are included.
\end{rem}

For \(j\in Q_0\), define \(\Plueck{}^{[P_j]}=\varPlueck{\srclab{j}}\).
Extending this to a group homomorphism, we obtain a Laurent monomial \(\Plueck{}^v\) in the initial source-labelled cluster of Plücker coordinates \(\varPlueck{\srclab{j}}\) for any \(v\in\Kgp_0(\projcat{A_D})\).
Abbreviating \(F=\Hom_{B_D}(\Tsrc,\blank)\colon\CM(B_D)\to\fgmod{A_D}\) and  \(G=\Ext^1_{B_D}(\Tsrc,\blank)\colon\CM(B_D)\to\fgmod{\stab{A}_D}\subseteq\fgmod{A_D}\), we may write Fu--Keller's cluster character formula for \(X\in\gproj\CM(B_D)\) as
\[\CCsrc(X)=\Plueck{}^{[FX]}\sum_{\dimvec{E}:E\leq GX}\chi(\qGrass{\dimvec{E}}{GX})\Plueck{}^{-[E]},\]
as in \cite[\S10]{CKP}.
This gives \(\CCsrc(X)\) a leading term indexed by \(0\leq GX\), with exponent \([FX]\), and a trailing term indexed by \(GX\leq GX\), with exponent \([FX]-[GX]\).
(Depending on a choice of cluster-algebraic convention, one or the other of these exponents is the \emph{$\mathbf{g}$-vector} of \(\CCsrc(X)\).)
The integer coefficients are computed as Euler characteristics of quiver Grassmannians; the details of this will not be important to us here.

Now observe that for any object \(X\in\CM(B_D)\), there is a short exact sequence
\begin{equation}
\label{eq:F'}
\begin{tikzcd}0\arrow{r}&F'X\arrow{r}&FX\arrow{r}&G\bOmega X\arrow{r}&0,\end{tikzcd}
\end{equation}
where \(\bOmega X\) is an arbitrary syzygy of \(X\), i.e.\ the kernel of an arbitrary (not necessarily minimal) projective cover \(PX\to X\).
Since \(G\) vanishes on projective \(B_D\)-modules, \(G\bOmega X\) is independent of this choice of syzygy, and the sequence \eqref{eq:F'} defines a third functor \(F'\colon\CM(B_D)\to\fgmod(A_D)\).
In particular, the assignment \(N\mapsto N/F'X\) is a bijection between submodules of \(FX\) which contain \(F'X\) and submodules of \(G\bOmega X\).
We may also describe \(F'X\) directly as the subspace of \(\Hom_{B_D}(\Tsrc,X)\) consisting of morphisms factoring over a projective \(B_D\)-module.
See \cite[Cor.~5.11]{CKP} for a proof that \(F'\) coincides with the functor of the same name used in \cite{CKP}. 

\begin{prop}
\label{p:g-vec}
Let \(X\in\CM(B_D)\) and let \(\bOmega X\in\CM(B_D)\) be an arbitrary syzygy of \(X\).
Then \(\bOmega X\in\gproj\CM(B_D)\), and in \(\Kgp_0(\projcat{\stab{A}_D})\) we have the identities \([F\Omega X]=-[F'X]\) and \([F\bOmega X]-[G\bOmega X]=-[FX]\).
\end{prop}
\begin{proof}
The first statement is \cite[Lem.~10.4]{CKP}. For the second, as in the proof of \cite[Prop.~10.2]{CKP}, we have for any submodule \(F'X\leq N\leq FX\) that
\[[F\bOmega X]-[N/F'X]=[FP]-[N]\]
in \(\Kgp_0(\projcat{A_D})\).
Projecting to \(\Kgp_0(\projcat{\stab{A}_D})\), where \([FP]=0\), and taking either \(N=F'X\) or \(N=FX\), we obtain \([F\bOmega X]=-[F'X]\) and \([F\bOmega X]-[G\bOmega X]=-[FX]\) as required.
\end{proof}

As a corollary, \(-[F'X]\) is the exponent of the leading term of \(\CCsrc(\bOmega X)\), i.e.\ that indexed by the zero module.
Similarly, the exponent of the trailing term, indexed by \(G\bOmega X\), is \(-[FX]\).

In \cite[\S5]{CKP}, we show that each \(A_D\)-module \(N\) with \(F'M_I\leq N\leq FM_I\) is isomorphic to \(N_\mu\) for a unique matching \(\mu\) with \(\bdry\mu=I\), and this assignment is a bijection between these intermediate submodules and perfect matching modules with this fixed boundary value.
This leads to an expression \cite[Thm.~10.3]{CKP} for \(\CCsrc(\bOmega M_I)\), for a carefully chosen syzygy \(\bOmega M_I\), as a dimer partition function counting perfect matchings with boundary value \(I\).
A second important consequence for our purposes here is that it will allow us to combine Propositions~\ref{p:mat-res} and \ref{p:g-vec} to compute the leading exponent of \(\CCsrc(\bOmega M_I)\) explicitly for any \(I\in\posit_D\); we will return to this in Section~\ref{s:proof}.

\subsection{Categorification of twists}

In this subsection, we recall and slightly extend a result from \cite{CKP} explaining how to categorify Muller--Speyer's left twist automorphism \(\ltwist\colon\CC[\openposvarcone_\posit]\to\CC[\openposvarcone_\posit]\).

\begin{thm}[{\cite[Thm.~12.2]{CKP}}]
\label{t:MS=CC}
Let \(M_I\in\CM(B)\) be a rank \(1\) module, and let
\[\begin{tikzcd}
0\arrow{r}&\bOmega M_I\arrow{r}&PM_I\arrow{r}&M_I\arrow{r}&0
\end{tikzcd}\]
be a short exact sequence in which \(PM_I\to M_I\) is a projective cover. Then
\[\ltwist\Plueck{I}=\frac{\CCsrc(\bOmega M_I)}{\CCsrc(PM_I)}.\]
\end{thm}

If \(M_I\in\gproj\CM(B)\), then we may use Theorem~\ref{t:CC-values} to rewrite the identity as
\[\ltwist\CCsrc(M_I)=\frac{\CCsrc(\bOmega M_I)}{\CCsrc(PM_I)},\]
cf.~Theorem~\ref{t:twist} below.
As a further special case, if \(I\in\srcneck_D\) is an element of the source necklace, so that \(M_I\in\projcat(B)\), then this identity reduces to \(\ltwist\Plueck{I}=\Plueck{I}^{-1}\) (as used in the proof of Proposition~\ref{p:sourceneck} to see that \(\ltwist^2\Plueck{I}=\Plueck{I}\) in this case).
Since \(\ltwist\) is a ring homomorphism and the cluster character is multiplicative on direct sums \cite[Thm.~3.3(c)]{FuKel}, it follows that
\begin{equation}
\label{eq:proj-twist}
\ltwist\CCsrc(P)=\CCsrc(P)^{-1}
\end{equation}
for all \(P\in\projcat(B)\).

\begin{thm}
\label{t:twist}
Let \(X\) be a reachable rigid object of \(\gproj\CM(B)\), and let
\[\begin{tikzcd}
0\arrow{r}&\bOmega X\arrow{r}&PX\arrow{r}&X\arrow{r}&0
\end{tikzcd}\]
be a short exact sequence in which \(PX\to X\) is a projective cover. Then
\begin{equation}
\label{eq:twist}
\ltwist\CCsrc(X) = \frac{\CCsrc(\bOmega X)}{\CCsrc(PX)}.
\end{equation}
\end{thm}
\begin{proof}
We show that the relevant identity is preserved under mutation of cluster-tilting objects.
Thus, the result will follow by induction, with Theorem~\ref{t:MS=CC} as a base case, using the fact (Theorem~\ref{t:categorify}) that \(\Tsrc=\bigoplus_{j\in Q_0}M(\srclab{j})\in\gproj\CM(B)\) is a cluster-tilting object all of whose indecomposable summands have rank \(1\).
Note that Theorem~\ref{t:MS=CC} extends directly to all \(X\in\add(\Tsrc)\) using the fact that the cluster character is multiplicative on direct sums.

To this end, let \(T\in\gproj\CM(B)\) be a cluster-tilting object, and assume that \eqref{eq:twist} holds for any \(X\in\add(T)\).
If \(U\in\add(T)\) is indecomposable and non-projective, we may consider the two exchange sequences
\[\begin{tikzcd}0\arrow{r}&U\arrow{r}&L\arrow{r}&U^*\arrow{r}&0,\end{tikzcd}\quad
\begin{tikzcd}0\arrow{r}&U^*\arrow{r}&R\arrow{r}&U\arrow{r}&0.\end{tikzcd}\]
which compute the mutation \(T^*=(T/U)\dsum U^*\) of \(T\) at \(X\).
By the multiplication formula for cluster characters, we have the exchange relation
\begin{equation}
\label{eq:mult-form1}
\CCsrc(U^*)=\frac{\CCsrc(L)+\CCsrc(R)}{\CCsrc(U)}.
\end{equation}

Now choose projective covers \(PU\to U\) and \(PU^*\to U^*\).
By the horseshoe lemma, we may construct diagrams
\begin{center}
$\begin{tikzcd}[column sep=1.1em,row sep=1.4em]
&0\arrow{d}&0\arrow{d}&0\arrow{d}\\
0\arrow{r}&\bOmega U\arrow{r}\arrow{d}&\bOmega L\arrow{r}\arrow{d}&\bOmega U^*\arrow{r}\arrow{d}&0\\
0\arrow{r}&PU\arrow{r}\arrow{d}&PX\dsum PU^*\arrow{r}\arrow{d}&PU^*\arrow{r}\arrow{d}&0\\
0\arrow{r}&U\arrow{r}\arrow{d}&L\arrow{r}\arrow{d}&U^*\arrow{r}\arrow{d}&0\\
&0&0&0
\end{tikzcd}$
\hfill
$\begin{tikzcd}[column sep=1.1em,row sep=1.4em]
&0\arrow{d}&0\arrow{d}&0\arrow{d}\\
0\arrow{r}&\bOmega U^*\arrow{r}\arrow{d}&\bOmega R\arrow{r}\arrow{d}&\bOmega U\arrow{r}\arrow{d}&0\\
0\arrow{r}&PU^*\arrow{r}\arrow{d}&PX^*\dsum PU\arrow{r}\arrow{d}&PX\arrow{r}\arrow{d}&0\\
0\arrow{r}&U^*\arrow{r}\arrow{d}&R\arrow{r}\arrow{d}&U\arrow{r}\arrow{d}&0\\
&0&0&0\end{tikzcd}$
\end{center}
with exact rows and columns.
Note that the syzygies \(\bOmega L\) and \(\bOmega R\) may have projective summands, even if \(PU\to U\) and \(PU^*\to U^*\) were chosen to be minimal.

Since the syzygy induces the inverse suspension functor on \(\stabgproj\CM(B)\), we find that
\[\Ext^1_B(V,W)=\Ext^1_B(\bOmega V,\bOmega W)\]
for any \(V,W\in\gproj\CM(B)\).
In particular, both \(\bOmega U\dsum \bOmega L\) and \(\bOmega U\dsum \bOmega R\) are, like \(U\dsum L\) and \(U\dsum R\), rigid, whereas
\[\dim\Ext^1_B(\bOmega U,\bOmega U^*)=\dim\Ext^1_B(U,U^*)=1.\]
This means in particular that the upper rows of the two diagrams are not split.
Thus, the multiplication formula also applies to these two rows, to give
\begin{equation}
\label{eq:mult-form2}
\CCsrc(\bOmega U^*)=\frac{\CCsrc(\bOmega L)+\CCsrc(\bOmega R)}{\CCsrc(\bOmega U)}.
\end{equation}
By assumption, we have
\begin{align*}
\ltwist\CCsrc(U)&=\frac{\CCsrc(\bOmega U)}{\CCsrc(PU)},\\
\ltwist\CCsrc(L)&=\frac{\CCsrc(\bOmega L)}{\CCsrc(PU)\CCsrc(PU^*)},\\
\ltwist\CCsrc(R)&=\frac{\CCsrc(\bOmega R)}{\CCsrc(PU)\CCsrc(PU^*)},
\end{align*}
since \(U\), \(L\) and \(R\) are all objects of \(\add(T)\).
Applying the twist to \eqref{eq:mult-form1} and using these identities, we find that
\[
\ltwist\CCsrc(U^*)=\frac{\ltwist\CCsrc(L)+\ltwist\CCsrc(R)}{\ltwist\CCsrc(U)}=\frac{\CCsrc(\bOmega L)+\CCsrc(\bOmega R)}{\CCsrc(\bOmega U)\CCsrc(PU^*)}=\frac{\CCsrc(\bOmega U^*)}{\CCsrc(PU^*)}
\]
by \eqref{eq:mult-form2}.
Since \(U^*\) is the unique indecomposable object of \(\add(T^*)\) not contained in \(\add(T)\), and the cluster character is multiplicative on direct sums, \eqref{eq:twist} holds for all \(X\in\add(T^*)\).
\end{proof}

\subsection{Categorification of quasi-cluster morphisms}
Finally, we recall a key result of Fraser and Keller \cite[Thm.~A.7]{KelWu}, which allows us to check that a particular map of algebras is a quasi-cluster equivalence by using categorifications of the relevant cluster algebras.

Let \(\cat{E}\) be an idempotent complete (also known as Karoubian) Frobenius exact category.
Despite the fact that \(\cat{E}\) may not be abelian, the idempotent completeness property means that it admits a bounded derived category \(\bdcat(\cat{E})\) obtained via the usual construction \cite{ThoTro}.
As a result, if \(T\in\cat{E}\) is an object, the inclusion \(\add(T)\to\cat{E}\) induces a canonical triangle functor \(\hcat[\bdd](\add{T})\to\bdcat(\cat{E})\) from the category of bounded complexes of objects in \(\add(T)\), with morphisms considered up to homotopy, to the bounded derived category of \(\cat{E}\).

Denote by \(\cat{P}\subseteq\cat{E}\) the full subcategory of projective-injective objects, and by \(\hcat[\bdd](\cat{P})\) its bounded homotopy category.
Then there is a Verdier localisation functor
\[\bdcat(\cat{E})\to\bdcat(\cat{E})/\hcat[\bdd](\cat{P})\eqdef\singcat(\cat{E})\simeq\stab{\cat{E}}.\]
Here the notation \(\singcat(\cat{E})\) refers to the \emph{singularity category} of \(\cat{E}\), defined by Orlov \cite{Orlov-SingCat} as this Verdier quotient, which is equivalent to the stable category \(\stab{\cat{E}}\defeq \cat{E}/\cat{P}\) by Buchweitz's famous result \cite[Thm.~4.4.1]{Buchweitz-MCM}.

Finally, assume that \(\cat{E}\) is stably \(2\)-Calabi--Yau (i.e.\ that \(\stab{\cat{E}}\) is a \(2\)-Calabi--Yau triangulated category), and that there is a cluster character \(\Psi\colon\cat{E}\to\clust{A}^+\), where \(\clust{A}^+\) is an upper cluster algebra with invertible frozen variables. For example, this happens if \((\cat{E},T)\) is a Frobenius \(2\)-CY realisation of some ordinary cluster algebra \(\clust{A}\), on which Fu--Keller's cluster character takes values in the upper cluster algebra \(\clust{A}^+\) by \cite[Thm.~1.3]{Plamondon-Generic}.

\begin{thm}[{\cite[Thm.~A.4]{KelWu}}]
\label{t:CC-lift}
Let \(\Psi\colon\cat{E}\to\clust{A}^+\) be a cluster character, where \(\clust{A}^+\) is an upper cluster algebra with invertible frozen variables, inducing a cluster character \(\underline{\Psi}\colon\stab{\cat{E}}\to\stab{\clust{A}}^+\) on the stable category. Then there is a unique function \(\Psi\colon\bdcat(\cat{E})\to\clust{A}^+\) such that
\begin{enumerate}
\item the diagram
\[\begin{tikzcd}
\cat{E}\arrow{r}{\Psi}\arrow{d}&\clust{A}^+\arrow[equal]{d}\\
\bdcat(\cat{E})\arrow{d}\arrow{r}{\Psi}&\clust{A}^+\arrow{d}\\
\stab{\cat{E}}\arrow{r}{\stab{\Psi}}&\stab{\clust{A}}^+,
\end{tikzcd}\]
commutes (justifying the abuse of notation), and
\item for any triangle
\[\begin{tikzcd}
P\arrow{r}&X\arrow{r}&Y\arrow{r}&P[1]
\end{tikzcd}\]
in \(\bdcat(\cat{E})\) with \(P\in\hcat[\bdd](\cat{P})\), we have \(\Psi(X)=\Psi(P)\Psi(Y)\).
\end{enumerate}
\end{thm}

\begin{thm}[{\cite[Thm.~A.7]{KelWu}}]
\label{t:FKW}
Let \((\cat{E},T)\) and \((\cat{F},U)\) be Frobenius \(2\)-CY realisations of cluster algebras \(\clust{A}\) and \(\clust{B}\) respectively, with cluster characters \(\Phi\colon\cat{E}\to\clust{A}^+\) and \(\Psi\colon\cat{F}\to\clust{B}^+\) taking values in the associated upper cluster algebras.
Let \(\eta\colon\clust{A}^+\isoto R\) and \(\nu\colon\clust{B}^+\isoto S\) be upper cluster structures, and let \(f\colon R\to S\) be a ring homomorphism.
Assume we have a commutative diagram
\begin{equation}
\label{eq:qcm-diagram}
\begin{tikzcd}
\hcat[\bdd](\add{T})\arrow{r}\arrow{d}{\widetilde{\varphi}}&\dcat[\bdd](\cat{E})\arrow{d}{\varphi}\arrow{r}&\stab{\cat{E}}\arrow{d}{\stab{\varphi}}\\
\hcat[\bdd](\add{U})\arrow{r}&\dcat[\bdd](\cat{F})\arrow{r}&\stab{\cat{F}}
\end{tikzcd}
\end{equation}
of triangle functors, in which the horizontal arrows are the canonical functors, and further
that the diagram
\begin{equation}
\label{eq:cc-diagram}
\begin{tikzcd}
\add(T)\arrow{r}{\eta\circ \Phi}\arrow{d}{\varphi}&R\arrow{d}{f}\\
\bdcat(\cat{F})\arrow{r}{\nu\circ \Psi}&S.
\end{tikzcd}
\end{equation}
commutes.
If \(\stab{\varphi}\) is a triangle equivalence and \(\stab{\varphi}T\) is mutation equivalent to \(U\) in \(\stab{\cat{F}}\), then \(f\) is a quasi-cluster morphism.
\end{thm}

Our goal in the remainder of the paper is to apply Theorem~\ref{t:FKW} to the Frobenius \(2\)-CY realisations \((\ginj{\CM(B)},\Ttgt)\) and \((\gproj{\CM(B)},\Tsrc)\) of \(\clust{A}_D\) and the (upper) cluster structures \(\GLtgt\colon\clust{A}_D\isoto\CC[\openposvarcone_\posit]\) and \(\GLsrc\colon\clust{A}_D\isoto\CC[\openposvarcone_\posit]\), taking \(f\) to be the identity map on this coordinate ring.
This will complete the proof of the quasi-coincidence conjecture.

\section{Proof of the quasi-coincidence conjecture}
\label{s:proof}

We fix a connected plabic graph \(D\) of type \((k,n)\) for the remainder of the paper and abbreviate \(A=A_D\), \(B=B_D\), and \(C=C_{k,n}\).
As in Section~\ref{s:categorification-posit}, we view the two associated categories \(\gproj\CM(B)\) and \(\ginj\CM(B)\) as full subcategories of \(\CM(C)\), via the fully faithful functor from Proposition~\ref{p:embed}.
Since the strategy of the proof is to apply Fraser--Keller's Theorem~\ref{t:FKW}, the bulk of the section is devoted to proving that the hypotheses of this theorem hold in our situation.

\begin{prop}
\label{p:der-eq}
The categories \(\bdcat(\gproj\CM(B))\), \(\bdcat(\ginj\CM(B))\) and \(\bdcat(\CM(B))\) are all tautologically equivalent, in the sense that the natural functor from each to \(\bdcat(\fgmod B)\) is an equivalence.
\end{prop}
\begin{proof}
We have a diagram of fully faithful exact functors
\[\begin{tikzcd}[row sep=0pt]
\ginj\CM(B)\arrow{dr}\\
&\CM(B)\arrow{r}&\fgmod{B}\\
\gproj\CM(B)\arrow{ur}
\end{tikzcd}\]
of exact categories, each of which embeds its source into its target as an extension-closed subcategory \cite[Prop.~7.2]{Pressland-Postnikov}.
When passing to derived categories, each of these functors induces a triangle equivalence: indeed, because \(B\) is Iwanaga--Gorenstein, every Cohen--Macaulay \(B\)-module has a finite resolution by objects of \(\gproj\CM(B)\) and a finite coresolution by objects of \(\ginj\CM(B)\).
Recall that objects of \(\CM(B)\) are those \(B\)-modules which are free and finitely generated over the central subalgebra \(Z\iso\powser{\CC}{t}\), which is Noetherian of global dimension \(1\).
Thus, \(\CM(B)\) is closed under subobjects and contains \(B\), so every finitely generated \(B\)-module has a finite resolution in \(\CM(B)\) (for example, by taking a projective cover and its kernel).
\end{proof}
\begin{cor}
\label{c:stab-equiv}
The stable categories \(\stabgproj\CM(B)\) and \(\stabginj\CM(B)\) are canonically triangle equivalent to the singularity category \(\singcat(B)=\bdcat(\fgmod{B})/\per{B}\).
\end{cor}
\begin{proof}
This combines Proposition~\ref{p:der-eq} with Buchweitz's theorem \cite[Thm.~4.4.1]{Buchweitz-MCM}.
\end{proof}

The previous two results yield a triangle equivalence \(\stab{\varphi}\colon\stabginj\CM(B)\isoto\stabgproj\CM(B)\) which lifts to a functor \(\varphi\) (indeed, to the tautological equivalence) between the derived categories \(\bdcat(\ginj\CM(B))=\bdcat(\fgmod B)=\bdcat(\gproj\CM(B))\).
This is the data of the right-hand commuting square in \eqref{eq:qcm-diagram} from Theorem~\ref{t:FKW}. In what follows, it will be convenient to use \(\varphi\) to denote the equivalence \(\bdcat(\fgmod{B})\isoto\bdcat(\gproj\CM(B))\), as well as its restrictions (along the equivalences induced by inclusions of exact categories as in Proposition~\ref{p:der-eq}) to equivalences \(\bdcat(\CM(B))\isoto\bdcat(\gproj\CM(B))\) and \(\bdcat(\ginj\CM(B))\isoto\bdcat(\gproj\CM(B))\).

Next we turn our attention to the hypothesis that \(\stab{\varphi}(\Ttgt)\) is mutation equivalent to \(\Tsrc\) in \(\stabgproj\CM(B)\), or equivalently that \(\Tsrc\) and \(\Ttgt\) are mutation equivalent in \(\singcat(B)\).

\begin{prop}
\label{p:mut-shift}
The mutation classes of \(\Tsrc\) and \(\Ttgt\) in \(\singcat(B)\) are closed under the suspension functor \(\Sigma\).
\end{prop}
\begin{proof}
Recall from Theorems~\ref{t:src-cat} and \ref{t:tgt-cat} that we have isomorphisms \(\Endalg{B}{\Tsrc}\iso A_D\iso\Endalg{B}{\Ttgt}\), where \(A_D\) is the dimer algebra of the plabic graph \(D\).
In particular, the Gabriel quiver of each endomorphism algebra coincides with the quiver \(Q_D\) (up to removable \(2\)-cycles, if \(D\) has bivalent nodes).

By a result of Ford and Serhiyenko \cite[Thm.~1.2]{ForSer}, the quiver \(\stab{Q}_D\) admits a green-to-red sequence.
By \cite[Prop.~5.17]{Keller-QDilog}, this sequence may then be interpreted categorically as a sequence of mutations from \(\Tsrc\) to \(\Sigma^{-1}\Tsrc\) in \(\stabgproj\CM(B)=\singcat(B)\), or from \(\Ttgt\) to \(\Sigma^{-1}\Ttgt\) in \(\stabginj\CM(B)=\singcat(B)\).
This fact may also be deduced by combining \cite[Prop.~2.10(2)]{BDP} with Nakanishi and Zelevinsky's tropical duality \cite[Thm.~1.2]{NakZel}, the interpretation of g-vectors in terms of (co)indices \cite[Prop.~4.3 or 6.2]{FuKel}, and Proposition~\ref{p:shift} below.

Since all shifts of \(\Tsrc\) and \(\Ttgt\) in \(\singcat(B)\) have isomorphic endomorphism algebras, the same argument applies to show that \(\Sigma^n\Tsrc\) is mutation equivalent to \(\Sigma^{n-1}\Tsrc\), and \(\Sigma^n\Ttgt\) to \(\Sigma^{n-1}\Ttgt\), for any \(n\in\ZZ\), which gives the result.
\end{proof}

In combination with Proposition~\ref{p:mut-shift}, the next result will imply that \(\Ttgt\) and \(\Tsrc\) are in the same mutation class in \(\singcat(B)\), as required.
In fact, it will turn out to be the key statement in establishing all the remaining hypotheses of Theorem~\ref{t:FKW}.

\begin{thm}
\label{t:sigma2}
Let \(j\in Q_D\) be a vertex. Then in \(\singcat(B)\) we have \(\Sigma^2eAe_j\iso\Zdual{(e_jAe)}\), and hence \(\Sigma^2\Tsrc\iso\Ttgt\).
\end{thm}

Recalling that \(\Tsrc=eA\) and \(\Ttgt=\Zdual{(Ae)}\) for \(A=A_D\) and \(e\) the boundary idempotent, the second isomorphism in Theorem~\ref{t:sigma2} follows immediately from the first by taking the direct sum over vertices.

\begin{rem}
As in Theorems~\ref{t:src-cat} and \ref{t:tgt-cat}, we have \(eAe_j=M(\srclab{j})\) and \(\Zdual{(e_jAe)}=M(\tgtlab{j})\) in \(\CM(C)\).
Given the relationship between Muller--Speyer's twist map and the suspension functor on \(\singcat(B)\) explained in Theorem~\ref{t:MS=CC}, we may view Theorem~\ref{t:sigma2} as a categorical version of \cite[Prop.~7.14]{MulSpe-Twist}.
\end{rem}

We prove Theorem~\ref{t:sigma2} in several steps, beginning with some generalities.
Recall (e.g.\ from \cite[\S4]{KelRei-StabCY}) that in a stably \(2\)-Calabi--Yau Frobenius exact category \(\cat{E}\) with cluster-tilting object \(T\), any object \(X\in\cat{E}\) fits into short exact sequences
\begin{equation}
\label{eq:ind-coind}
\begin{tikzcd}[row sep=0pt]
0\arrow{r}&T^1\arrow{r}&T^0\arrow{r}&X\arrow{r}&0,\\
0\arrow{r}&X\arrow{r}&T_0\arrow{r}&T_1\arrow{r}&0
\end{tikzcd}
\end{equation}
with \(T^i,T_i\in\add(T)\).
We write \(F=\Hom_{\cat{E}}(T,\blank)\) for the Yoneda equivalence \(\add(T)\isoto\projcat{A}\), where \(A=\Endalg{\cat{E}}{T}\).
For \(\stab{A}=\stabEndalg{\cat{E}}{T}\) the endomorphism algebra of \(T\) in the stable category \(\stab{\cat{E}}\), we may consider the projection \(\Kgp_0(\projcat{A})\to\Kgp_0(\projcat{\stab{A}})\) with kernel generated by the classes \([FP]\) for a \(P\) a projective object of \(\cat{E}\), as in Section~\ref{s:categorification-pms}.

We define the \emph{index} of \(X\in\cat{E}\) by \(\ind(X)=[FT^0]-[FT^1]\in\Kgp_0(\projcat{\stab{A}})\) and the \emph{coindex} by \(\coind(X)=[FT_0]-[FT_1]\in\Kgp_0(\projcat{\stab{A}})\).
While the individual objects \(T^i\) and \(T_i\) can depend on the choice of sequences \eqref{eq:ind-coind}, the index and coindex do not.
Because we project the index and coindex to \(\Kgp_0(\projcat{\stab{A}})\), they can also be computed by choosing triangles
\[\begin{tikzcd}[row sep=0pt]
T^1\arrow{r}&T^0\arrow{r}&X\arrow{r}&\Sigma T^1,\\
X\arrow{r}&T_0\arrow{r}&T_1\arrow{r}&\Sigma X
\end{tikzcd}\]
in the stable category \(\stab{\cat{E}}\) with \(T^i,T_i\in\add(T)\), and taking \(\ind(X)=[FT^0]-[FT^1]\) and \(\coind(X)=[FT_0]-[FT_1]\).
Note here that \(\cat{E}\) and \(\stab{\cat{E}}\) have the same objects, so it makes sense to apply \(F\) to an object of \(\stab{\cat{E}}\) (but not to a morphism).
By this second description, we see that our definitions are compatible with those of Palu \cite[\S2.1]{Palu-CC} for the \(2\)-Calabi--Yau triangulated category \(\stab{\cat{E}}\) (although Palu does not apply the Yoneda isomorphism, preferring to work in \(\add(T)\)).

\begin{prop}
\label{p:shift}
Let \(U\in\add(T)\).
Then if \(X\in\cat{E}\) is rigid, we have \(\ind(X)=-[FU]\) if and only if \(X\iso\Sigma U\) in \(\stab{\cat{E}}\), and \(\coind(X)=-[FU]\) if and only if \(X\iso\Sigma^{-1}U\) in \(\stab{\cat{E}}\).
\end{prop}
\begin{proof}
Firstly, the existence of the triangle
\[\begin{tikzcd}
U\arrow{r}&0\arrow{r}&\Sigma U\arrow{r}&\Sigma U
\end{tikzcd}\]
in \(\stab{\cat{E}}\), in which \(0,U\in\add(T)\), shows that \(\ind(\Sigma U)=-[FU]\).
Note also that \(\Sigma U\) is rigid, because \(T\) and hence \(U\) is.
Now by \cite[Thm.~2.3]{DehKel}, rigid objects in \(\stab{\cat{E}}\) are determined by their indices, and so up to isomorphism \(\Sigma U\) is the only rigid object of \(\stab{\cat{E}}\) with index \(-[FU]\).
The statement for the coindex may be shown completely analogously.
\end{proof}

\begin{prop}
\label{p:ind-coind}
For any \(X\in\cat{E}\), we have \(\ind(X)=[FX]\) and \(\coind(X)=[FX]-[GX]\) in \(\Kgp_0(\projcat{\stab{A}})\).
\end{prop}
\begin{proof}
Since \(T\) is cluster-tilting, and hence rigid, we have \(GT=0\).
Thus, the upper sequence in \eqref{eq:ind-coind} remains exact under \(F\), and \(\ind(X)=[FT^0]-[FT^1]=[FX]\).
On the other hand, applying \(F\) to the lower sequence in \eqref{eq:ind-coind} yields
\[\begin{tikzcd}
0\arrow{r}&FX\arrow{r}&FT_0\arrow{r}&FT_1\arrow{r}&GX\arrow{r}&0,
\end{tikzcd}\]
and so \(\coind(X)=[FT_0]-[FT_1]=[FX]-[GX]\).
\end{proof}

When \(\cat{E}\) satisfies the necessary assumptions to admit a Fu--Keller cluster character \(\Phi\) as in Section~\ref{s:categorification-pms} (for example, if \(\cat{E}=\gproj\CM(B)\)), Proposition~\ref{p:ind-coind} shows in particular that the exponents of the two extremal terms of \(\Phi(X)\), indexed by \(0\) and \(GX\) respectively, are \(\ind(X)\) and \(\coind(X)\).

We now apply these considerations to \(\cat{E}=\gproj\CM(B)\) and \(T=\Tsrc=eA\) for \(A\) the dimer algebra of \(D\) and \(e\) the boundary idempotent.
Recall from Theorem~\ref{t:src-cat} that there are natural isomorphisms \(A=\Endalg{B}{\Tsrc}\) and \(A/AeA=\stab{A}=\stabEndalg{B}{\Tsrc}\), and from Section~\ref{s:categorification-pms} (specifically \eqref{eq:F'}) that there is a short exact sequence
\[\begin{tikzcd}
0\arrow{r}&F'X\arrow{r}&FX\arrow{r}&G\bOmega X\arrow{r}&0
\end{tikzcd}\]
for any \(X\in\CM(B)\) and any syzygy \(\bOmega X\) of \(X\).

\begin{cor}
\label{c:CM-ind}
For \(X\in\CM(B)\) and \(\bOmega X\) an arbitrary syzygy, we have \(\ind(\bOmega X)=-[F'X]\) and \(\coind(\bOmega X)=-[FX]\) in \(\Kgp_0(\projcat{\stab{A}})\).
\end{cor}
\begin{proof}
This combines Propositions~\ref{p:g-vec} and \ref{p:ind-coind}.
\end{proof}

Recall Muller--Speyer's matchings \(\MSsrc{j}\) and \(\MStgt{j}\) from Proposition~\ref{p:MSmats}, defined in terms of downstream and upstream wedges.
As explained in Theorem~\ref{t:MSmats}, the perfect matching module \(N(\MSsrc{j})\) is isomorphic to the indecomposable projective \(A\)-module \(Ae_j\), while \(N(\MStgt{j})\) is isomorphic to \(\Zdual{(e_jA)}\), which is indecomposable injective in \(\CM(A)\).
By the double centraliser property \(A=\Endalg{B}{eA}\) from Theorem~\ref{t:src-cat}, together with Corollary~\ref{c:bdry-value}, we have \(Ae_j=F(eAe_j)\iso FM(\srclab{j})\), and the next lemma gives similar descriptions for the injective object \(\Zdual{(e_jA)}\).

\begin{lem}
\label{l:A-inj}
For each \(j\in Q_0\), we have \(N(\MStgt{j})\iso\Zdual{(e_jA)}\iso F'(\Zdual{(e_jAe)})\iso F'M(\tgtlab{j})\).
\end{lem}
\begin{proof}
The first isomorphism is Theorem~\ref{t:MSmats} and the third is Corollary~\ref{c:bdry-value}, so it remains to prove the second.
Recall from Proposition~\ref{p:Zdual} that the duality functor \(\Zdual{(\blank)}\) provides equivalences \(\CM(B)^\op\isoto\CM(B^{\op})\) and \(\CM(A)^\op\isoto\CM(A^{\op})\).
Moreover, there is a commutative diagram
\[\begin{tikzcd}
\CM(A)^\op\arrow{r}{e}\arrow{d}{\Zdual{(\blank)}}&\CM(B)^\op\arrow{d}{\Zdual{(\blank)}}\\
\CM(A^\op)\arrow{r}{e}&\CM(B^\op)
\end{tikzcd}\]
where we abuse notation by denoting restriction to the boundary algebra by \(e\) in both cases.
Commutativity follows from tensor--Hom adjunction, since for \(N\in\CM(A)\) we have
\[e(\Zdual{N})=\Hom_{A^{\op}}(eA,\Hom_Z(N,Z))=\Hom_Z(eA\otimes_AN,Z)=\Zdual{(eN)}.\]

As observed in \cite[\S5]{CKP}, the functor \(F'\colon\CM(B)\to\CM(A)\) is left adjoint to \(e\colon\CM(A)\to\CM(B)\).
Write \(F^\op\) for the right adjoint of \(e\colon\CM(A^\op)\to\CM(B^\op)\), where the notation refers to the fact that \(F^\op=\Hom_{B^\op}(eA^\op,\blank)\) is the analogue of the functor \(F\) for the algebras \(A^\op\) and \(B^\op\).

Now, by uniqueness of adjoints, we have \(\Zdual{\smash{(F^\op X)}}=F'(\Zdual{X})\) for any \(X\in\CM(B^\op)\); the interchanging of left and right adjoints is a consequence of the contravariance of \(\Zdual{(\blank)}\).
Taking \(X=e_jAe\), and noting that \(F^\op(e_jAe)=e_jA\) by applying Theorem~\ref{t:src-cat} to \(D^\op\), the result follows.
\end{proof}

\begin{lem}
\label{l:upmat-class}
Let \(j\in Q_0\). Then \([N(\MStgt{j})]=[Ae_j]\) after projection to \(\Kgp_0(\projcat{\stab{A}})\).
\end{lem}
\begin{proof}
We exploit the fact that a combinatorial calculation involving perfect matchings has two interpretations, one for \(D\) and one for the opposite diagram \(D^\op\).
Let \(\mu\) be a perfect matching of \(Q\) (and therefore also a perfect matching of \(Q^\op\)), and let \(N_\mu(Q)\) be the corresponding perfect matching \(A\)-module.
Then by \eqref{eq:white} from Proposition~\ref{p:mat-res}, we have
\begin{equation}
\label{eq:white}
[N_{\mu}(Q)]=\sum_{\substack{j\in Q_0\\\text{int}}}[Ae_j]-\wt[\white](\mu)\in\Kgp_0(\projcat{\stab{A}}),
\end{equation}
where the sum is over internal, i.e.\ mutable, vertices of \(Q\).

A direct comparison to the construction in \cite[Def.~4.3]{CKP} shows that \(\Zdual{N_\mu(Q)}\) is (canonically isomorphic to) the perfect matching \(A^{\op}\)-module \(N_\mu(Q^\op)\) associated to \(\mu\) when viewing it as a perfect matching of \(Q^{\op}\).
Applying Proposition~\ref{p:mat-res} again, but this time using \eqref{eq:black}, we therefore have
\begin{equation}
\label{eq:black}
[\Zdual{N_\mu(Q)}]=\sum_{\substack{j\in Q_0\\\text{int}}}[e_jA]-\wt[\black](\mu)\in\Kgp_0(\projcat{\stab{A}^\op}).
\end{equation}

There is a lattice isomorphism \(\Kgp_0(\projcat{A})\isoto\Kgp_0(\projcat{A^\op})\), projecting to a lattice isomorphism \(\Kgp_0(\projcat{\stab{A}})\isoto\Kgp_0(\projcat{\stab{A}^\op})\), given by identifying \([P_j]=[Ae_j]\mapsto[e_jA]\); this amounts to identifying each lattice with \(\ZZ^{Q_0}\) using the basis of projective modules, and its indexing by the common set of quiver vertices.
Because a face is positively oriented in \(Q\) if and only if it is negatively oriented in \(Q^\op\), for any perfect matching \(\mu\) this isomorphism identifies the value \(\wt[\white](\mu)\) computed on \(Q\) with the value \(\wt[\black](\mu)\) computed on \(Q^{\op}\).
It thus follows from \eqref{eq:white} and \eqref{eq:black} that it also identifies \([N_\mu(Q)]\) with \([\Zdual{N_\mu(Q)}]\).

Since \(\MStgt{j}(D^\op)=\MSsrc{j}(D)\) by Proposition~\ref{p:op}\ref{p:op-MS}, we have \(\Zdual{\smash{N(\MStgt{j})}}\iso e_jA\) by \cite[Cor.~7.6]{CKP}, and hence \([\Zdual{\smash{N(\MStgt{j})}}]=[e_jA]\) in \(\Kgp_0(\projcat{\stab{A}^\op})\).
The corresponding identity in \(\Kgp_0(\projcat{\stab{A}})\) is then \([N(\MStgt{j})]=[Ae_j]\), as required.
\end{proof}

\begin{rem}
We emphasise that \([N(\MStgt{j})]\ne[Ae_j]\) in \(\Kgp_0(\projcat{A})\), before projecting to the Grothendieck group for the stable endomorphism algebra; indeed, this would contradict \cite[Cor.~6.16]{CKP} by implying that \(\MStgt{j}=\MSsrc{j}\).
To obtain correct identities in \(\Kgp_0(\projcat{A})\) and \(\Kgp_0(\projcat{A}^\op)\), the formulae \eqref{eq:white} and \eqref{eq:black} must be modified by adding additional terms from the kernels of the projections as in \cite[Prop.~6.11]{CKP}, and these additional terms are not related by our naive isomorphism \(\Kgp_0(\projcat{A})\iso\Kgp_0(\projcat{A^\op})\).
\end{rem}

\begin{proof}[Proof of Theorem~\ref{t:sigma2}]
Combining Lemmas~\ref{l:A-inj} and \ref{l:upmat-class}, we find that
\[[F'(\Zdual{(e_jAe)})]=[Ae_j]\]
in \(\Kgp_0(\projcat{\stab{A}})\), and so \(\ind(\bOmega\Zdual{(e_jAe)})=-[Ae_j]\) by Corollary~\ref{c:CM-ind}.
But \(Ae_j=F(eAe_j)\) and \(eAe_j\in\add(\Tsrc)\), so \(\bOmega\Zdual{(e_jAe)}\iso\Sigma(eAe_j)\) in \(\singcat(B)\) by Proposition~\ref{p:shift}.
Since \(\bOmega\Zdual{(e_jAe)}=\Sigma^{-1}\Zdual{(e_jAe)}\) in \(\singcat(B)\), this completes the proof.
\end{proof}

Theorem~\ref{t:sigma2} is also the main step in lifting the tautological equivalence \(\bdcat(\ginj\CM{B})=\bdcat(\gproj\CM{B})\) to a functor between the appropriate homotopy categories, in order to obtain the left-hand commuting square in \eqref{eq:qcm-diagram} from Theorem~\ref{t:FKW}.

\begin{prop}
\label{p:homot-lift}
There is a functor \(\widetilde{\varphi}\colon\hcat[\bdd](\add{\Ttgt})\to\hcat[\bdd](\add{\Tsrc})\) such that the diagram
\[\begin{tikzcd}
\hcat[\bdd](\add{\Ttgt})\arrow{d}{\widetilde{\varphi}}\arrow{r}&\bdcat(\ginj\CM(B))\arrow{d}{\varphi}\\
\hcat[\bdd](\add{\Tsrc})\arrow{r}&\bdcat(\gproj\CM(B))
\end{tikzcd}\]
commutes, where the horizontal arrows are the canonical functors.
\end{prop}
\begin{proof}
By the construction of \(\varphi\), the right-hand square in the extended diagram
\begin{equation}
\label{eq:ext-diag}
\begin{tikzcd}
\hcat[\bdd](\add{\Ttgt})\arrow{d}{\widetilde{\varphi}}\arrow{r}&\bdcat(\ginj\CM(B))\arrow{d}{\varphi}\arrow{r}{\sim}&\bdcat(\fgmod{B})\arrow[equal]{d}\\
\hcat[\bdd](\add{\Tsrc})\arrow{r}&\bdcat(\gproj\CM(B))\arrow{r}{\sim}&\bdcat(\fgmod{B})
\end{tikzcd}
\end{equation}
commutes, and so, given that the horizontal maps in this right-hand square are equivalences, it is sufficient to provide a functor $\widetilde{\varphi}$ such that the outer square of \eqref{eq:ext-diag} commutes.

It follows from Theorem~\ref{t:sigma2} that there is an exact sequence
\[\begin{tikzcd}
0\arrow{r}&\Tsrc\arrow{r}&P_1\arrow{r}&P_0\arrow{r}&\Ttgt\arrow{r}&0
\end{tikzcd}\]
in \(\CM(B)\), where \(P_0,P_1\in\projcat{B}\).
In particular, the complex
\[\begin{tikzcd}
\xi\colon\Tsrc\arrow{r}&P_1\arrow{r}&P_0
\end{tikzcd}\]
in \(\bdcat(\gproj\CM(B))\), extended by zeros and with \(P_0\) in degree \(0\), is quasi-isomorphic to \(\Ttgt\) when viewed as an object of \(\bdcat(\fgmod{B})\), so it suffices to construct a functor \(\widetilde{\varphi}\colon\hcat[\bdd](\add{\Ttgt})\to\hcat[\bdd](\add{\Tsrc})\) such that \(\widetilde{\varphi}(\Ttgt)=\xi\).
We will deduce the existence of such a functor from \cite[Thm.~6.4]{Rickard-Morita} (see also \cite{Keller-ChainComp}) by showing that \(\xi\) is a tilting object in \(\hcat[\bdd](\add\Tsrc)\) with endomorphism algebra isomorphic to \(\Endalg{B}{\Ttgt}=A\).
(Strictly speaking, this proves more than is necessary, since it will also imply that the functor \(\widetilde{\varphi}\) thus constructed is a triangle equivalence.)

We first calculate the endomorphism algebra of \(\xi\), and begin by applying the equivalence \(F=\Hom_B(\Tsrc,\blank)\colon\hcat[\bdd](\add{\Tsrc})\to\hcat[\bdd](\projcat{A})=\bdcat(\fgmod{A})\), recalling for the final equality that \(\gldim{A}<\infty\) \cite[Thm.~3.7]{Pressland-Postnikov}.
Since \(F\) is left exact as a functor on \(\fgmod{B}\) and \(\Tsrc\) is rigid, the complex \(F\xi\) is exact at \(F\Tsrc\) and \(FP_1\).
Because the map \(P_0\to\Ttgt\) is a projective cover, the cohomology of \(F\xi\) at \(FP_0\) is the subspace of \(\Hom_B(\Tsrc,\Ttgt)\) consisting of maps factoring over a projective \(B\)-module; that is, it is \(F'\Ttgt\).
Thus, in \(\bdcat(\fgmod{A})\) we have \(F\xi\cong F'\Ttgt\).
Now \(F\) is fully faithful by Yoneda's lemma, and \(F'\) is fully faithful by \cite[Lem.~2.2]{CBS}, observing as in \cite[\S5]{CKP} that it is the intermediate extension functor \cite{Kuhn} associated to the idempotent \(e\).
We therefore have
\[\Endalg{B}{\xi}=\Endalg{A}{F\xi}\iso\Endalg{A}{F'\Ttgt}=\Endalg{B}{\Ttgt}=A,\]
and so it remains only to show that \(F'\Ttgt\) is tilting as an \(A\)-module.

By summing the isomorphisms from Lemma~\ref{l:A-inj} over \(j\in Q_0\), we see that \(F'\Ttgt=\Zdual{A}\). Since this object is an injective cogenerator of \(\CM(A)\), and \(A\) has finite global dimension, it is in particular tilting as required.
\end{proof}

Finally, we establish commutativity of the diagram \eqref{eq:cc-diagram} for the cluster-tilting object \(\Ttgt\).
In proving the next lemma, we denote the suspension functor in \(\bdcat(\gproj\CM(B))\) by \([1]\), to distinguish it from the suspension functor \(\Sigma\) of \(\singcat(B)\).

\begin{lem}
\label{l:CC-varphi}
Let \(X\in\CM(B)\), and consider the commutative diagram
\begin{equation}
\label{eq:gp-app}
\begin{tikzcd}
0\arrow{r}&\bOmega X\arrow{r}{\kappa}&PX\arrow{r}&X\arrow{r}&0\\
0\arrow{r}&\bOmega X\arrow[equal]{u}\arrow{r}&QX\arrow{r}\arrow{u}&\bSigma\bOmega X\arrow{r}\arrow{u}&0
\end{tikzcd}
\end{equation}
with exact rows, in which \(PX\to X\) is a projective cover and \(\bOmega X\to QX\) is a left \(\projcat(B)\)-approximation. Then the lower row of \eqref{eq:gp-app} is a short exact sequence in \(\gproj\CM(B)\), and
\begin{equation}
\label{eq:CC-varphi}\CCsrc(\varphi X)=\CCsrc(\bSigma\bOmega X)\frac{\CCsrc(PX)}{\CCsrc(QX)}.
\end{equation}
\end{lem}

\begin{proof}
We first justify that the diagram \eqref{eq:gp-app} is well-defined, and that its lower row is in \(\gproj\CM(B)\).
The syzygy \(\bOmega X\) is in \(\gproj\CM(B)\) by \cite[Lem.~10.4]{CKP}, and so the left \(\projcat(B)\)-approximation \(\bOmega X\to QX\) is an injective envelope in \(\gproj\CM(B)\), hence in particular an admissible monomorphism.
Thus, its cokernel \(\bSigma\bOmega X\) lies in \(\gproj\CM(B)\) as well.
The functor \(\bSigma\) descends to the suspension functor on \(\singcat(B)=\stabgproj{\CM{B}}\), hence the notation

Since \(PX,\bOmega X\in\gproj\CM(B)\), it follows from the upper row of \eqref{eq:gp-app} that \(\varphi X \iso (\bOmega X\map{\kappa} PX)\) in \(\bdcat(\CM(B))\).
Since this \(2\)-term complex is also the cone of the morphism \(\kappa\), viewed as a map of stalk complexes in \(\bdcat(\gproj\CM(B))\), we see that
\[\CCsrc(\varphi X)=\CCsrc(\bOmega X[1])\CCsrc(PX),\]
using Theorem~\ref{t:CC-lift}.
Applying the same theorem to the triangle in \(\bdcat(\gproj\CM(B))\) induced from the lower row of \eqref{eq:gp-app}, we also have
\[\CCsrc(\bOmega X[1])=\frac{\CCsrc(\bSigma\bOmega X)}{\CCsrc(QX)},\]
and the result follows by combining these two equations.
\end{proof}

\begin{prop}
\label{p:CCsrc}
Let \(M_I\in\ginj\CM(B)\) be a rank \(1\) module.
Then  \(\CCsrc(\varphi M_I)=\Plueck{I}\in\CC[\openposvarcone_\posit]\).
\end{prop}
\begin{proof}
Apply Lemma~\ref{l:CC-varphi} to \(M_I\), and then take the twist of the resulting identity \eqref{eq:CC-varphi} to obtain
\[\ltwist\CCsrc(\varphi M_I)=\ltwist\CCsrc(\bSigma\bOmega M_I)\frac{\ltwist\CCsrc(PM_I)}{\ltwist\CCsrc(QM_I)},\]
where the objects on the right-hand side are all defined via a diagram of the form \eqref{eq:gp-app}, and in particular all lie in \(\gproj\CM(B)\).

Now we use Theorem~\ref{t:twist} to compute these twisted cluster characters.
We have \(\ltwist\CCsrc(PM_I)=\CCsrc(PM_I)^{-1}\) and \(\ltwist\CCsrc(QM_I)=\CCsrc(QM_I)^{-1}\) since both \(PM_I\) and \(QM_I\) are projective \(B\)-modules.
In the singularity category \(\singcat(B)\), we have \(\bSigma\bOmega M_I\iso\Sigma\Sigma^{-1} M_I\iso M_I\).
By the main result of \cite{OPS-PGs} (cf.~the proof of \cite[Thm.~7.6]{Pressland-Postnikov}), \(M_I\) is reachable from \(\Ttgt\); here we use that \(M_I\in\ginj\CM(B)\).
Hence \(\bSigma\bOmega M_I\) is reachable from \(\Tsrc\) by Proposition~\ref{p:mut-shift} and Theorem~\ref{t:sigma2}.
Since \(QM_I\to\bSigma\bOmega M_I\) is a projective cover, we therefore have
\[\ltwist\CCsrc(\bSigma\bOmega M_I)=\frac{\CCsrc(\bOmega M_I)}{\CCsrc(QM_I)},\]
by Theorem~\ref{t:twist}.
It follows that
\[\ltwist\CCsrc(\varphi M_I)=\frac{\CCsrc(\bOmega M_I)}{\CCsrc(QM_I)}\frac{\CCsrc(QM_I)}{\CCsrc(PM_I)}=\frac{\CCsrc(\bOmega M_I)}{\CCsrc(PM_I)}=\ltwist\Plueck{I},\]
where the final equality is Theorem~\ref{t:MS=CC} (originally from \cite{CKP}).
Since \(\ltwist\) is an isomorphism \cite[Thm.~6.7]{MulSpe-Twist}, we have \(\CCsrc(\varphi M_I)=\Plueck{I}\) as required.
\end{proof}
 
\begin{cor}
\label{c:cc-diag}
The diagram
\[\begin{tikzcd}
\add(\Ttgt)\arrow{r}{\CCtgt}\arrow{d}{\varphi}&\CC[\openposvarcone_\posit]\arrow[equal]{d}\\
\bdcat(\gproj\CM(B))\arrow{r}{\CCsrc}&\CC[\openposvarcone_\posit]
\end{tikzcd}\]
commutes.
\end{cor}
\begin{proof}
Each summand of \(\Ttgt\) is a rank \(1\) module \(M_I\) in \(\ginj\CM(B)\), and \(\CCtgt(M_I)=\Plueck{I}\) by Theorem~\ref{t:CC-values}.
We also have \(\CCsrc(\varphi M_I)=\Plueck{I}\) by Proposition~\ref{p:CCsrc}.
\end{proof}

Note that \cite[Thm.~7.6]{Pressland-Postnikov} is the analogous statement to Proposition~\ref{p:CCsrc} for \(M_I\in\gproj\CM(B)\).
We expect that in fact \(\CCsrc(\varphi M_I)=\Plueck{I}\) for any \(M_I\in\CM(B)\) (and more generally that \(\CCsrc(\varphi X)=\CCtgt(X)\) for any \(X\in\bdcat(\fgmod{B})\)), but at present we depend on reachability arguments which restrict us to one of the stably \(2\)-Calabi--Yau Frobenius categories \(\gproj\CM(B)\) or \(\ginj\CM(B)\) in which these make sense.

We may now conclude by proving the main theorem, namely that Conjecture~\ref{conj:qc-conj} is true.

\begin{thm}
\label{t:main-thm}
For any positroid \(\posit\), the source-labelled and target-labelled cluster structures on \(\CC[\openposvarcone_\posit]\) quasi-coincide.
\end{thm}
\begin{proof}
By Corollary~\ref{c:conn}, we may assume that \(\posit\) is connected.
Thus, \(\posit=\posit_D\) for a connected Postnikov diagram \(D\), and we let \(A=A_D\), \(B=eAe\), \(\Tsrc=eA\) and \(\Ttgt=\Zdual{(Ae)}\) be the associated algebras and cluster-tilting objects.

The result now follows by applying Fraser--Keller's Theorem~\ref{t:FKW} to the categorifications \((\ginj\CM(B),\Tsrc)\) and \((\gproj\CM(B),\Ttgt)\) and the identity map \(\id\colon\CC[\openposvar_\posit]\to\CC[\openposvar_\posit]\).
The existence of the commutative diagram \eqref{eq:qcm-diagram}, and the fact that \(\underline{\varphi}\) is an equivalence, combines Proposition~\ref{p:der-eq}, Corollary~\ref{c:stab-equiv} and Proposition~\ref{p:homot-lift}.
Commutativity of the diagram \eqref{eq:cc-diagram} is Corollary~\ref{c:cc-diag}.
Finally, \(\stab{\varphi}(\Tsrc)\) is mutation equivalent to \(\Ttgt\) by the combination of Proposition~\ref{p:mut-shift} and Theorem~\ref{t:sigma2}.
\end{proof}

\begin{eg}
We continue Examples~\ref{eg:running1} and \ref{eg:running2}.
Note from Tables~\ref{tab:src-eg} and \ref{tab:tgt-eg} that while \(\Delta_{157}\) is an initial cluster variable for the target-labelled cluster structure on this open positroid variety, it is not a cluster variable for the source-labelled cluster structure, and indeed the module \(M_{157}\) does not appear in \(\gproj\CM(B)\).
If we compute a second syzygy of this module as follows
\begin{equation}
\label{eq:syz2}
\begin{tikzcd}M_{135}\arrow{r}&\begin{matrix}M_{134}\\\dsum\\M_{356}\end{matrix}\arrow{r}&\begin{matrix}M_{167}\\\dsum\\M_{345}\end{matrix}\arrow{r}&M_{157}
\end{tikzcd}
\end{equation}
we find that the resulting module \(M_{135}\) is the initial source-labelled cluster variable attached to the same quiver vertex.
This sequence tells us that \(M_{157}=\Sigma^2(M_{135})\) in the singularity category \(\singcat(B)\), as predicted by Theorem~\ref{t:sigma2}.

We may compare this calculation to that of the weight of the upstream wedge matching \(\MStgt{b}\) at the vertex \(b\) with target label \(157\), as shown in Figure~\ref{f:match}.
\begin{figure}
\begin{tikzpicture}[scale=3,baseline=(bb.base),yscale=-1]
\path (0,0) node (bb) {};

\foreach \n/\m/\a in {1/4/0, 2/3/0, 3/2/5, 4/1/10, 5/7/0, 6/6/-3, 7/5/0}
{ \coordinate (b\n) at (\bstart-\seventh*\n+\a:1.0);
  \draw (\bstart-\seventh*\n+\a:1.1) node {\(\m\)}; }

\foreach \n/\m in {8/1, 9/2, 10/3, 11/4, 14/5, 15/6, 16/7}
  {\coordinate (b\n) at ($0.65*(b\m)$);}

\coordinate (b13) at ($(b15) - (b16) + (b8)$);
\coordinate (b12) at ($(b14) - (b15) + (b13)$);

\foreach \n/\x/\y in {13/-0.03/-0.03, 12/-0.22/0.0, 14/-0.07/-0.03, 11/0.05/0.02, 16/-0.02/0.02}
  {\coordinate (b\n)  at ($(b\n) + (\x,\y)$); } 

\foreach \h/\t in {1/8, 2/9, 3/10, 4/11, 5/14, 6/15, 7/16, 
 8/9, 9/10, 10/11,11/12, 12/13, 13/8, 14/15, 15/16, 12/14, 13/15, 8/16}
{ \draw [bipedge] (b\h)--(b\t); }

\foreach \n in {8,10,12,15} 
  {\draw [\graphcolor] (b\n) circle(\dotrad) [fill=white];} \foreach \n in {9,11,13, 14,16}  
  {\draw [\graphcolor] (b\n) circle(\dotrad) [fill=\graphcolor];} 

\foreach \e/\f/\t in {2/9/0.5, 4/11/0.5, 5/14/0.5, 7/16/0.5, 
 8/9/0.5, 9/10/0.5, 10/11/0.5,11/12/0.5, 12/13/0.45, 8/13/0.6, 
 14/15/0.5, 15/16/0.6, 12/14/0.45, 13/15/0.4, 8/16/0.6}
{\coordinate (a\e-\f) at ($(b\e) ! \t ! (b\f)$); }

\foreach \n/\m/\l/\a in {1/124/4/0, 2/234/3/-1, 3/345/2/1, 4/456/1/10, 5/256/7/5, 6/267/6/0, 7/127/5/0}
{ \draw [\frozcolor] (\bstart+\seventh/2-\seventh*\n+\a:1) node (q\m) {\scriptsize \(\l\)}; }

\foreach \m/\l/\a/\r in {247/b/\bstart/0.42, 245/c/10/0.25, 257/a/210/0.36}
{ \draw [\quivcolor] (\a:\r) node (q\m) {\scriptsize \(\l\)}; }

\foreach \t/\h/\a in {124/247/13, 245/234/18, 257/247/14, 245/257/15,
 267/257/6, 256/245/-9, 345/245/0}
{ \draw [quivarrow]  (q\t) edge [bend left=\a] (q\h); }

\foreach \t/\h/\a in {234/124/-20, 456/345/-16, 456/256/22, 256/267/22, 127/267/-20,
 127/124/19}
{ \draw [frozarrow]  (q\t) edge [bend left=\a] (q\h); }

\foreach \t/\h/\a in {234/345/19, 247/127/11, 247/245/34, 257/256/0, 245/456/-16}
{ \draw [matcharrow]  (q\t) edge [bend left=\a] (q\h); }
\end{tikzpicture}
\caption{The upstream wedge matching \(\MStgt{b}\) at vertex \(b\). Matched arrows are shown in magenta.}
\label{f:match}
\end{figure}
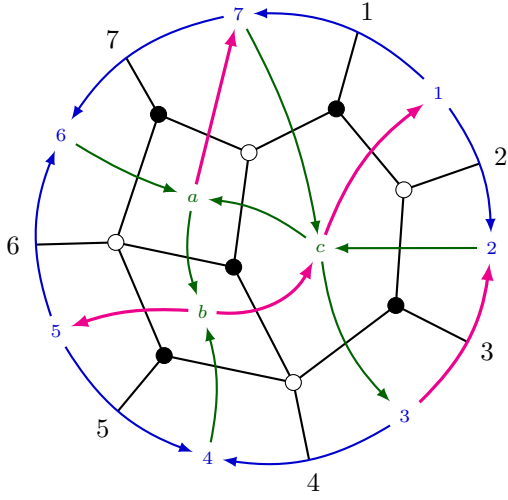
Identifying \(\Kgp_0(\projcat{A})\) with \(\ZZ^{Q_0}=\Span{v_a,v_b,v_c,v_1,\dotsc,v_7}_{\ZZ}\) via the basis of projectives, we may compute using \cite[Prop.~6.11]{CKP} that in this Grothendieck group
\begin{align*}
[N(\MStgt{b})]&=v_a+v_b+v_c\\
&\qquad +v_1+v_5+v_7\\
&\qquad-v_6-v_7-v_a-v_4-v_c\\
&=v_b+v_1-v_4+v_5-v_6,
\end{align*}
where care must be taken when computing the contribution from boundary arrows (in the second line) since we have not standardised the plabic graph in Figure~\ref{f:match}.
This class projects to \(v_b\) in \(\Kgp_0(\projcat{\stab{A}})\), as predicted by Lemma~\ref{l:upmat-class}.
We also observe that the source labels of vertices \(1\) and \(5\) are \(167\) and \(345\), while those of \(4\) and \(6\) are \(134\) and \(356\), so the identity corresponds to the exact sequence \eqref{eq:syz2}.

Now we compute a representative of \(\Sigma(M_{135})\in\singcat(B)\) from \(\gproj\CM(B)\) using left approximations by projectives; this yields the exact sequence
\begin{equation}
\label{eq:syz-2}
\begin{tikzcd}M_{135}\arrow{r}&\begin{matrix}M_{134}\\\dsum\\M_{356}\end{matrix}\arrow{r}&\begin{matrix}M_{367}\\\dsum\\M_{345}\end{matrix}\arrow{r}&M_{357}.
\end{tikzcd}
\end{equation}

Since the two cluster structures quasi-coincide, the target-labelled cluster variable \(\Plueck{157}\) must be expressible as the product of a cluster variable and a Laurent monomial in frozen variables from the source-labelled structure.
We have \(\Plueck{157}=\CCtgt(M_{157})\) by Theorem~\ref{t:CC-values}, and \(\CCtgt(M_{157})=\CCsrc(\varphi M_{157})\) by Lemma~\ref{l:CC-varphi}.
Using \eqref{eq:syz2} and \eqref{eq:syz-2} together with Theorem~\ref{t:CC-lift}, we may compute that
\begin{equation}
\label{eq:qcoin}
\Plueck{157}=\CCsrc(\varphi M_I)=\Plueck{357}\frac{\Plueck{167}\Plueck{345}}{\Plueck{367}\Plueck{345}}=\Plueck{357}\frac{\Plueck{167}}{\Plueck{367}},
\end{equation}
giving the required expression for \(\Plueck{157}\) in terms of source-labelled cluster variables.
This identity can be verified using the short Plücker relation
\[\Plueck{357}\Plueck{167}-\Plueck{157}\Plueck{367}=\Plueck{137}\Plueck{567}=0\]
on \(\openposvar_{\posit}\), where \(\Plueck{567}=0\).

The leftmost maps in \eqref{eq:syz2} and \eqref{eq:syz-2} are the same because any syzygy \(\bOmega(M_{157})\) is already Gorenstein projective \cite[Lem.~10.4]{CKP}.
This is why it suffices to compute a single syzygy and a single left \(\projcat(B)\)-approximation, as in Lemma~\ref{l:CC-varphi}, to obtain a quasi-coincidence identity of the form \eqref{eq:qcoin}. \end{eg}

\section{The twist}
\label{s:twist}

Via much the same approach as used for Theorem~\ref{t:main-thm}, we may show that the left twist automorphism \(\ltwist\colon\CC[\openposvarcone_\posit]\to\CC[\openposvarcone_\posit]\) from \cite{MulSpe-Twist} is a quasi-cluster morphism from the target-labelled cluster structure \(\GLsrc\) to the source-labelled cluster structure \(\GLtgt\).
Consider the diagram \eqref{eq:qcm-diagram} constructed in Section~\ref{s:proof} for the case that \(\varphi\colon\bdcat(\ginj\CM(B))\isoto\bdcat(\gproj\CM(B))\) is the tautological equivalence induced from identifying each category with \(\bdcat(\fgmod{B})\).
Since every arrow in this diagram is a triangle functor, we may postcompose the vertical arrows with the inverse suspension functor on their codomains to obtain a new commutative diagram
\begin{equation}
\label{eq:twist-funct}
\begin{tikzcd}
\hcat[\bdd](\add{\Ttgt})\arrow{r}\arrow{d}{[-1]\circ\widetilde{\varphi}}&\bdcat(\ginj\CM(B))\arrow{r}\arrow{d}{[-1]\circ\varphi}&\stabginj\CM(B)\arrow{d}{\Sigma^{-1}\circ\stab{\varphi}}\\
\hcat[\bdd](\add{\Tsrc})\arrow{r}&\bdcat(\gproj\CM(B))\arrow{r}&\stabgproj\CM(B).
\end{tikzcd}
\end{equation}
We observe that \([-1]\circ\varphi\colon\bdcat(\ginj\CM(B))\to\bdcat(\gproj\CM(B))\) is induced from the inverse suspension functor on \(\bdcat(\fgmod{B})\), just as \(\varphi\) was induced from the identity functor.
Every vertical arrow is a triangle equivalence, since this is true for \(\varphi\), \(\stab{\varphi}\) and \(\widetilde{\varphi}\) as shown in Section~\ref{s:proof}.

In the next two statements, we will use the following fact.
Let \(X\in\CM(B)\), and let \(PX\to X\) be a projective cover with kernel \(\bOmega X\).
Then the resulting triangle
\[\begin{tikzcd}
PX[-1]\arrow{r}&X[-1]\arrow{r}&\bOmega X\arrow{r}&PX
\end{tikzcd}\]
in \(\bdcat(\CM(B))\) remains exact under the equivalence \(\varphi\), and thus we see that
\begin{equation}
\label{eq:cc-shift}
\CCsrc(\varphi X[-1])=\frac{\CCsrc(\bOmega X)}{\CCsrc(PX)},
\end{equation}
using as in Section~\ref{s:proof} that both \(PX\) and \(\bOmega X\) are Gorenstein projective.
\begin{prop}
\label{p:twist-base}
The diagram
\begin{equation}
\label{eq:twist-cc}
\begin{tikzcd}\add(\Ttgt)\arrow{r}{\CCtgt}\arrow{d}{[-1]\circ\varphi}&\CC[\openposvarcone_\posit]\arrow{d}{\ltwist}\\
\bdcat(\gproj\CM(B))\arrow{r}{\CCsrc}&\CC[\openposvarcone_{\posit}]
\end{tikzcd}
\end{equation}
commutes.
\end{prop}
\begin{proof}
For any \(M_I\in\ginj\CM(B)\), we have
\[\ltwist\CCtgt(M_I)=\ltwist(\Plueck{I})=\frac{\CCsrc(\bOmega M_I)}{\CCsrc(PM_I)}=\CCsrc(\varphi M_I[-1])\]
for any projective cover \(PM_I\to M_I\) with kernel \(\bOmega M_I\); the first equality here is Theorem~\ref{t:CC-values} (originally from \cite{Pressland-Postnikov}), the second is Theorem~\ref{t:MS=CC} (originally from \cite{CKP}), and the third is obtained by applying \eqref{eq:cc-shift} to \(M_I\).
\end{proof}

\begin{thm}
\label{t:twist-qcm}
The left twist automorphism \(\ltwist\colon\CC[\openposvarcone_\posit]\to\CC[\openposvarcone_\posit]\) is a quasi-cluster morphism from \(\GLtgt\) to \(\GLsrc\).
\end{thm}
\begin{proof}
The statement reduces to the connected case by induction on Theorem~\ref{t:twist-conn}.
For \(\Tsrc\) and \(\Ttgt\) the cluster-tilting objects associated to a connected Postnikov diagram, we have \(\Sigma^{-1}\stab{\varphi}(\Ttgt)=\Sigma^{-1}\Ttgt=\Sigma\Tsrc\in\singcat(B)\) by Theorem~\ref{t:sigma2}, and thus \(\Sigma^{-1}\stab{\varphi}(\Ttgt)\) is mutation equivalent to \(\Tsrc\) by Proposition~\ref{p:mut-shift}.
The result then follows from Theorem~\ref{t:FKW}, using the commutative diagrams \eqref{eq:twist-funct} and \eqref{eq:twist-cc}.
\end{proof}

As remarked in the introduction, it has been proved independently by Casals, Le, Sherman-Bennett and Weng \cite{CLSBW} that the right twist \(\rtwist=\ltwist^{-1}\) is a quasi-cluster morphism.
We close by commenting briefly on the logic of the two arguments, and showing furthermore that Theorem~\ref{t:twist-qcm} implies Theorem~\ref{t:main-thm}.

\begin{prop}
\label{p:qcms}
One can show, without using that \(\GLsrc\) and \(\GLtgt\) quasi-coincide, that the following statements are equivalent:
\begin{enumerate}
\item\label{src-to-tgt} the left twist \(\ltwist\) is a quasi-cluster morphism from \(\GLtgt\) to \(\GLsrc\);
\item\label{tgt-to-tgt} the left twist \(\ltwist\) is a quasi-cluster morphism from \(\GLtgt\) to \(\GLtgt\);
\item\label{src-to-src} the left twist \(\ltwist\) is a quasi-cluster morphism from \(\GLsrc\) to \(\GLsrc\).
\end{enumerate}
Moreover, any of these statements implies that \(\GLsrc\) and \(\GLtgt\) quasi-coincide.
\end{prop}
\begin{proof}
By \cite[Prop.~7.13]{MulSpe-Twist} and a calculation in the proof of \cite[Thm.~5.17]{FSB}, showing that Definition~\ref{d:qcm}\ref{d:qcm-yhat} is satisfied, \(\rtwist^2\) is a quasi-cluster morphism from \(\GLsrc\) to \(\GLtgt\).
Now, since \(\ltwist=\rtwist^{-1}\) \cite[Thm.~6.7]{MulSpe-Twist}, we have identities
\[\rtwist^2\circ\ltwist=\ltwist^{-1}=\ltwist\circ\rtwist^2,\]
from which we deduce the equivalence of \ref{src-to-tgt}--\ref{src-to-src}.
Similarly, the identity \(\id=\rtwist^2\circ(\ltwist)^2\), in which we view \(\rtwist^2\) as a quasi-cluster morphism from \(\GLtgt\) to \(\GLsrc\) and \(\ltwist\) as a quasi-cluster automorphism of \(\GLtgt\), shows that \ref{tgt-to-tgt}, and hence any of \ref{src-to-tgt}--\ref{src-to-src}, implies the quasi-coincidence of \(\GLsrc\) and \(\GLtgt\).
\end{proof}

\begin{rem}
The arguments in \cite{CLSBW} most directly prove Proposition~\ref{p:qcms}\ref{tgt-to-tgt} (or rather, the equivalent statement for \(\rtwist=\ltwist^{-1}\)), and deduce the quasi-coincidence from this.
The method is to show that \(\rtwist\) coincides with the Donaldson--Thomas transformation, known to be a quasi-cluster morphism for any cluster algebra admitting a green-to-red sequence, as \(\clustalg{D}\) does by \cite{ForSer}.
We implicitly do this in our proof of Theorem~\ref{t:twist-qcm}, since the Donaldson--Thomas transformation is the quasi-cluster morphism induced by the shift automorphism \([1]\) (see \cite[Ex.~A.10]{KelWu}).

We have opted instead to prove Proposition~\ref{p:qcms}\ref{src-to-tgt}, but this choice was largely arbitrary: categorically, the equivalence of the three statements in Proposition~\ref{p:qcms} corresponds to the fact that the derived categories of \(\gproj\CM(B)\) and \(\ginj\CM(B)\) are tautologically equivalent.
Our choice was primarily motivated by the fact that \(\Sigma^{-1}\) takes Gorenstein injective objects to Gorenstein projective objects in the singularity category, by \cite[Lem.~10.4]{CKP}, making it natural to view the induced quasi-cluster morphism as going from \(\GLtgt\) to \(\GLsrc\).
As we pointed out in Remark~\ref{r:choices}, the argument in Theorem~\ref{t:twist-conn} to reduce to connected positroids is also insensitive to the choice of which pair of cluster structures we are aiming to relate.
\end{rem}

\section*{Acknowledgements}
Many thanks are due to Bernhard Keller for sharing in advance his results with Chris Fraser \cite[Appendix]{KelWu}, on which this work crucially depends, and for several useful discussions both about these results and the problem in general, most notably at ICRA 2022 in Montevideo and Buenos Aires, the conference \emph{Representation Theory and Triangulated Categories} in Paderborn in September 2022, and at ARTA IX in Kingston, ON, in 2023.
I also thank Darius Dramburg, Franco Rota and David Speyer for useful conversations about some of the geometric arguments needed in Section~\ref{s:connected}, Melissa Sherman-Bennett for a helpful discussion concerning the results of \cite{CLSBW}, and the anonymous referee for suggesting corrections and improvements.
This work has benefited greatly from two different long-term collaborations: one with İlke Çanakçı and Alastair King, to which the connection will be obvious, and a second with Jan E.\ Grabowski, from which much of my understanding of indices and coindices derives.
The author was supported by the EPSRC postdoctoral fellowship EP/T001771/2 while this work was conducted.

\defbibheading{bibliography}[\refname]{\section*{#1}}
\sloppy\printbibliography
\end{document}